\documentclass[12pt]{article}
\usepackage{amsmath}
\usepackage{amssymb}
\usepackage{mathrsfs}
\usepackage{amsthm}
\usepackage{authblk}
\usepackage{color}
\usepackage{graphicx}
\usepackage{caption}
\usepackage{subcaption}
\usepackage{makecell}
\usepackage{soul}
\usepackage{bbm} 
\usepackage{ulem}
\usepackage{cancel} 
\usepackage{algorithm}
\usepackage{algorithmicx}
\usepackage{amssymb}
\usepackage{threeparttable}
\usepackage{titlesec}
\usepackage{tikz}
\usepackage{bbding}
\usepackage{pifont}
\usepackage{titletoc}
\usepackage{hhline}
\usepackage{bbm}
\usepackage{adjustbox}

\theoremstyle{definition}
\newtheorem{definition}{Definition}
\newtheorem{theorem}{Theorem}
\newtheorem{lemma}{Lemma}
\newtheorem{assumption}{Assumption}
\newtheorem{corollary}{Corollary} 
\newtheorem{remark}{Remark}
\newtheorem{proposition}{Proposition}
\newtheorem{fact}{Fact}
\usepackage{hyperref}
\usepackage[margin=1.2in]{geometry}

\usepackage{hyperref}

\newcommand{\h}{\pmb{h}}

\newcommand{\x}{\pmb{x}}
\newcommand{\y}{\pmb{y}}
\newcommand{\z}{\pmb{z}}
\newcommand{\M}{\pmb{M}}
\newcommand{\X}{\pmb{X}}

\newcommand{\Z}{\pmb{Z}}
\newcommand{\xu}{\pmb{u}}

\newcommand{\pp}{\pmb{\varphi}}
\newcommand{\ppp}{\pmb{\Phi}}

\newcommand{\E}{\mathbb{E}}

\newcommand{\R}{\mathbb{R}}
\newcommand{\C}{\mathbb{C}}
\newcommand{\mD}{\mathcal{D}}

\newcommand{\mS}{\mathcal{S}}
\newcommand{\mA}{\mathcal{A}}
\newcommand{\mT}{\mathcal{T}}

\newcommand{\mO}{\mathcal{O}}
\newcommand{\mM}{\mathcal{M}}
\newcommand{\mN}{\mathcal{N}}

\newcommand{\mF}{\mathcal{F}}

\def\lV{\left\lVert}
\def\rV{\right\lVert}
\def\lv{\left\lvert}
\def\rv{\right\lvert}
\def\lk{\left(}
\def\rk{\right)}
\def\lg{\langle}
\def\rg{\rangle}
\def\lz{\left[}
\def\rz{\right]}

\begin{document}
\title{Stable Phase Retrieval: Optimal Rates in Poisson and Heavy-tailed Models}
\author{Gao Huang\footnote{ School of Mathematical Sciences, Zhejiang University, Hangzhou 310027, P. R. China. Email: hgmath@zju.edu.cn}}	
\author{Song Li\footnote{ School of Mathematical Sciences, Zhejiang University, Hangzhou 310027, P. R. China. Corresponding author. Email: songli@zju.edu.cn}}	
\author{Deanna Needell\footnote{Department of Mathematics, University of California,
Los Angeles, CA 90095, USA. Email: deanna@math.ucla.edu}}

\affil{}
\renewcommand*{\Affilfont}{\small\it}
	
	\maketitle
	\begin{abstract}	
We investigate stable recovery guarantees for phase retrieval under two realistic and challenging noise models: the Poisson model and the heavy-tailed model. 
Our analysis covers both nonconvex least squares (NCVX-LS) and convex least squares (CVX-LS) estimators.
For the Poisson model, we demonstrate that in the high-energy regime where the true signal $\x$ exceeds a certain energy threshold,  both estimators achieve a signal-independent, minimax optimal error rate $\mO\lk\sqrt{\frac{n}{m}}\rk$, with $n$ denoting the signal dimension and $m$ the number of sampling vectors.
To the best of our knowledge, these are the first minimax optimal recovery guarantees established for the Poisson model.
In contrast, in the low-energy regime, the NCVX-LS estimator attains an error rate of {\footnotesize$\mO\lk\lV\x\rV^{1/4}_2\cdot\lk\frac{n}{m}\rk^{1/4}\rk$}, which decreases as the energy of signal $\x$ diminishes 
and remains nearly optimal with respect to the oversampling ratio.
This demonstrates a signal-energy-adaptive behavior in the Poisson setting. 
For the heavy-tailed model with noise having a finite $q$-th moment ($q>2$), both estimators attain the minimax optimal error rate {\footnotesize$\mO\lk \frac{\lV \xi \rV_{L_q}}{\lV \x \rV_2} \cdot \sqrt{\frac{n}{m}} \rk$} in the high-energy regime, while the NCVX-LS estimator further achieves the minimax optimal rate {\footnotesize$\mO\lk \sqrt{\lV \xi \rV_{L_q}}\cdot \lk\frac{n}{m}\rk^{1/4} \rk$} in the low-energy regime.

Our analysis builds on two key ideas: the use of multiplier inequalities to handle noise that may exhibit dependence on the sampling vectors, and a novel interpretation of Poisson noise as sub-exponential in the high-energy regime yet heavy-tailed in the low-energy regime.
These insights form the foundation of a unified analytical framework, which we further apply to a range of related problems, including sparse phase retrieval, low-rank positive semidefinite matrix recovery, and random blind deconvolution, 
demonstrating the versatility and broad applicability of our approach.	 
      	\end{abstract}	
        
        
		
\section{Introduction}	

Consider a set of $ m$ quadratic equations taking the form
\begin{equation}\label{phaseless}
		y_{k}=\lv\lg\pp_{k},\x \rg\rv^{2},\quad k=1,\cdots,m,
		\end{equation}
where the observations $\left\{y_{k}\right\}_{k=1}^m$ and the design vectors $\{ \pp_{k} \}_{k=1}^{m}$ in $V=\C^{n}$ are known and the goal is to reconstruct the unknown vector $\x\in\C^{n}$.
This problem, known as phase retrieval \cite{fienup1982phase}, arises in a broad range of applications, including X-ray crystallography, diffraction imaging, microscopy, astronomy, optics, and quantum mechanics; see, e.g., \cite{candes2015phase}.

From an application standpoint, the stability of the reconstruction performance is arguably the most critical consideration. 
That is, we focus on scenarios where the observed data may be corrupted by noise, which means that we only have access to noisy measurement of $\lv\lg\pp_{k},\x \rg\rv^{2}$.
There are various sources of noise contamination, including thermal noise, background noise, and instrument noise, among others; see, e.g., \cite{chang2019advanced}.
A common type of noise arises from the operating mode of the detector \cite{chen2017solving,fannjiang2020numerics,diederichs2024wirtinger}, particularly in imaging applications such as CCD cameras, fluorescence microscopy and optical coherence tomography (OCT),  where variations in the number of photons are detected. 
As a result, the measurement process can be modeled as a counting process, which is mathematically represented by Poisson observation model,
\begin{equation}\label{poisson}
		y_{k}\overset{\text{ind.}}{\sim}\text{Poisson}\lk\lv\lg\pp_{k},\x \rg\rv^{2}\rk,\quad k=1,\cdots,m.
		\end{equation}
This means that the observation data $y_k$ at each pixel position (or measurement point $k$) follows the Poisson distribution with parameter $\lv\lg\pp_{k},\x \rg\rv^{2}$.
Poisson noise is an adversarial type of noise that depends not only on the design vectors but also on the true signal, with its intensity diminishing as the signal energy decreases, thereby complicating the analysis; see, e.g., \cite{diederichs2024wirtinger,dirksen2025spectral,allain2025phasebook}.
Another common source of noise is the nonideality of optical and imaging systems, as well as the generation of super-Poisson noise by certain sensors; see, e.g., \cite{yeh2015experimental}. 
This type of noise typically exhibits a heavy-tailed distribution, meaning that the probability density is higher in regions with larger values (far from the mean). 
We model the observations $\left\{y_k\right\}_{k=1}^{m}$ using a heavy-tailed observation model, 
\begin{equation}\label{heavy}
		y_{k}=\lv\lg\pp_{k},\x \rg\rv^{2}+\xi_k,\quad k=1,\cdots,m,
		\end{equation}
where $\left\{\xi_k\right\}_{k=1}^{m}$ represent heavy-tailed noise that satisfies certain statistical properties.
Heavy-tailed noise contains more outliers, which contradicts the sub-Gaussian or sub-exponential noise assumptions commonly used in the theoretical analysis of standard statistical procedures \cite{ibragimov2015heavy}.
Therefore, addressing heavy-tailed model and characterizing its stable performance in phase retrieval remains a challenge; see, e.g., \cite{chen2022error,buna2024robust}.

Now, a natural and important question arises:
\begin{itemize}
\item[] \textbf{\textit{Where does the phase retrieval problem stand in terms of minimax optimal statistical performance when the observations follow Poisson distributions \eqref{poisson} or are contaminated by heavy-tailed noise \eqref{heavy}?}}

\end{itemize}

\noindent Unfortunately, to our best knowledge, the existing theoretical understanding for phase retrieval under Poisson model \eqref{poisson} and heavy-tailed model \eqref{heavy} remains far from satisfactory, as we shall discuss momentarily.

\subsection{Prior Art and Bottlenecks}\label{prior}

\subsubsection{Poisson Model}

We begin by reviewing results from the literature on the Poisson model \eqref{poisson}; a summary is provided in Table~\ref{tab:1}.
In a breakthrough work~\cite{candes2013phaselift}, Cand{\'e}s, Strohmer, and Voroninski established theoretical guarantees for phase retrieval using the PhaseLift approach and demonstrated its stability in the presence of bounded noise. 
Moreover, their experiments showed that the PhaseLift approach performs robustly under Poisson noise, with stability comparable to the case of Gaussian noise. 
However, they did not provide a theoretical justification for this observation. 
Furthermore, in the discussion section of~\cite{candes2013phaselift}, they suggested that assuming random noise, such as Poisson noise, could lead to sharper error bounds compared to the case of bounded noise.
	
To handle the Poisson model~\eqref{poisson}, Chen and Cand{\'e}s in~\cite{chen2017solving} proposed a Poisson log-likelihood estimator and introduced a novel approach called truncated Wirtinger flow to solve it, which improves upon the original Wirtinger flow method introduced in~\cite{candes2015phase4}.
Under the assumption of Gaussian sampling and in the real case, they proved the algorithm's convergence at the optimal sampling order $m=\mathcal{O}\lk n\rk$ and established its robustness against bounded noise. 
Furthermore, leveraging the error bound derived for bounded noise, they obtained an $\mathcal{O}\lk1\rk$ error bound under Poisson noise, provided that the true signal lies in the high-energy regime, i.e., $\lV\x\rV^2_2\ge\log^{3} m$.
Moreover, under a fixed oversampling ratio, they presented a minimax lower bound for the Poisson setting, demonstrating that if also the signal energy exceeds $\log^{3} m$, 
then no estimator can achieve a mean estimation error better than $\Omega\lk\sqrt{\frac{n}{m}}\rk$; see Theorem 1.6 in~\cite{chen2017solving}.
Since the Poisson model~\eqref{poisson} characterizes the numbers of photons diffracted by the specimen (input $\x$) and detected by the optical sensor (output $\y$), reliable detection requires that the specimen be sufficiently illuminated. Motivated by this physical constraint, Chen and Candès~\cite{chen2017solving} concentrated on the high-energy regime, where photon counts are large enough to yield stable estimation under Poisson noise.
Nevertheless, despite assuming that the signal lies in the high-energy regime, their analysis still leaves a gap between the derived upper bound $\mathcal{O}\lk1\rk$ and the minimax lower bound $\Omega\lk\sqrt{\frac{n}{m}}\rk$.
 	
In a very recent work~\cite{dirksen2025spectral}, Dirksen et al. proposed a constrained optimization problem based on the spectral method to assess the stable performance of phase retrieval under Poisson noise.
In their estimator, the optimization is constrained to maintain the same energy level as the true signal $\x$,
thereby requiring prior knowledge of $\x$.
Still under the assumption of Gaussian sampling, in the real case and at the sampling order $m=\mathcal{O}\lk n\log n\rk$, they provided an error bound  
	\begin{equation}\label{prior of low-dose}
	    \textbf{dist}\lk\z_{\star},\x\rk\lesssim \lk1+\lV\x\rV_2\rk\cdot\lk\log m\rk^{1/2}\lk\log n\rk^{1/4}\lk\frac{n}{m}\rk^{1/4}.
        \end{equation}
Here, $\z_{\star}$ is the solution of the estimator and the distance $\textbf{dist}\lk\z_{\star},\x\rk$ is defined in Section \ref{setup}.
This error rate is valid without imposing restriction on the energy of the truth signal $\x$.
In this way, they extended the results of \cite{chen2017solving} to the low-energy regime.
The focus on the low-energy regime is motivated by biological applications, where only a low illumination dose can be applied to avoid damaging sensitive specimens such as viruses \cite{glaeser1971limitations}. 
In ptychography, this challenge is further amplified since the same object is measured repeatedly, resulting in extremely low photon counts, poor signal-to-noise ratios, and limited reconstruction quality with existing methods.
Although the error bound \eqref{prior of low-dose} in~\cite{dirksen2025spectral} extends to the low-energy regime, 
it still falls short of attaining the minimax lower bound established in~\cite{chen2017solving}, even in the high-energy regime.
Moreover, the error bound \eqref{prior of low-dose} does not vanish as the signal energy decreases; instead, it remains bounded by $\widetilde{\mathcal{O}}\lk\left(\frac{n}{m}\right)^{1/4}\rk$\footnote{The notation $\widetilde{\mathcal{O}}$ denotes an asymptotic upper bound that holds up to logarithmic factors.} in the low-energy regime, which contradicts the fundamental property of Poisson noise—its intensity diminishes as the signal energy decreases.

    To summarize, the Poisson model \eqref{poisson} currently faces some major bottlenecks: 
    Current theoretical analyses have not yet achieved the known minimax lower bound $\Omega\lk\sqrt{\frac{n}{m}}\rk$ in the high-energy regime.
    Moreover, in the low-energy regime, the error estimate of existing method does not decay in line with the decreasing energy of true signal, and a corresponding minimax theory for this regime is lacking.

\begin{table}[ht!]
\caption{Phase Retrieval under Poisson Model} 
\centering
\renewcommand\arraystretch{1.4}
\begin{threeparttable}
\resizebox{\textwidth}{!}{
\begin{tabular}{ccc}
    \hline\hline
    Reference & Estimator  & Error Bound \\
    \hline
    Chen and Cand{\`e}s \cite{chen2017solving} & Poisson log-likelihood &  $\mO(1)$\tnote{1} \\
    \hline
    Dirksen et al. \cite{dirksen2025spectral} & Spectral method &  $\widetilde{\mO}\left( \lk1+\lV\x\rV_2\rk\cdot\left(\frac{n}{m}\right)^{1/4} \right)$ \\
    \hline
    \textbf{Our paper} & NCVX-LS &  
    \begin{tabular}[c]{@{}c@{}}
    $\mO\lk\sqrt{\frac{n}{m}}\rk$ (high-energy)\\
    $\mO\lk\lV\x\rV_2^{1/4}\cdot\lk\frac{n}{m}\rk^{1/4}\rk$ (low-energy)
    \end{tabular} \\
    & CVX-LS  & 
    \begin{tabular}[c]{@{}c@{}}
    $\mathcal{O}\left(\sqrt{\frac{n}{m}}\right)$ (high-energy)\\
    $\mathcal{O}\left(\sqrt{\frac{1}{\|\x\|_2}}\cdot\sqrt{\frac{n}{m}}\right)$ (low-energy)
    \end{tabular} \\
    \hline\hline
\end{tabular}
}
\begin{tablenotes}
\footnotesize
\item[1] The guarantee in \cite{chen2017solving} does not apply to the low-energy regime.
\item[2] The error bounds in the above results are all evaluated using the distance $\textbf{dist}\lk\z_{\star}, \x\rk$.
\end{tablenotes}
\end{threeparttable}
\label{tab:1}
\end{table}

\subsubsection{Heavy-tailed Model}

We proceed to review results on additive random noise models, with particular attention to heavy-tailed model \eqref{heavy}; see Table~\ref{tab:2} for a summary.
Eldar and Mendelson \cite{eldar2014phase} aimed to understand the stability of phase retrieval under symmetric mean-zero sub-Gaussian noise (with sub-Gaussian norm\footnote{For $\alpha \ge 1$, the $\psi_\alpha$-norm of a random variable $X$ is
$\lV X\rV_{\psi_\alpha} := \inf\{ t >0: \mathbb{E} \exp(\lv X\rv^\alpha/t^\alpha) \leq 2\}.$
Specifically, $\alpha=2$ yields the sub-Gaussian norm, and $\alpha=1$ yields the sub-exponential norm.
Equivalent definitions of these two norms can be found in \cite[Section 2]{vershynin2018high}.} bounded by $\sqrt{n}$) and established an error bound {\small$\mathcal{O}\lk \lV\xi\rV_{\psi_2}\cdot\sqrt{\frac{n\log^2 n}{m}}\rk$} in a squared-error sense for empirical $\ell_q$ risk minimization, where the parameter $q$ should be chosen close to $1$ and specified by other parameters.
Cai and Zhang \cite{cai2015rop}, building on the PhaseLift framework of \cite{candes2013phaselift}, proposed a constrained convex optimization problem and established that at the sampling rate $m=\mathcal{O}\lk n\log n\rk$, the estimation error measured by $\lV \Z_{\star}-\x\x^*\rV_F$ (where $\Z_{\star}$ denotes the estimator’s solution) is bounded by {\small$\mathcal{O}\lk\lV\xi\rV_{\psi_2}\cdot\min\left\{\frac{n\log m}{m}+\sqrt{\frac{n}{m}},1\right\}\rk$} for i.i.d. mean-zero sub-Gaussian noise.
Lecu{\'e} and Mendelson \cite{lecue2015minimax} investigated least squares estimator (i.e., empirical $\ell_2$ risk minimization) and obtained an error bound {\small$\mathcal{O}\lk\frac{\lV\xi\rV_{\psi_2}}{\lV\x\rV_2}\cdot\sqrt{\frac{n\log m}{m}}\rk$}
with respect to $\textbf{dist}\lk\z_{\star},\x\rk$ under i.i.d. mean-zero sub-Gaussian noise. 
In addition, they further pointed out that in the case of i.i.d. Gaussian noise $\mathcal{N}\lk0,\sigma^2\rk$, no estimator can achieve a mean squared error better than {\small$\Omega\lk \min\left\{\frac{\sigma}{\lV\x\rV_2}\sqrt{\frac{n}{m}},\lV\x\rV_2\right\}\rk$}.
Cai et al. \cite{cai2016optimal} and Wu and Rebeschini \cite{wu2023nearly}
implemented the minimax error estimation for the sparse phase retrieval algorithm in the presence of independent centered sub-exponential noise.
In the non-sparse setting, their results yield the error bound {\small$\mathcal{O}\lk\frac{\lV\xi\rV_{\psi_1}}{\lV\x\rV_2}\cdot\sqrt{\frac{n\log n}{m}}\rk$}, which matches the minimax lower bound of \cite{lecue2015minimax} when $\lV\x\rV_2$ is sufficiently large, up to a logarithmic factor.
	
In a recent work \cite{chen2022error}, Chen and K.Ng also considered the same least squares estimator as \cite{lecue2015minimax}. 
They first established an improved upper bound applicable to bounded noise, and from this, derived an error bound
 {\small$\mathcal{O}\lk\frac{\lV\xi\rV_{\psi_1}}{\lV\x\rV_2}\cdot\sqrt{\frac{n\lk\log m\rk^2}{m}}\rk$} for i.i.d. mean-zero sub-exponential noise.
Therefore, this result is nearly comparable to those established in \cite{cai2016optimal,wu2023nearly}.
Moreover, they extended their analysis to i.i.d. symmetric heavy-tailed noise using a truncation technique. 
Assuming the noise has a finite moment of order $q >1$ (a necessary condition for their bound to converge), 
they obtained an error bound 
\begin{equation}\label{prior of heavy}
  \textbf{dist}\lk\z_{\star},\x\rk\lesssim \frac{\lV\xi\rV_{L_q}}{\lV\x\rV_2}\cdot\lk\sqrt{\frac{n}{m}}\rk^{1-\frac{1}{q}} \lk\log m\rk^2.  
\end{equation}
However, their result significantly deviated from the minimax lower bound $\Omega\lk\frac{\sigma}{\lV\x\rV_2}\sqrt{\frac{n}{m}}\rk$ for Gaussian noise \cite{lecue2015minimax} when $\lV\x\rV_2$ is sufficiently large.
Moreover, their analysis is limited in that it provides guarantees only for a specific signal $\x$, rather than uniformly over all $\x\in\C^n$.

In light of these bottlenecks, Chen and K.Ng in \cite{chen2022error} explicitly posed an \textbf{\textit{open problem}}: whether faster convergence rates or uniform recovery guarantees could be achieved under heavy-tailed noise (see the “Concluding Remarks” section of~\cite{chen2022error}).
Furthermore, as in the Poisson model \eqref{poisson}, the corresponding minimax theory for the low-energy regime remains undeveloped, with existing analyses primarily focusing on the high-energy regime where $\lV\x\rV_2$ is sufficiently large.

\begin{table}[ht!]
\caption{Phase Retrieval under Heavy-tailed Model} 
\centering
\renewcommand\arraystretch{1.4}
\begin{threeparttable}
\scriptsize
\begin{tabular}{ccc}
    \hline\hline
    Reference & Noise Type  & Error Bound \\
    \hline
    Eldar and Mendelson \cite{eldar2014phase} & symmetric sub-Gaussian  & $\mathcal{O}\lk \lV\xi\rV_{\psi_2}\cdot\sqrt{\frac{n\log^2 n}{m}}\rk$ \\
    \hline
    Cai and Zhang \cite{cai2015rop} & sub-Gaussian & $\mO\lk\lV\xi\rV_{\psi_2}\cdot\min\left\{ \frac{n\log m}{m} + \sqrt{\frac{n}{m}}, 1 \right\} \rk$ \\
    \hline
    Lecué and Mendelson \cite{lecue2015minimax} & sub-Gaussian  & $\mO \lk \frac{\lV\xi\rV_{\psi_2}}{\lV\x\rV_2}\cdot\sqrt{\frac{n\log m}{m}} \rk$ \\
    \hline
    Cai et al. \cite{cai2016optimal};Wu and Rebeschini \cite{wu2023nearly}  & sub-exponential  & $\mO \lk\frac{\lV\xi\rV_{\psi_1}}{\lV\x\rV_2}\cdot\sqrt{\frac{n\log n}{m}} \rk$ \\
    \hline
    Chen and K. Ng \cite{chen2022error} & symmetric heavy-tailed ($q > 1$)  & $\mO \lk \frac{\lV\xi\rV_{L_q}}{\lV\x\rV_2}\cdot\lk\sqrt{\frac{n}{m}} \rk^{1 - \frac{1}{q}} (\log m)^2 \rk$ \\
    \hline
    \textbf{Our paper (NCVX-LS)} & heavy-tailed ($q > 2$)  & $\mO\lk\min\left\{\frac{\lV\xi\rV_{L_q}}{\lV\x\rV_2}\cdot\sqrt{\frac{n}{m}},\,\sqrt{\lV\xi\rV_{L_q}}\cdot\lk\frac{n}{m}\rk^{1/4}\right\}\rk$ \\
    \cline{1-3}
    \textbf{Our paper (CVX-LS)} & heavy-tailed ($q > 2$)  & $\mO\lk \frac{\lV\xi\rV_{L_q}}{\lV\x\rV_2}\cdot\sqrt{\frac{n}{m}} \rk$ \\
    \hline\hline
\end{tabular}
\begin{tablenotes}
\footnotesize
\item[1] The error bounds in \cite{eldar2014phase,cai2015rop} are measured in a squared-error sense or the Frobenius norm,
whereas\\ the other works use the distance $\textbf{dist}\lk\z_{\star}, \x\rk$ to quantify recovery accuracy.
\item[2] The result in \cite{chen2022error} does not establish uniform recovery guarantees valid for all signals.
\end{tablenotes}
\end{threeparttable}
\label{tab:2}
\end{table}

\subsubsection{Stable Phase Retrieval}

Numerous works on phase retrieval have investigated its stability properties \cite{bandeira2014saving,balan2015invertibility,cahill2016phase,alaifari2019stable,freeman2024stable,cahill2025group,FREEMAN2026101801} or stable recovery guarantees under bounded noise \cite{candes2013phaselift,candes2014solving,iwen2017robust,chen2017solving,kueng2017low,kabanava2016stable,zhang2017nonconvex,wang2018phase,pfander2019robust,krahmer2020complex,chen2022error}.
Here, stability often refers to lower Lipschitz bounds of the nonlinear phaseless operator \cite{balan2015invertibility,FREEMAN2026101801}, which can quantify the robustness of phase retrieval under bounded noise, whether deterministic or adversarial.
For least squares estimators or $\ell_2$-loss-based iterative algorithms, the error bound under bounded noise typically takes the form $\mathcal{O}\lk \frac{\lV\pmb{\xi}\rV_2}{\sqrt{m}\lV\x\rV_2}\rk$ \cite{chen2017solving,kabanava2016stable,zhang2017nonconvex, wang2018phase,chen2022error}.
However, for the Poisson and heavy-tailed models considered in this paper, such a bound is far from optimal \cite{chen2017solving,chen2022error}.
Another line of work \cite{hand2017phaselift,li2016low,zhang2018median,duchi2019solving,huang2023robust,huang2025adversarial,buna2024robust,kim2025robust} investigated the robustness of phase retrieval in the presence of outliers, which often arise due to sensing errors or model mismatches \cite{yeh2015experimental}.
Most of these studies typically focused on mixed noise settings, where the observation model includes both bounded noise (or random noise) and outliers. 
Notably, the outliers may be adversarial—deliberately corrupting part of the observed data \cite{duchi2019solving,huang2023robust,huang2025adversarial}. 
Thereby, the treatment in these works also differs significantly from random noise models considered in this paper. 

\subsection{Contributions of This Paper}

This paper investigates stable recovery guarantees for phase retrieval under two realistic and challenging noise settings, Poisson model~\eqref{poisson} and heavy-tailed model~\eqref{heavy}, using both nonconvex least squares (NCVX-LS) and convex least squares (CVX-LS) estimators.
Our key contributions are summarized as follows:
\begin{enumerate}
\item For the Poisson model~\eqref{poisson}, we demonstrate that both NCVX-LS and CVX-LS estimators attain the minimax optimal error rate $\mO\lk \sqrt{\frac{n}{m}} \rk$ once $\lV\x\rV_2$ exceeds a certain threshold. 
In this high-energy regime, the error bound is signal-independent.
In contrast, in the low-energy regime, the NCVX-LS estimator attains an error bound {\small$\mO\lk \lV\x\rV_2^{1/4}\cdot \lk \frac{n}{m} \rk^{1/4} \rk$}, which decays as the signal energy decreases.
By establishing the corresponding minimax lower bound, we further show that this rate remains nearly optimal with respect to the oversampling ratio.
These results improve upon the theoretical guarantees of Chen and Candès~\cite{chen2017solving} and Dirksen et al.~\cite{dirksen2025spectral}.
To the best of our knowledge, this is the first work that provides minimax optimal guarantees for the Poisson model in the high-energy regime, along with recovery bounds that explicitly adapt to the signal energy in the low-energy regime.

\item For the heavy-tailed model~\eqref{heavy}, we show that both the NCVX-LS and CVX-LS estimators achieve an error bound  {\small$\mathcal{O}\left(\frac{\lV \xi \rV_{L_q}}{\lV \x \rV_2} \cdot \sqrt{\frac{n}{m}}\right)$} in the high-energy regime, where the noise variables are heavy-tailed with a finite $q$-th moment ($q>2$) and may exhibit dependence on the sampling vectors.
This bound holds uniformly over all signals and matches the minimax optimal rate.
In the low-energy regime, the NCVX-LS estimator further achieves an error bound {\small$\mO\lk \sqrt{\lV \xi \rV_{L_q}}\cdot\lk \frac{n}{m} \rk^{1/4} \rk$}, which is likewise minimax optimal by our newly established minimax lower bound in this regime.
These results strengthen existing guarantees and resolve the open problem posed by Chen and K. Ng~\cite{chen2022error}.

\item  We propose a unified framework for analyzing the minimax stable performance of phase retrieval.
The key innovations in our framework are twofold: leveraging multiplier inequalities to handle noise that may depend on the sampling vectors, and providing a novel perspective on Poisson noise, which behaves as sub-exponential in the high-energy regime but heavy-tailed in the low-energy regime.
We further extend our framework to related problems, including sparse phase retrieval, low-rank positive semidefinite (PSD) matrix recovery, and random blind deconvolution, highlighting the broad applicability and theoretical strength of our approach.
\end{enumerate}

\subsection{Notation and Outline}

Throughout this paper, absolute constants are denoted by $c,c_1,C,C_1,L,\widetilde{L},L_1$, etc.
The notation $a \lesssim b$ implies that there are absolute constants $C$ for which $ a \le Cb$, $a \gtrsim b$ implies that $ a \ge Cb$,
and $a \asymp b$ implies that there are absolute constants $
0 < c < C$ for which $cb\le a\le Cb$.
The analogous notation $a \lesssim_{K} b$ and $a \gtrsim_K b$ refer to a constant that depends only on the parameter $K$. 
We also recall that $[n] = \{1,\hdots,n\}$.

We employ a variety of norms and spaces.
Let $\lV\, \cdot\, \rV_2$ be the standard Euclidean norm, and let $\ell_2^n$ be the normed space $\lk \C^n,\lV\, \cdot\, \rV_2\rk$. 
Let $\left\{\lambda_{k}\lk\Z\rk\right\}_{k=1}^{r}$ be a singular value sequence of a rank-$r$ matrix $\Z$ in descending order.
Let $\lV\Z\rV_{*}=\sum_{k=1}^{r}\lambda_{k}\lk\Z\rk$ denote the the nuclear norm; 
$\lV\Z\rV_{F}=\lk\sum_{k=1}^{r}\lambda^{2}_{k}\lk\Z\rk\rk^{1/2}$ is the Frobenius norm;
and $\lV\Z\rV_{op}=\lambda_{1}\lk\Z\rk$ denotes the operator norm. 
Let $\mathbb{S}^{n-1}$ denote the Euclidean unit sphere in $\C^n$ with respect to $\lV\, \cdot\, \rV_2$ and $\mathbb{S}_{F}$ denote the unit sphere in $\C^{n\times n}$ with respect to $\lV\, \cdot\, \rV_F$.
Let $\mathcal{S}^n$ denotes the vector space of all Hermitian matrices in $\C^{n\times n}$ and $\mathcal{S}^{n}_{+}$ denotes the set of all PSD Hermitian matrices in $\C^{n\times n}$.
The expectation is denoted by $\E$, and $\mathbb{P}$ denotes the probability of an event. 
The $L_p$-norm of a random variable $X$ is defined as $\lV X\rV_{L_p} = \lk\E \lv X\rv^p\rk^{1/p}$. 	

The organization of this paper is as follows.
Section~\ref{setup} presents the problem setup, and Section~\ref{main} states the main results. 
Section~\ref{Architecture} outlines the overall proof framework.
Section~\ref{ineq} introduces the multiplier inequality, a key technical tool, and Section~\ref{small0} describes the small ball method and the lower isometry property.
Section~\ref{proof of main} provides detailed proofs of the main theoretical results, and Section~\ref{minimax} establishes  minimax lower bounds for both two models.
Numerical simulations validating our theory are presented in Section~\ref{exp}, and additional applications of our framework are explored in Section~\ref{app}.
Section~\ref{discussion} concludes with a discussion of contributions and future research directions. 
Supplementary proofs are included in the Appendix.

\section{Problem Setup}\label{setup}

In this paper, we analyze the stable performance of phase retrieval in the presence of Poisson and heavy-tailed noise using the widely adopted least squares approach, as explored in \cite{candes2015phase4,lecue2015minimax,cai2016optimal,zhang2017nonconvex,sun2018geometric,chen2022error,buna2024robust,mcrae2025nonconvex}. 
Specifically, we examine two different estimators, with the first being the \textbf{nonconvex least squares (NCVX-LS)} approach,
        \begin{equation}\label{model1}
		\begin{array}{ll}
			\text{minimize}   & \quad\lV\ppp\lk\z\rk-\y\rV_2\\
			\text{subject to} & \quad  \z\in \C^{n},\\				 
		\end{array}
	\end{equation}	
where $\y:=\left\{y_k\right\}_{k=1}^{m}$ denotes the observation and $\ppp\lk\z\rk$ represents the phaseless operator $$\ppp\lk\z\rk:=\left\{\lv\lg\pp_k,\z\rg\rv^2\right\}_{k=1}^{m}.$$
Since it is impossible to recover the global sign (we cannot distinguish $\x$ from $e^{i\varphi}\x$), we will evaluate the solution using the euclidean distance modulo a global sign: 
for complex-valued signals, the distance between the solution $\z_{\star}$ of \eqref{model1} and the true signal $\x$ is 
 \begin{eqnarray}\label{dist}
		\textbf{dist}\lk\z_{\star},\x\rk:=\min_{\varphi\in\lz0,2\pi\rk}\lV e^{i\varphi}\z_{\star}-\x\rV_{2}.
					\end{eqnarray}		
                        
By the well known lifting technique \cite{candes2015phase,candes2013phaselift,candes2014solving}, the phaseless equations (\ref{phaseless}) can be transformed into the linear form $y_k=\lg\pp_k\pp_k^*,\x\x^*\rg$.
This reformulation allows the phase retrieval problem to be cast as a low-rank PSD matrix recovery problem.
Accordingly, the second estimator we consider in this paper is the 
\textbf{convex least squares (CVX-LS)} approach,

  \begin{equation}\label{model2}
		\begin{array}{ll}
			\text{minimize}   & \quad\lV\mA\lk\Z\rk-\y\rV_2\\
			\text{subject to} & \quad \Z\in\mS_{+}^{n}.\\				 
		\end{array}
	\end{equation}
	Here, $\mA\lk\Z\rk$ denotes the linear operator $\mA\lk\Z\rk:=\left\{\lg\pp_k\pp_k^*,\Z\rg\right\}_{k=1}^{m}$ and $\mS_{+}^{n}$ represents the PSD cone
	in $\C^{n\times n}$.
    Owing to the convexity of the formulation in \eqref{model2}, its global solution can be efficiently and reliably computed via convex programming.
	Denote the solution of \eqref{model2} by $\Z_{\star}$. Since we do not claim that $\Z_{\star}$ has low rank, we suggest estimating $\x$ by extracting the largest rank-1 component; see, e.g., \cite{candes2013phaselift}. 
	In other words, we write $\Z_{\star}$ as
	$$\Z_{\star}=\sum_{i=1}^{n}\lambda_{i}\lk\Z_{\star}\rk\xu_i\xu_i^{*},$$
	where its eigenvalues are in decreasing order and $\left\{\xu_i\right\}_{i=1}^{n}$ are mutually orthogonal, and we set
	 \begin{eqnarray}\label{solution}
	 \z_{\star}=\sqrt{\lambda_{1}\lk\Z_{\star}\rk}\xu_1 
	 \end{eqnarray}
	  as an alternative solution.

We now outline the required sampling and noise assumptions.
Following the setup in \cite{eldar2014phase,chen2015exact,cai2015rop,krahmer2020complex,chen2022error,huang2025low,mcrae2025nonconvex}, we consider sub-Gaussian sampling.

\begin{assumption}[Sampling]\label{sample}
The sampling vectors $\left\{\pp_k\right\}_{k=1}^{m}$ are independent copies of a random vector $\pp\in\C^n$, whose entries $\left\{\varphi_j\right\}_{j=1}^{n}$ are independent copies of a variable $\varphi$ satisfying:
 $\lV\varphi\rV_{\psi_2}=K,\E\lk\varphi\rk=\E\lk\varphi^2\rk=0,\E\lk\lv\varphi\rv^2\rk=1$ and $\E\lk\lv\varphi\rv^4\rk=1+\mu$ with $\mu>0$.
\end{assumption}

As stated before, we take into account two different noise models, namely Poisson model \eqref{poisson} and heavy-tailed model \eqref{heavy}. 
For the latter, we require certain statistical properties to hold.

\begin{assumption}[Noise]\label{noise0}
The two different noise models we consider are:
\begin{itemize}
\item[$\mathrm{(a)}$] Poisson model in (\ref{poisson}), that is, the probability 
\begin{eqnarray}
\mathbb{P}\lk y_k=\ell\rk=\frac{1}{\ell!}e^{-\lv\lg\pp_{k},\x \rg\rv^{2}}\lk\lv\lg\pp_{k},\x \rg\rv^{2}\rk^{\ell},\quad \ell=0,1,2,\cdots;
\end{eqnarray}
\item[$\mathrm{(b)}$] Heavy-tailed model in (\ref{heavy}) involve noise terms $\left\{\xi_k\right\}_{k=1}^{m}\in\R^m$, which are independent copies of a random variable $\xi$ satisfying $\E\lk\xi\mid\pp\rk=0$ (note that $\xi$ is not necessarily independent of $\pp$). 
Moreover, $\xi$ belongs to the space $L_q$ for some $q > 2$, that is, 
$\lV\xi\rV_{L_q}=\lk\E\lk\lv\xi^q\rv\rk\rk^{\frac{1}{q}}<\infty.$
\end{itemize}
\end{assumption}
 
We take a moment to elaborate on our assumptions.
For the sampling assumption, we require $\E\lk\varphi\rk=0$ and $\E\lk\lv\varphi\rv^2\rk=1$, thus $\pp$ is a complex isotropic random vector satisfying $\E\lk\pp\rk=\pmb{0}$ and $\E\lk\pp\pp^*\rk=\pmb{I}_n$.
In addition, we impose the conditions $\E\lk\lv\varphi\rv^4\rk=1+\mu$ with $\mu>0$ and $\E\lk\varphi^{2}\rk = 0$ to avoid certain ambiguities.
If instead $\E\lk\lv\varphi\rv^4\rk=\E\lk\lv\varphi\rv^2\rk=1$ (i.e., $\lv\varphi\rv=1$ almost surely, with the Rademacher variable as a special case),
then the standard basis vectors of $\C^n$ would become indistinguishable.
Similarly, if $\E\lk\lv\varphi^{2}\rv\rk=\E\lk\lv\varphi\rv^2\rk=1$ (i.e., $\varphi=\lambda\tilde{\varphi}$ almost surely for some fixed $\lambda\in\C$ and  $\tilde{\varphi}\in\R$ is a real random variable), then $\x$ would be indistinguishable from its complex conjugate $\overline{\x}$.
Hence, we assume $\E\lk\varphi^{2}\rk=0$ for the sake of simplicity.
For a more detailed discussion on these conditions, see \cite{krahmer2020complex}.
As an example, the complex Gaussian variable $\varphi=\frac{1}{\sqrt{2}}\lk X+iY\rk$, where $X,Y \sim \mathcal{N}(0,1)$ are independent, satisfies the conditions on $\varphi$ in Assumption~\ref{sample}, with its sub-Gaussian norm $K$ being an absolute constant.

Regarding the noise assumption, Poisson noise is a standard case and has been extensively discussed in \cite{chen2017solving,chang2018total,barmherzig2019holographic,diederichs2024wirtinger,romer2024one,dirksen2025spectral,allain2025phasebook}.
For heavy-tailed noise, it appears necessary for the least squares estimator that the moment condition $\lV\xi\rV_{L_q}<\infty$ holds for some $q>2$ (see, e.g., \cite{han2019convergence}), and this requirement is commonly adopted in the literature (see, e.g., \cite{lecue2018regularization}).
One could potentially relax this condition by using alternative robust estimators or by imposing additional restrictions on the noise.
Notably, we assume $\E\lk\xi\mid\pp\rk=0$, which implies that $\xi$ is generally not independent of $\pp$, thereby broadening the class of admissible noise models.
For example, Poisson noise can serve as a special case. 
We can treat the noise in Poisson model \eqref{poisson} as an additive term, denoted by $\xi$, and we rewrite it as:
\begin{eqnarray*}
\xi=\text{Poisson}\lk\lv\lg\pp,\x \rg\rv^{2}\rk-\lv\lg\pp,\x \rg\rv^{2}.
\end{eqnarray*}
It is evident that $\xi$ depends on both the sampling term $\pp$ and the true signal $\x$, yet satisfies $\E\lk\xi\mid\pp\rk=0$; moreover, it is evident that its noise level is governed by both $\pp$ and $\x$.

\section{Main Results}\label{main}

In this paper, we demonstrate that, under appropriate conditions on the sampling vectors and noise,
the estimation errors of NCVX-LS  (\ref{model1}) and CVX-LS (\ref{model2}) attain the minimax optimal rates under both the Poisson model \eqref{poisson} and the heavy-tailed model \eqref{heavy}.
Moreover, we establish  adaptive behavior with respect to the signal energy in both  models.

\subsection{Poisson Model}

We begin with a result for the Poisson model \eqref{poisson} that applies uniformly across the entire range of signal energy. 

\begin{theorem}\label{thm poisson}
Suppose that sampling vectors $\left\{\pp_k\right\}_{k=1}^{m}$ satisfy Assumption \ref{sample}, and that the Poisson model \eqref{poisson} follows the distribution specified in Assumption~\ref{noise0}~(a).
Then there exist some universal constants $L,c,C_1,C_2,C_3 > 0$ dependent only on $K$ and $\mu$ such that when $m\ge L n$, with probability at least $1-\mathcal{O}\lk e^{-c n}\rk$, simultanesouly for all signals $\x\in\C^n$, 
the estimates produced by the NCVX-LS estimator obey
\begin{align}
\label{dist1}
\textbf{dist}\lk\z_\star,\x\rk
\le C_1 \min \bigg\{ 
    & \max\left\{ K, \frac{1}{\lV\x\rV_2} \right\} \cdot \sqrt{\frac{n}{m}}, \notag \\
    & \max\left\{ 1, \sqrt{K\lV\x\rV_2} \right\} \cdot \lk \frac{n}{m} \rk^{1/4} 
\bigg\}.
\end{align}
For the CVX-LS estimator, one has 
\begin{eqnarray}\label{dist2}
\lV\Z_{\star}-\x\x^*\rV_{F}
\le C_2 \max\left\{1, K\lV\x\rV_{2} \right\}\cdot \sqrt{\frac{n}{m}}.
\end{eqnarray}
By finding the largest eigenvector with largest eigenvalue of $\Z_{\star}$, one can also construct an estimate obeying
 \begin{eqnarray}\label{dist3}
		\textbf{dist}\lk\z_{\star},\x\rk\le C_3 \max\left\{K,\frac{1}{\lV\x\rV_{2}} \right\}\cdot\sqrt{\frac{n}{m}}.
		\end{eqnarray}	
\end{theorem}

We compare our results with those of Chen and Cand{\'e}s \cite{chen2017solving} and Dirksen et al. \cite{dirksen2025spectral}; see Table~\ref{tab:1} for a brief sketch.
Theorem \ref{thm poisson} establishes that, in the high-energy regime when $\lV\x\rV_2\ge\frac{1}{K}$, at the optimal sampling order $m=\mathcal{O}\lk n\rk$, for a broader class of sub-Gaussian sampling, both the NCVX-LS and CVX-LS estimators achieves at least the following error bound:
\begin{eqnarray}\label{sqrt}
\textbf{dist}\lk\z_{\star},\x\rk\le C\lk K,\mu\rk\sqrt{\frac{n}{m}}.
\end{eqnarray}
This result improves upon the existing upper bounds established in \cite{chen2017solving} and \cite{dirksen2025spectral}.
Specifically, the error bound $\mathcal{O}\lk 1\rk$ in \cite{chen2017solving} does not vanish as the oversampling ratio increases, 
and the error bound $\widetilde{\mO}\left( \lV\x\rV_2\cdot\left(\frac{n}{m}\right)^{1/4} \right)$ (see \eqref{prior of low-dose} in Section~\ref{prior}) in \cite{dirksen2025spectral} roughly grows linearly with $\lV\x\rV_2$ and exhibits a suboptimal convergence rate of $\widetilde{\mO}\lk\lk\frac{n}{m}\rk^{1/4}\rk$.
In contrast, our result \eqref{sqrt} achieves the minimax optimal rate $\mO\lk\sqrt{\frac{n}{m}}\rk$ without dependence on
$\lV\x\rV_2$.
The corresponding minimax lower bound is provided in Theorem~\ref{minimax poisson} below.

For the low-energy regime when $\lV\x\rV_2\le \frac{1}{K}$, Theorem~\ref{thm poisson} establishes that the NCVX-LS estimator achieves the following error bound:
\begin{eqnarray}\label{1/4}
 \textbf{dist}\lk\z_\star,\x\rk\le C_1\min\left\{ \frac{1}{\lV\x\rV_{2}} \cdot \sqrt{\frac{n}{m}},\,\lk\frac{n}{m}\rk^{1/4}\right\}\le C_1 \lk\frac{n}{m}\rk^{1/4}.
		\end{eqnarray}
The result in \cite{chen2017solving} does not apply in this low-energy regime. 
Our result \eqref{1/4} matches the error bound $\widetilde{\mO}\left(\left(\frac{n}{m}\right)^{1/4} \right)$ (see \eqref{prior of low-dose} in Section~\ref{prior}) given in \cite{dirksen2025spectral}, but slightly improves upon it by moving certain logarithmic factors.
For the CVX-LS estimator, Theorem~\ref{thm poisson} establishes an error bound $\mO\lk\sqrt{\frac{n}{m}}\rk$ with respect to the distance $\lV\Z_{\star}-\x\x^*\rV_{F}$, and $\mO\lk\frac{1}{\lV\x\rV_{2}}\sqrt{\frac{n}{m}}\rk$ with respect to the distance $\textbf{dist}\lk\z_{\star},\x\rk$.
The latter is slightly weaker than that for the NCVX-LS estimator in this regime.

Note that the intensity of Poisson noise diminishes as the energy of $\x$ decreases.
However, in the low-energy regime, apart from the result of \cite{chen2017solving}, which does not apply, the error bounds in \cite{dirksen2025spectral} and in our Theorem~\ref{thm poisson} (e.g., \eqref{dist1}, \eqref{dist2}) remain independent of $\lV\x\rV_2$, and therefore do not diminish as $\lVert\x\rVert_2$ decreases.
Hence, in this regime, we expect the error bounds to improve accordingly, scaling with the energy of $\x$.
To capture this behavior more precisely, we present the following theorem, at the cost of a slightly weaker probability guarantee compared to Theorem~\ref{thm poisson}.

\begin{theorem}\label{low-dose}
Suppose that sampling vectors $\left\{\pp_k\right\}_{k=1}^{m}$ satisfy Assumption~\ref{sample}, and that the Poisson model \eqref{poisson} follows the distribution specified in Assumption~\ref{noise0}~(a).
Let $\Gamma:=\left\{\x\in\C^n:\lV\x\rV_2\le\frac{1}{K}\right\}$.
Then there exist some universal constants $L,c,C_1,C_2,C_3 > 0$ dependent only on $K$ and $\mu$ such that when $m\ge L n$, with probability at least 
\begin{eqnarray*}
1-\mathcal{O}\lk \frac{\log^4m}{m}\rk-\mathcal{O}\lk e^{-c n}\rk,
\end{eqnarray*}
simultanesouly for all signals $\x\in\Gamma$, the estimates produced by the NCVX-LS estimator obey
 \begin{eqnarray}
\textbf{dist}\lk\z_\star,\x\rk \le C_1 \min\left\{\sqrt{\frac{K}{\lV\x\rV_2}}\cdot\sqrt{\frac{n}{m}},\,\lk K\lV\x\rV_{2}\rk^{1/4}\cdot\lk\frac{n}{m}\rk^{1/4}\right\}.		
\end{eqnarray}
For the CVX-LS estimator, we can obtain 
\begin{eqnarray}\label{low-dose:F}
\lV\Z_{\star}-\x\x^*\rV_{F}
\le C_2 \sqrt{ K\lV\x\rV_{2} }\cdot \sqrt{\frac{n}{m}}. 
\end{eqnarray}
By finding the largest eigenvector with largest eigenvalue of $\Z_{\star}$, we can construct an estimate obeying
 \begin{eqnarray}\label{low-dose:dist}
		\textbf{dist}\lk\z_{\star},\x\rk\le C_3 \sqrt{\frac{K}{\lV\x\rV_{2}}} \cdot\sqrt{\frac{n}{m}}.
					\end{eqnarray}	
                    \end{theorem}
\begin{remark}
In contrast to Theorem~\ref{thm poisson}, which exploits the sub-exponential behavior of Poisson noise, Theorem~\ref{low-dose} relies on a different insight: 
in the low-energy regime, the observation $\text{Poisson}\lk\lv\lg\pp,\x \rg\rv^{2}\rk$ is highly likely to take value zero, while nonzero outcomes occur only rarely. 
These nonzero observations induce large relative deviations from the true intensity $\lv\lg\pp,\x \rg\rv^{2}$ and can thus be regarded as heavy-tailed outliers.
This heavy-tailed interpretation naturally leads to a slightly weaker high-probability guarantee in Theorem~\ref{low-dose} compared to Theorem~\ref{thm poisson}.
\end{remark}

Theorem~\ref{low-dose} significantly refines the recovery guarantees in the low-energy regime.
Specifically, the NCVX-LS estimator achieves an error bound 
\begin{eqnarray}
\mathcal{O}\lk \lV\x\rV_{2}^{1/4}\cdot\lk\frac{n}{m}\rk^{1/4}\rk.
\end{eqnarray}
This result refines the explicit dependence on $\lV\x\rV_2$, thereby offering a nontrivial decay in error as the energy of $\x$ decreases. 
Moreover, by Theorem~\ref{minimax poisson} below, this bound is nearly optimal with respect to the oversampling ratio $\frac{m}{n}$.
In contrast, the guarantee in~\cite{dirksen2025spectral} remains fixed at the rate $\widetilde{\mathcal{O}}\left(\left(\tfrac{n}{m}\right)^{1/4}\right)$, regardless of the signal energy.
Besides, the bounds for the CVX-LS estimator also benefits from this adaptive behavior. 
Although \eqref{low-dose:F} and \eqref{low-dose:dist} in Theorem~\ref{low-dose} do not attain the same error rate as the NCVX-LS estimator,
\eqref{low-dose:F} nonetheless scales as $\mathcal{O}\lk\sqrt{\lV\x\rV_{2}} \cdot\sqrt{\frac{n}{m}}\rk$ in Frobenius norm,
exhibiting a decay in error as the energy of $\x$ decreases.
Meanwhile, \eqref{low-dose:dist} provides a bound on $\mathrm{dist}\left(\z_{\star},\x\right)$ with an inverse square-root dependence on $\lV\x\rV_2$, improving upon \eqref{dist3} in Theorem~\ref{thm poisson}.


We further establish fundamental lower bounds on the minimax estimation error for the Poisson model~\eqref{poisson} under complex Gaussian sampling.

\begin{theorem}\label{minimax poisson}
Suppose that $\left\{\pp_k\right\}_{k=1}^m \overset{\text{i.i.d.}}{\sim} \mathcal{C} \mathcal{N}\lk \pmb{0}, \pmb{I}_n \rk$, where $m,n$ are sufficiently large and $m \ge Ln$ for some sufficiently large constant $L > 0$.
With probability approaching 1, the minimax risk under the Poisson model \eqref{poisson} obeys:
    \begin{itemize}
		\item[$\mathrm{(a)}$] If $\frac{m}{n^2}\le\frac{L_1}{\log^3 m}$ for some universal constant 
$L_1>0$, then for any $\x\in\C^n\setminus\{\pmb{0}\}$,
        \begin{equation*}
\inf\limits_{\widehat{\x}}\sup\limits_{\x\in\C^n}\E\left[\textbf{dist}\lk\widehat{\x},\x\rk\right]\ge C_1 \min\left\{\lV\x\rV_2,\frac{\sqrt{\frac{n}{m}}}{1+\frac{\log^{3/4} m}{\sqrt{\lV\x\rV_2}}\cdot\lk \frac{m}{n}\rk^{1/4}}\right\};
\end{equation*}
\item[$\mathrm{(b)}$] If $\frac{m}{n}\le L_2\log m$ for some universal constant 
$L_2>0$,
then for any $\x\in\C^n\setminus\{\pmb{0}\}$ such that $\lV\x\rV_2=o\lk\frac{\sqrt\frac{n}{m}}{\log^{3/2}m}\rk$,
\begin{equation*}
\inf\limits_{\widehat{\x}}\sup\limits_{\x\in\C^n}\E\left[\textbf{dist}\lk\widehat{\x},\x\rk\right]\ge C_2 \sqrt{\lV\x\rV_2}\cdot\frac{\lk\frac{n}{m}\rk^{1/4}}{\log^{5/4}m}.
\end{equation*}
 \end{itemize}
Here, $C_1,C_2> 0$ are universal constants independent of $n$ and $m$, and the infimum is over all estimators $\widehat{\x}$.
\end{theorem}

Building on the minimax lower bounds established above, we now examine the optimality of our results in Theorem~\ref{thm poisson} and Theorem~\ref{low-dose}:
    \begin{enumerate}
\item \textbf{High-energy regime:}
\textbf{Part}~$(\mathrm{a})$ of Theorem~\ref{minimax poisson} implies that, if 
\begin{equation*}
\lV\x\rV_2=\Omega\lk\log^{3/2}m\cdot\sqrt\frac{m}{n}\rk,
\end{equation*}
then no estimator can attain an estimation error smaller than $\Omega\left(\sqrt{\frac{n}{m}}\right)$.
This lower bound matches the upper bound $\mathcal{O}\lk\sqrt{\frac{n}{m}}\rk$ achieved by both the NCVX-LS and CVX-LS estimators in Theorem~\ref{thm poisson} when $\lV\x\rV_2\ge1/K$, thereby confirming their minimax optimality under the Poisson model~\eqref{poisson} in the high-energy regime.
\textbf{Part}~$(\mathrm{a})$ of Theorem~\ref{minimax poisson} holds under the condition $Ln\le m\le L_1\frac{ n^2}{\log^3 m}$, which broadens the result of \cite{chen2017solving}, where the minimax lower bound was established only for a fixed oversampling ratio $\frac{m}{n}$.
\item \textbf{Intermediate-energy regime:}  If if $c_1 \sqrt{\frac{n}{m}}\le\lV\x\rV_2\le c_2\sqrt{\frac{m}{n}}$ for some positive constants $c_1,c_2$, then \textbf{Part}~$(\mathrm{a})$ of Theorem~\ref{minimax poisson} implies a minimax lower bound at the oder of $\lV\x\rV_2\asymp\sqrt\frac{n}{m}$, which nearly matches the performance of both NCVX-LS and CVX-LS in Theorem~\ref{low-dose} for fixed oversampling ratio $\frac{m}{n}$.
\item \textbf{Low-energy regime:}
In the low-energy regime that $\lV\x\rV_2=o\lk\frac{\sqrt\frac{n}{m}}{\log^{5/2}m}\rk$, \textbf{Part}~$(\mathrm{b})$ of  Theorem~\ref{minimax poisson} provide a minimax lower bound 
\begin{equation*}
\Omega\lk \sqrt{\lV\x\rV_2}\cdot\frac{\lk\frac{n}{m}\rk^{1/4}}{\log^{5/4}m}\rk.
\end{equation*}
This rate depends on both $\lV\x\rV_2$ and the oversampling ratio $\frac{m}{n}$, scaling as $\sqrt{\lV\x\rV_2}$ and $\lk\frac{n}{m}\rk^{1/4}$.
Our NCVX-LS estimator in Theorem~\ref{low-dose} achieves an error bound $\mathcal{O}\lk\lV\x\rV_2^{1/4}\cdot\lk\frac{n}{m}\rk^{1/4}\rk$,
which scales as $\lV\x\rV_2^{1/4}$ and $\lk\frac{n}{m}\rk^{1/4}$.
Thus, this upper bound is nearly optimal with respect to the oversampling ratio $\frac{m}{n}$, up to a $\log^{5/4}m$ factor.
However, there remains a small gap in the dependence on $\lV\x\rV_2$ between the minimax lower bound and our upper bound. 
This gap may be closed by considering alternative estimators; see Section~\ref{discussion} for further comments. 
    \end{enumerate}

\subsection{Heavy-tailed Model}

We state our results for phase retrieval under heavy-tailed model \eqref{heavy} here.

\begin{theorem}\label{thm heavy}
Suppose that sampling vectors $\left\{\pp_k\right\}_{k=1}^{m}$ satisfy Assumption~\ref{sample}, and the heavy-tailed model \eqref{heavy} satisfies the condition in Assumption \ref{noise0} $\mathrm{(b)}$ with $q>2$.
Then there exist some universal constants $L,c,C_1,C_2,C_3 > 0$ dependent only on $K,\mu$ and $q$ such that when provided that $m\ge L n$, with probability at least $$1-\mathcal{O}\lk m^{-\lk\lk q/2\rk-1\rk}\log^{q} m\rk-\mathcal{O}\lk e^{-c n}\rk,$$ simultanesouly for all signals $\x\in\C^n$, the estimates produced by the NCVX-LS estimator obey
 \begin{eqnarray}
		\textbf{dist}\lk\z_\star,\x\rk \le C_1 \min\left\{\frac{\lV\xi\rV_{L_q}}{\lV\x\rV_2}\cdot\sqrt{\frac{n}{m}},\,\sqrt{\lV\xi\rV_{L_q}}\cdot\lk\frac{n}{m}\rk^{1/4}\right\}.
				\end{eqnarray}	
For the CVX-LS estimator, we have 
 \begin{eqnarray}
		\lV\Z_{\star}-\x\x^*\rV_{F}\le C_2\lV\xi\rV_{L_q}\cdot\sqrt{\frac{n}{m}}.
					\end{eqnarray}	
By finding the largest eigenvector with largest eigenvalue of $\Z_{\star}$, one can construct an estimate obeying
 \begin{eqnarray}
		\textbf{dist}\lk\z_{\star},\x\rk\le C_3\frac{\lV\xi\rV_{L_q}}{\lV\x\rV_2}\cdot\sqrt{\frac{n}{m}}.
					\end{eqnarray}	
\end{theorem}

We highlight the distinctions and improvements of Theorem~\ref{thm heavy} over prior work; see Table~\ref{tab:2} for a summary.
Specifically, Theorem~\ref{thm heavy} shows that for all signal $\x\in\C^n$ and i.i.d. mean-zero heavy-tailed noise $\xi$,
which may depend on the sampling term and satisfies
a finite $q$-th moment for some $q > 2$, both the NCVX-LS and CVX-LS estimators attain the error bound \begin{eqnarray*}
\mathcal{O}\lk\frac{\lV\xi\rV_{L_q}}{\lV\x\rV_2}\cdot\sqrt{\frac{n}{m}}\rk.
\end{eqnarray*}
We will later show in Theorem~\ref{minimax heavy} that this rate is nearly minimax optimal in the high-energy regime (i.e., when $\lV\x\rV_2$ exceeds a certain threshold).
Moreover, the NCVX-LS estimator achieves the error bound 
\begin{eqnarray*}
\mathcal{O}\lk\sqrt{\lV\xi\rV_{L_q}}\cdot\lk\frac{n}{m}\rk^{1/4}\rk,
\end{eqnarray*}
which is also nearly minimax optimal, as discussed after Theorem~\ref{minimax heavy}.

Our results improve upon the previous error bound (see~\eqref{prior of heavy} in Section~\ref{prior}) in~\cite{chen2022error} by eliminating the dependence on $q$ in the oversampling ratio $\frac{m}{n}$ and by providing uniform guarantees for all signals $\x\in\C^n$, thereby resolving the open question posed therein of whether faster convergence rates than~\eqref{prior of heavy} and uniform recovery under heavy-tailed noise can be achieved.
Our analysis also removes two restrictive assumptions imposed in~\cite{chen2022error}, namely, the symmetry of the noise and its independence from the sampling vectors.
This substantially broadens the applicability of our results to more realistic and potentially dependent noise models.
Our results answer the question posed in \cite{chen2022error} affirmatively for the regime $q > 2$, whereas \cite{chen2022error} considered the broader regime $q > 1$.
For the low-moment regime $1 \le q \le 2$, or in the absence of moment assumptions, stronger structural conditions on the noise (such as the symmetry assumption in~\cite{chen2022error} or specific distributional assumptions in~\cite{shen2025computationally}) and more robust estimation techniques (e.g., the Huber estimator~\cite{sun2020adaptive,yu2024low,shen2025computationally}) may be required.
A comprehensive study of this low-moment setting is left for future work.

We conclude this section with the following theorem, which establishes fundamental minimax lower bounds for the estimation error under Gaussian noise. 
This theorem provides a benchmark for evaluating the stability of estimators in the heavy-tailed model~\eqref{heavy}.
The result in \textbf{Part}~$(\mathrm{a})$  aligns with that of Lecué and Mendelson~\cite{lecue2015minimax}, whereas \textbf{Part}~$(\mathrm{b})$ appears to be novel.

\begin{theorem}\label{minimax heavy}
Consider the noise model $y_{k}=\lv\lg\pp_{k},\x \rg\rv^{2}+\xi_k,\,k\in[m]$,
where $\left\{\pp_k\right\}_{k=1}^m \overset{\text{i.i.d.}}{\sim} \mathcal{C} \mathcal{N}\lk \pmb{0}, \pmb{I}_n \rk$ and $\left\{\xi_k\right\}_{k=1}^m \overset{\text{i.i.d.}}{\sim}\mathcal{N}\lk 0, \sigma^2 \rk$ are independent of $\left\{\pp_k\right\}_{k=1}^m$.
Suppose that $m,n$ are sufficiently large and $m \ge Ln$ for some sufficiently large constant $L > 0$.
With probability approaching 1,
the minimax risk obeys:
    \begin{itemize}
		\item[$\mathrm{(a)}$] For any $\x\in\C^n\setminus\{\pmb{0}\}$,
        \begin{equation*}
\inf\limits_{\widehat{\x}}\sup\limits_{\x\in\C^n}\E\left[\textbf{dist}\lk\widehat{\x},\x\rk\right]\ge C_1 \min\left\{\lV\x\rV_2,\frac{\sqrt{\frac{n}{m}}}{\lV\x\rV_2\sqrt{\log m}/\sigma+\lk\frac{\log m}{\sigma^2}\rk^{1/4}\cdot\lk \frac{n}{m}\rk^{1/4}}\right\};
\end{equation*}
\item[$\mathrm{(b)}$] For any $\x\in\C^n\setminus\{\pmb{0}\}$ such that $\lV\x\rV_2=o\lk \sqrt{\sigma}\cdot\frac{\lk\frac{n}{m}\rk^{1/4}}{\log^{1/4}m}\rk$,
\begin{equation*}
\inf\limits_{\widehat{\x}}\sup\limits_{\x\in\C^n}\E\left[\textbf{dist}\lk\widehat{\x},\x\rk\right]\ge C_2 \sqrt{\sigma}\cdot\frac{\lk\frac{n}{m}\rk^{1/4}}{\log^{1/4}m}.
\end{equation*}
 \end{itemize}
Here, $C_1,C_2> 0$ are universal constants independent of $n$ and $m$, and the infimum is over all estimators $\widehat{\x}$.
\end{theorem}

We next examine the minimax optimality of our results in Theorem~\ref{thm heavy}. 
   \begin{enumerate}
\item \textbf{High-energy regime:}
\textbf{Part}~$(\mathrm{a})$ of Theorem~\ref{minimax heavy} states that, if 
\begin{equation*}
\lV\x\rV_2=\Omega\lk\sqrt{\sigma}\cdot\log^{5/4}m\lk\frac{n}{m}\rk^{1/4}\rk,
\end{equation*}
then no estimator can attain an  error rate smaller than 
$\Omega\left(\frac{\sigma}{\lV\x\rV_2}\cdot\sqrt{\frac{n}{m\log m}}\right).$
This lower bound coincides, up to a $\sqrt{\log m}$ factor, with the upper bound $\mathcal{O}\lk\frac{\lV\xi\rV_{L_q}}{\lV\x\rV_2}\cdot\sqrt{\frac{n}{m}}\rk$, attained by both the NCVX-LS and CVX-LS estimators in Theorem~\ref{thm heavy}, thereby establishing their minimax optimality under the heavy-tailed model~\eqref{heavy} in the high-energy regime.
\item \textbf{Intermediate-energy regime:}  If $\lV\x\rV_2\asymp\sqrt{\sigma}\cdot\lk\frac{n}{m}\rk^{1/4}$,
 then \textbf{Part}~$(\mathrm{a})$ of Theorem~\ref{minimax heavy} yields a minimax lower bound of order $\lV\x\rV_2\asymp\sqrt{\sigma}\cdot\lk\frac{n}{m}\rk^{1/4}$, up to logarithmic factors.
This rate coincides with the performance achieved by both the NCVX-LS and CVX-LS estimators in Theorem~\ref{thm heavy}.
\item \textbf{Low-energy regime:}
If $\lV\x\rV_2=o\lk \sqrt{\sigma}\cdot\frac{\lk\frac{n}{m}\rk^{1/4}}{\log^{1/4}m}\rk$, \textbf{Part}~$(\mathrm{b})$ of Theorem~\ref{minimax heavy} establishes a minimax lower bound of
\begin{equation*}
\Omega\lk  \sqrt{\sigma}\cdot\frac{\lk\frac{n}{m}\rk^{1/4}}{\log^{1/4}m}\rk,
\end{equation*}
which matches, up to a $\log^{1/4}m$ factor, the upper bound achieved by our NCVX-LS estimator in Theorem~\ref{thm heavy},
thereby establishing its minimax optimality in the low-energy regime.
\end{enumerate}

\section{Towards An Architecture}\label{Architecture}

To unify the treatment of Poisson model \eqref{poisson} and heavy-tailed model \eqref{heavy}, we express the Poisson observations as follows:
\begin{eqnarray*}
			y_{k}=\lv\lg\pp_{k},\x \rg\rv^{2}+\xi_k ,\quad k=1,\cdots,m,
				\end{eqnarray*}
	where $\xi_k:=\text{Poisson}\lk\lv\lg\pp_{k},\x \rg\rv^{2}\rk-\lv\lg\pp_{k},\x \rg\rv^{2}$.
	Note that in this case, the noise term $\left\{\xi_k\right\}_{k=1}^{m}$ depends on both the sampling vectors $\left\{\pp_k\right\}_{k=1}^{m}$ and the ground truth $\x$.
     
In order to handle the NCVX-LS estimator \eqref{model1}, we first perform a natural decomposition on $\ell_2$-loss as in \cite{mendelson2015learning,lecue2018regularization,chen2022error},
which allows us to obtain the empirical form
 \begin{eqnarray*}
 \begin{aligned}
		\mathcal{P}_{m}\lk\z\rk:&=\lV\ppp\lk\z\rk-\y\rV_{2}^{2}-\lV\ppp\lk\x\rk-\y\rV_{2}^{2}\\
	        &=\sum_{k=1}^{m}\lv\lg\pp_k\pp_k^*,\z\z^*-\x\x^{*}\rg\rv^2-2\sum_{k=1}^{m}\xi_k\lg\pp_k\pp_k^*,\z\z^*-\x\x^{*}\rg.
		\end{aligned}
				\end{eqnarray*}
				
Hence, one may bound $\mathcal{P}_{m}\lk\z\rk$ from below by showing that with high probability for some specific admissible set $\mathcal{E}\subset\C^{n\times n}$,
\begin{itemize}
\item the \textbf{\textit{Sampling Lower Bound Condition}} (\textit{SLBC}) with respect to the Frobenius norm ($\lV\,\cdot\,\rV_{F}$) holds, that is, there exists a positive constant $\alpha$ such that
 \begin{eqnarray}\label{non}
		\sum_{k=1}^{m}\lv\lg\pp_k\pp_k^*,\M\rg\rv^2\ge \alpha \lV\M\rV^{2}_{F}, \quad\forall \ \M\in\mathcal{E},
				\end{eqnarray}
\item the \textbf{\textit{Noise Upper Bound Condition}} (\textit{NUBC}) with respect to the Frobenius norm ($\lV\,\cdot\,\rV_{F}$) holds, that is, there exists a positive constant $\beta$ such that
 \begin{eqnarray}\label{noise}
		\lv \sum_{k=1}^{m}\xi_k\lg\pp_k\pp_k^*,\M\rg \rv\le \beta \lV\M\rV_{F}, \quad\forall \ \M\in\mathcal{E}.
				\end{eqnarray}
\end{itemize}

By the optimality of $\z_\star$, we have $\mathcal{P}_{m}\lk\z_\star\rk\le0$. 
Therefore, if we define the admissible set $\mathcal{E}$ as
\begin{eqnarray}\label{E1}
		\mathcal{E}_{\text{ncvx}}:=\left\{ \z\z^*-\x\x^*:\z,\x\in\C^{n}\right\}
	\end{eqnarray}
and if the sampling vectors $\{\pp_k\}_{k=1}^{m}$ satisfy both \textit{SLBC} \eqref{non} and \textit{NUBC} \eqref{noise} with respect to $\lV\,\cdot\,\rV_{F}$, then, conditioned on that event, the estimation error for the NCVX-LS estimator  \eqref{model1}  over all $\x\in\C^n$ is bounded by
 \begin{eqnarray}\label{noise2}
		\lV\z_\star\z_\star^{*}-\x\x^{*}\rV_{F}\le\frac{2\beta}{\alpha}.
				\end{eqnarray}
To derive a $\textbf{dist}(\z_{\star}, \x)$-type estimation bound defined in \eqref{dist}, we present the following distance inequality.
	 \begin{proposition}\label{dis1}
The distance between $\textbf{dist}\lk\z_{\star},\x\rk$ and $\lV\z_\star\z_\star^{*}-\x\x^{*}\rV_{F}$ satisfies that
\begin{equation*}\label{f1}
\lV\z_\star\z_\star^{*}-\x\x^{*}\rV_{F}\ge\frac{1}{2}\max\left\{\textbf{dist}\lk\z_{\star},\x\rk\cdot \lV\x\rV_2,\textbf{dist}^2\lk\z_{\star},\x\rk\right\}.
 \end{equation*}
\end{proposition}
\begin{proof}
See Appendix \ref{app dis1}.
\end{proof}
\noindent Combining \eqref{noise2} with Proposition \ref{dis1}, we obtain the following error bound for the NCVX-LS estimator \eqref{model1}:
                \begin{eqnarray}\label{error1}
		\textbf{dist}\lk\z_\star,\x\rk \le\min\left\{\frac{1}{\lV\x\rV_{2}}\cdot\frac{4\beta}{\alpha},2\sqrt{\frac{\beta}{\alpha}}\right\}.
				\end{eqnarray}
				
Using a similar approach, we handle the CVX-LS estimator \eqref{model2}.
By natural decomposition and for all $\Z\in\mS_{+}^{n}$, we have
\begin{eqnarray*}
 \begin{aligned}
		\mathcal{P}_{m}\lk\Z\rk:&=\lV\mA\lk\Z\rk-\y\rV_{2}^{2}-\lV\mA\lk\x\x^{*}\rk-\y\rV_{2}^{2}\\
		  &=\sum_{k=1}^{m}\lv\lg\pp_k\pp_k^*,\Z-\x\x^{*}\rg\rv^2-2\sum_{k=1}^{m}\xi_k\lg\pp_k\pp_k^*,\Z-\x\x^{*}\rg.
		\end{aligned}
				\end{eqnarray*}
In this case and to establish a uniform recovery result over all $\x\in\C^n$, we define the admissible set as
\begin{eqnarray}\label{E2}
		\mathcal{E}_{\text{cvx}}:=\left\{\Z-\x\x^*:\Z\in\mS_{+}^{n},\x\in\C^n\right\}.
	\end{eqnarray}
	Unlike the admissible set $\mathcal{E}_{\text{ncvx}}$, which is confined to a low-rank structure (the elements in $\mathcal{E}_{\text{ncvx}}$ have rank at most 2), $\mathcal{E}_{\text{cvx}}$ spans the entire PSD cone. 
	As a result, its geometric complexity is nearly as large as that of the entire ambient space.	
	To address this, we adopt the strategy outlined in \cite{krahmer2020complex}, which partitions the admissible set $\mathcal{E}_{\text{cvx}}$ into two components.
    This strategy can be viewed as a variation of the rank null space properties (rank NSP) \cite{recht2011null,kabanava2016stable}.
	 In particular, the following proposition states that any matrix in $\mathcal{E}_{\text{cvx}}$ possesses at most one negative eigenvalue. 
     
\begin{proposition}[\cite{krahmer2020complex}]\label{one}
Suppose that $\M\in \mathcal{E}_{\text{cvx}}$. 
Then $\M$ has at most one strictly negative eigenvalue.
\end{proposition}

\begin{proof}
See Appendix~\ref{alg}.
\end{proof}

Recall that for a matrix $\M\in \mathcal{S}^n$, we denote its eigenvalues by $\left\{ \lambda_i \lk\M\rk \right\}^n_{i=1} $ in decreasing order. 
By Proposition \ref{one}, we know that $ \lambda_i \lk\M\rk  \ge 0$ for all $ i \in \left[n-1\right] $ and also for all $\M \in \mathcal{E}_{\text{cvx}}$. 
We then partition $\mathcal{E}_{\text{cvx}}$ into two components: an approximately low-rank subset 
\begin{eqnarray}\label{E21}
		\mathcal{E}_{\text{cvx,1}}:= \left\{\M\in  \mathcal{E}_{\text{cvx}}:   -\lambda_n \lk\M\rk > \frac{1}{2}  \sum_{i=1}^{n-1} \lambda_i \lk\M\rk   \right\},
			\end{eqnarray}
  and an almost PSD subset 
\begin{eqnarray}\label{E22}
		\mathcal{E}_{\text{cvx,2}}:= \left\{\M\in  \mathcal{E}_{\text{cvx}}:  -\lambda_n \lk\M\rk \le \frac{1}{2}  \sum_{i=1}^{n-1} \lambda_i \lk\M\rk  \right\}.
	\end{eqnarray}
	The reason why the elements in $\mathcal{E}_{\text{cvx,1}}$ are approximately of low rank is that $-\lambda_n \lk\M\rk$ dominates. 
	In contrast, the elements in $\mathcal{E}_{\text{cvx,2}}$ are instead better approximated by PSD matrices, as $-\lambda_n \lk\M\rk$ can be negligible.
  The proposition below describes the approximate low-rank structure of $\mathcal{E}_{\text{ncvx}}$ and $\mathcal{E}_{\text{cvx,1}}$.
\begin{proposition}\label{pro low-rank}
The admissible sets $\mathcal{E}_{\text{ncvx}}$ and $\mathcal{E}_{\text{cvx,1}}$ satisfy:
\begin{itemize}
\item[$\mathrm{(a)}$] For all $\M\in\mathcal{E}_{\text{ncvx}}$, we have $\lV\M\rV_{*}\le \sqrt{2}\lV\M\rV_{F}$;
\item[$\mathrm{(b)}$] For all $\M\in\mathcal{E}_{\text{cvx,1}}$, we have $ \lV\M\rV_{*}\le 3\lV\M\rV_{F}$.
\end{itemize}
\end{proposition}
\begin{proof}
See Appendix \ref{proof:app low-rank}.
\end{proof}
\noindent Therefore, the analysis of $\mathcal{E}_{\text{cvx,1}}$ can still be carried out in a manner analogous to that of $\mathcal{E}_{\text{ncvx}}$, based on the similarity in their approximate low-rank structures.
In contrast, for $\mathcal{E}_{\text{cvx,2}}$, we can exploit its approximate PSD property to facilitate the analysis. 
Thus, we can take into account the following transformed conditions with respect to the nuclear norm ($\lV\,\cdot\,\rV_{*}$):
\begin{itemize}
\item the \textbf{\textit{Sampling Lower Bound Condition}} (\textit{SLBC}) with respect to the nuclear norm ($\lV\,\cdot\,\rV_{*}$) is that, there exists a positive constant $\widetilde{\alpha}$ such that
 \begin{eqnarray}\label{non_0}
		\sum_{k=1}^{m}\lv\lg\pp_k\pp_k^*,\M\rg\rv^2\ge \widetilde{\alpha} \lV\M\rV^{2}_{*}, \quad\forall \ \M\in\mathcal{E};
				\end{eqnarray}
\item the \textbf{\textit{Noise Upper Bound Condition}} (\textit{NUBC}) with respect to the nuclear norm ($\lV\,\cdot\,\rV_{*}$) is that, there exists a positive constant  $\widetilde{\beta}$ such that
 \begin{eqnarray}\label{noise_0}
		\lv\sum_{k=1}^{m}\xi_k\lg\pp_k\pp_k^*,\M\rg \rv\le  \widetilde{\beta} \lV\M\rV_{*}, \quad\forall \ \M\in\mathcal{E}.
				\end{eqnarray}
\end{itemize}

Therefore, if $\{\pp_k\}_{k=1}^{m}$ are sampling vectors for which both (\ref{non}) and (\ref{noise}) hold when restricted to $\mathcal{E}_{\text{cvx,1}}$ 
 and if $\Z_{\star}-\x\x^*$ falls into $\mathcal{E}_{\text{cvx,1}}$, then conditioned on that event, we have
 \begin{eqnarray*}
 \lV \pmb{Z}_{\star} - \x\x^{*} \rV_{F} \le \dfrac{2\beta}{\alpha}.
 \end{eqnarray*}
 Similarly,  if $\{\pp_k\}_{k=1}^{m}$ are sampling vectors for which both (\ref{non_0}) and (\ref{noise_0}) hold when restricted to $\mathcal{E}_{\text{cvx,2}}$ 
 and if $\Z_{\star}-\x\x^*$ falls into $\mathcal{E}_{\text{cvx,2}}$, then we obtain
\begin{eqnarray*}
\lV \pmb{Z}_{\star} - \x\x^{*} \rV_{*} \le \dfrac{2\widetilde{\beta}}{\widetilde{\alpha}}.
\end{eqnarray*}
    Since $\Z_{\star}-\x\x^*$ lies in either $\mathcal{E}_{\text{cvx,1}}$ or $\mathcal{E}_{\text{cvx,2}}$ and $\lV\,\cdot\,\rV_{F}\le\lV\,\cdot\,\rV_{*}$, the estimation error for the CVX-LS estimator (\ref{model2}) satisfies that
 \begin{eqnarray}\label{error2}
		\lV\pmb{Z}_{\star}-\x\x^{*}\rV_{F}\le2\max\left\{\frac{\beta}{\alpha},\frac{\widetilde{\beta}}{\widetilde{\alpha}}\right\}.
				\end{eqnarray}
To obtain a $\textbf{dist}(\z_{\star},\x)$-type estimation bound, we construct $\z_{\star}$ as defined earlier in \eqref{solution}.
We provide the following distance inequality, whose proof is based on the perturbation theory and the $\sin \theta$ theorem; see Corollary 4 in \cite{demanet2014stable} or Lemma A.2 in \cite{iwen2020phase} for the detailed arguments. 
Hence, the details are omitted here.

     \begin{proposition}[\cite{demanet2014stable,iwen2020phase}]\label{dis2}
     Let $\z_{\star} = \sqrt{\lambda_{1}\lk\Z_{\star}\rk} \xu_1$, where $\lambda_{1}\lk \Z_{\star}\rk$ denotes the largest eigenvalue of $\Z_{\star}$, and $\xu_1$ is its corresponding eigenvector.
     If $\lV\Z_{\star}-\x\x^*\rV_F\le\eta\lV\x\rV_2^2$, then
 \begin{equation*}\label{f122}
\textbf{dist}\lk\z_{\star},\x\rk\le\lk1+2\sqrt{2}\rk\eta\lV\x\rV_2.
 \end{equation*}
\end{proposition}
\noindent As a consequence of \eqref{error2} and Proposition \ref{f122}, setting $\eta=2\max\left\{\frac{\beta}{\alpha},\frac{\widetilde{\beta}}{\widetilde{\alpha}}\right\}/\lV\x\rV_2^2$, we obtain the following error bound for the CVX-LS estimator \eqref{model2}:
\begin{eqnarray}\label{error3}
 \textbf{dist}\lk\z_{\star},\x\rk\le  \frac{2+4\sqrt{2}}{\lV\x\rV_2} \max\left\{\frac{\beta}{\alpha},\frac{\widetilde{\beta}}{\widetilde{\alpha}}\right\}.
  \end{eqnarray}

\section{Multiplier Inequalities}\label{ineq}

To obtain upper bounds for the parameters $\beta$ and $\widetilde{\beta}$ in Section~\ref{Architecture}, which satisfy the \textbf{\textit{Noise Upper Bound Condition}} (\textit{NUBC}) over various admissible sets, we employ a powerful analytical tool: the multiplier inequalities.
The main results of this section establish bounds for two different classes of multipliers—sub-exponential and heavy-tailed multipliers.
In particular, Poisson noise, which we analyze in detail later, will be shown to fall into both categories.

\begin{theorem}[Multiplier Inequalities]\label{mul inq}

Suppose that $\{\pp_{k}\}_{k=1}^{m}$ are independent copies of a random vector $\pp\in\C^n$ whose entries $\left\{\varphi_j\right\}_{j=1}^{n}$ are i.i.d., mean 0, variance 1, and $K$-sub-Gaussian, and $\{\xi_{k}\}_{k=1}^{m}$ are independent copies of a random variable $\xi$,  but $\xi$ need not be independent of $\pp$.
				\begin{itemize}
		\item[$\mathrm{(a)}$] If $\xi$ is sub-exponential, then there exist positive constants $c_1,C_1,L$ dependent only on $K$ such that when provided $m\ge L n$,
		 with probability at least $1-2\exp\lk-c_1n\rk$,
     \begin{eqnarray}
	\lV\frac{1}{\sqrt{m}}\sum_{k=1}^{m}\left(\xi_k\pp_k\pp_{k}^{*}-\E\xi\pp\pp^{*}\right)\rV_{op}
	\le C_1 \lV\xi\rV_{\psi_1}\sqrt{n};
				\end{eqnarray}

		\item[$\mathrm{(b)}$] If $\xi\in L_q$ for some $q>2$, 
		then there exist positive constants $c_2,c_3,C_2,\widetilde{L}$ dependent only on $K$ and $q$ such that when provided $m\ge \widetilde{L} n$, with probability at least $1-c_2m^{-\lk q/2-1\rk}\log^{q} m-2\exp\lk-c_3n\rk$,
     \begin{eqnarray}
	\lV\frac{1}{\sqrt{m}}\sum_{k=1}^{m}\lk\xi_k\pp_k\pp_{k}^{*}-\E\xi\pp\pp^{*}\rk\rV_{op}\le C_2 \lV\xi\rV_{L_q}\sqrt{n}.
				\end{eqnarray}
		\end{itemize}
\end{theorem}

\begin{remark}\label{remark1}
We make the following remarks on Theorem~\ref{mul inq}.
    \begin{enumerate}
\item
The results also extend to asymmetric sampling of the form $\left\{\pmb{a}_k\pmb{b}^*_k\right\}_{k=1}^{m}$, where $\left\{\pmb{a}_k\right\}_{k=1}^{m}$ and $\left\{\pmb{b}_k\right\}_{k=1}^{m}$ are all independent copies of a random vector $\pp\in\C^n$ whose entries $\left\{\varphi_j\right\}_{j=1}^{n}$ are i.i.d., mean 0, variance 1, and $K$-sub-Gaussian.
\item The proof of Theorem~\ref{mul inq} builds on deep results by Mendelson~\cite{mendelson2016upper} on generic chaining bounds for multiplier processes (see Section~\ref{subsecion:mul pro}), we present the detailed proof of Theorem~\ref{mul inq} in Section~\ref{subsecion:proof for mul ineq}.
\end{enumerate}
\end{remark}

\subsection{Upper Bounds for \textit{NUBC}}

Building on the multiplier inequalities in Theorem~\ref{mul inq}, we can derive upper bounds for the \textit{NUBC} across various admissible sets in the presence of sub-exponential and heavy-tailed multipliers.
We begin by considering the case where the multiplier follows a sub-exponential distribution.

\begin{corollary}\label{mul inq for subexponential}
Suppose that $\left\{\pp_{k}\right\}_{k=1}^{m}$ and $\left\{\xi_{k}\right\}_{k=1}^{m}$ satisfy the conditions in Theorem \ref{mul inq}.
If $\xi$ is sub-exponential, then there exist positive constants $c,C_1,C_2,L$ dependent only on $K$ such that, when provided $m\ge L n$, with probability at least $1-2\exp\lk-cn\rk$, the following inequalities hold:
				\begin{itemize}
		\item[$\mathrm{(a)}$] 
		For all $\M\in \mathcal{E}_{\text{ncvx}}$ or all $\M\in \mathcal{E}_{\text{cvx,1}} $, one has
  \begin{eqnarray*}
  \lv\left\langle\sum_{k=1}^{m}\lk\xi_k\pp_k\pp_{k}^{*}-\E\xi\pp\pp^{*}\rk,\M\right\rangle\rv
	\le C_1 \lV\xi\rV_{\psi_1}\sqrt{mn} \lV\M\rV_{F};
    \end{eqnarray*}
	
		\item[$\mathrm{(b)}$] 
		  For all $\M\in \mathcal{E}_{\text{cvx,2}} $, one has
        \begin{eqnarray*}
	\lv\left\langle\sum_{k=1}^{m}\lk\xi_k\pp_k\pp_{k}^{*}-\E\xi\pp\pp^{*}\rk,\M\right\rangle\rv
	\le C_2 \lV\xi\rV_{\psi_1}\sqrt{mn}\lV\M\rV_{*}.
    \end{eqnarray*}

		\end{itemize}
\end{corollary}

Similarly, we can derive upper bounds for the \textit{NUBC} in the case of a heavy-tailed multiplier.

\begin{corollary}\label{mul inq for heavy}
Suppose that $\left\{\pp_{k}\right\}_{k=1}^{m}$ and $\left\{\xi_{k}\right\}_{k=1}^{m}$ satisfy the conditions in Theorem \ref{mul inq}.
If $\xi\in L_q$ for some $q>2$, then there exist positive constants $c_1,c_2,C_1,C_2,L$ dependent only on $K$ and $q$ such that, when provided $m\ge L n$, with probability at least $1-c_1m^{-\lk q/2-1\rk}\log^{q} m-2\exp\lk-c_2n\rk$,
the following inequalities hold:
				\begin{itemize}
		\item[$\mathrm{(a)}$] 
		For all $\M\in \mathcal{E}_{\text{ncvx}}$ or all $\M\in \mathcal{E}_{\text{cvx,1}} $, one has
  \begin{eqnarray*}
  \lv\left\langle\sum_{k=1}^{m}\lk\xi_k\pp_k\pp_{k}^{*}-\E\xi\pp\pp^{*}\rk,\M\right\rangle\rv
	\le C_1 \lV\xi\rV_{L_q}\sqrt{mn} \lV\M\rV_{F};
    \end{eqnarray*}
	
		\item[$\mathrm{(b)}$] 
		  For all $\M\in \mathcal{E}_{\text{cvx,2}} $, one has
        \begin{eqnarray*}
	\lv\left\langle\sum_{k=1}^{m}\lk\xi_k\pp_k\pp_{k}^{*}-\E\xi\pp\pp^{*}\rk,\M\right\rangle\rv
	\le C_2 \lV\xi\rV_{L_q}\sqrt{mn}\lV\M\rV_{*}.
    \end{eqnarray*}

		\end{itemize}
\end{corollary}

We now turn to the proofs of these two corollaries.

\begin{proof}[Proof of Corollary~\ref{mul inq for subexponential} and Corollary~\ref{mul inq for heavy}]
We begin by proving \textbf{Part}~$\mathrm{(a)}$ of corollary~\ref{mul inq for subexponential}.
	For all $\M \in \mathcal{E}_{\text{ncvx}}$, we have 
\begin{eqnarray*}
\begin{aligned}
	    \lv  \left\langle\sum_{k=1}^{m}\lk\xi_k\pp_k\pp_{k}^{*}-\E\xi\pp\pp^{*}\rk,\M\right\rangle\rv
		&\le\lV\sum_{k=1}^{m}\xi_k\pp_k\pp_k^*-m\E\xi\pp\pp^*\rV_{op}\lV\M\rV_{*}\\
		&\le \sqrt{2}\lV\sum_{k=1}^{m}\xi_k\pp_k\pp_k^*-m\E\xi\pp\pp^*\rV_{op}\lV\M\rV_{F}\\
		&\lesssim_{K} \lV\xi\rV_{\psi_1}\sqrt{mn}\lV\M\rV_{F}.
		\end{aligned}
		\end{eqnarray*}
Here, the first line follows from the dual norm inequality.
In the second line, we have used \textbf{Part}~$\mathrm{(a)}$ of Proposition~\ref{pro low-rank}. 
In the third line, we have used  \textbf{Part}~$\mathrm{(a)}$ of Theorem~\ref{mul inq}, which holds with probability at least 
 $1- \mathcal{O}\lk e^{-cn}\rk$ when $m\gtrsim_{K} n$.
 For $\M \in \mathcal{E}_{\text{cvx,1}}$, the argument proceeds analogously, except that we now invoke \textbf{Part}~$\mathrm{(b)}$ of Proposition~\ref{pro low-rank}.

The proof of \textbf{Part} $\mathrm{(b)}$ of Corollary~\ref{mul inq for subexponential} follows directly from \textbf{Part} ~$\mathrm{(a)}$ of Theorem~\ref{mul inq}, since for all $\M \in \mathcal{E}_{\text{cvx,2}}$, we have 
\begin{eqnarray*}
\begin{aligned}
		 \lv \left\langle\sum_{k=1}^{m}\lk\xi_k\pp_k\pp_{k}^{*}-\E\xi\pp\pp^{*}\rk,\M\right\rangle\rv
		&\le\lV\sum_{k=1}^{m}\xi_k\pp_k\pp_k^*-m\E\xi\pp\pp^*\rV_{op}\lV\M\rV_{*}\\
		&\lesssim_{K} \lV\xi\rV_{\psi_1}\sqrt{mn}\lV\M\rV_{*}.
		\end{aligned}
		\end{eqnarray*}

The proof of Corollary~\ref{mul inq for heavy} closely follows that of Corollary~\ref{mul inq for subexponential}, with the only difference being the use of \textbf{Part}~$\mathrm{(b)}$ of Theorem~\ref{mul inq}.
As a result, the established probability bound is no longer exponentially decaying.
\end{proof}

\subsection{Multiplier Processes}\label{subsecion:mul pro}

To prove the multiplier inequalities in Theorem \ref{mul inq}, we employ the multiplier processes developed by Mendelson in \cite{mendelson2016upper,mendelson2017multiplier}.
Let $\lk\Omega,\mu\rk$ be an arbitrary probability space in which case $\mathcal{F}$ is a class of real-valued functions on $\Omega$,
$X$ be a random variable on $\Omega$ and $X_1,\cdots,X_m$ be independent copies of $X$.
Let $\xi$ be a random variable that need not be independent of $X$ and  $\lk X_k,\xi_k\rk_{k=1}^m$ to be $m$ independent copies of $\lk X,\xi\rk$, we define the centered multiplier processes indexed by $\mathcal{F}$ as
\begin{equation}\label{F}
\sup_{f \in \mathcal{F}} \lv \frac{1}{\sqrt{m}} \sum_{k=1}^m \lk\xi_k f\lk X_k\rk - \E \xi f\lk X\rk\rk  \rv.
\end{equation}

To estimate multiplier processes \eqref{F} that are based on some natural complexity parameter of the underlying class $\mathcal{F}$, which captures its geometric structure, one may rely on Talagrand’s $\gamma_{\alpha}$-functionals and their variants.
For a more detailed description of Talagrand's $\gamma_{\alpha}$-functionals, we refer readers to the seminal work~\cite{talagrand2014upper}.

\begin{definition}
For a metric space $\lk \mT,d\rk$, an admissible sequence of $\mT$ is a collection of
subsets $\mT_s\subset \mT$, whose cardinality satisfies for every $s\ge1, \lv\mT_s\rv\le2^{2^{s}}$ and $\lv\mT_{0}\rv=1$. 
For $\alpha\ge1,s_0\ge0$, define the $\gamma_{s_0,\alpha}$-functional by
\begin{equation*}
		\gamma_{s_0,\alpha}\lk\mT,d\rk=\inf_{\mT}\sup_{t\in\mT}\sum_{s\ge s_0}^{\infty}2^{s/\alpha}d\lk t,\mT_s\rk,
	\end{equation*}
where the infimum is taken all admissible sequences of $\mT$ and $d\lk t,\mT_s\rk$ denotes the distance from $t$ to set $\mT_s$. 
When $s_0=0$, we shall write $\gamma_{\alpha}\lk\mT,d\rk$ instead of $\gamma_{s_0,\alpha}\lk\mT,d\rk$.
Obviously, one has $\gamma_{s_0,\alpha}\lk\mT,d\rk\le \gamma_{\alpha}\lk\mT,d\rk$.
\end{definition}

The $\gamma_{2}$-functional effectively characterizes \eqref{F} when $\mathcal{F}\subset L_2$. 
However, once $\mF$ extends beyond this regime, the $\gamma_{2}$-functional along with its variant $\gamma_{s_0,2}$-functional, is no longer sufficient.
This motivates the introduction of its related functionals.
Following the language in \cite{mendelson2016upper}, we provide the following definition.

\begin{definition}\label{Tala}
For a random variable $Z$ and $p \ge 1$, set
\begin{equation*}
\lV Z\rV_{\lk p\rk} = \sup_{1 \le q \le p} \frac{\lV Z\rV_{L_q}}{\sqrt{q}}.
\end{equation*}
Given a class of functions $\mF$, $u \ge 1$ and $s_0 \ge 0$, put
\begin{equation} \label{gamma}
{\Lambda}_{s_0,u}\lk \mF\rk = \inf \sup_{f \in \mF} \sum_{s \ge s_0} 2^{s/2} \lV f - \pi_s f\rV_{\lk u^2 2^s\rk},
\end{equation}
where the infimum is taken with respect to all sequences $\lk\mF_s\rk_{s \ge 0}$ of subsets of $\mF$, and of cardinality $\lv \mF_s\rv \le 2^{2^s}$. 
$\pi_s f$ is the nearest point in $\mF_s$ to $f$ with respect to the $\lV\, \cdot\, \rV_{ \lk u^2 2^s\rk}$ norm.
Finally, let
\begin{equation*}
\widetilde{\Lambda}_{s_0,u}\lk\mF\rk={\Lambda}_{s_0,u}\lk\mF\rk+2^{s_0/2}\sup_{f \in \mF} \lV\pi_{s_0}f\rV_{\lk u^2 2^{s_0}\rk}.
\end{equation*}
\end{definition}

We provide additional explanations and perspectives on the above definition.
$\lV Z\rV_{\lk p\rk} $ measures the local sub-Gaussian behavior of random variable $Z$, which means that it  takes into account the growth of $Z$'s moments up to a fixed level $p$.
In comparison, the $\lV\,\cdot\,\rV_{\psi_2}$ norm of $Z$ captures its behavior across arbitrary moment orders,
\begin{equation*}
\lV Z \rV_{\psi_2} \asymp \sup_{q \ge 2} \frac{\lV Z\rV _{L_q}}{\sqrt{q}}.
\end{equation*}
This implies  that for any $2 \le p <\infty$, $\lV Z\rV_{\lk p\rk}\le \lV Z \rV_{\psi_2}$.
In fact, for any $u \ge 1$ and $s \ge s_0$, by definition of $\Lambda_{s_0,u}\lk\mF\rk$, one has
\begin{equation*}
{\Lambda}_{s_0,u}\lk\mF\rk \lesssim \inf \sup_{f \in \mF} \sum_{s \ge s_0} 2^{s/2} \lV f - \pi_s f\rV_{\psi_2},
\end{equation*}
and thus $\widetilde{\Lambda}_{0,u}\lk\mF\rk\lesssim \gamma_2\lk \mF,\psi_2\rk$.
Hence, we may rely on $\widetilde{\Lambda}_{s_0,u}(\mF)$ to yield satisfactory bounds in the case where $\mF$ does not belong to $L_2$.
We now provide the following estimates from \cite{mendelson2016upper}, which state that $\widetilde{{\Lambda}}_{s_0,u}\lk\mF\rk$ can be used to bound multiplier processes in a relatively general situation.

\begin{lemma}[\cite{mendelson2016upper}]\label{multiplier}
Let $\{X_{k}\}_{k=1}^{m}$ be independent copies of $X$ and $\{\xi_{k}\}_{k=1}^{m}$ be independent copies of  $\xi$,
and $\xi$ need not be independent of $X$.
				\begin{itemize}
		\item[$\mathrm{(a)}$] Let $\xi$ be sub-exponential.
		There are some absolute constants $c_0,c_1,c_2,c_3$ and $C$ for which the following holds.
		 Fix an integer $s_0 \ge 0$ and $w,u>c_0$. 
		 Then with probability at least $1-2\exp\lk-c_1m w^2\rk-2\exp\lk-c_2u^2 2^{s_0}\rk$,
\begin{eqnarray*}
\sup_{f \in \mF} \lv\frac{1}{\sqrt{m}} \sum_{k=1}^m \lk \xi_k f\lk X_k\rk - \E \xi f \lk X\rk\rk \rv \le C w u\lV \xi\rV_{\psi_1}\widetilde{\Lambda}_{s_0,c_3u}\lk\mF\rk;
\end{eqnarray*}

		\item[$\mathrm{(b)}$] Let $\xi\in L_q$ for some $q>2$.
		There are some positive constants $\widetilde{c_0},\tilde{c_1},\tilde{c_2},\tilde{c_3}$ and $\widetilde{C}$ that depend only on $q$ for which the following holds.
		 Fix an integer $s_0 \ge 0$ and $w,u>\widetilde{c_0}$. 
Then with probability at least $1-\tilde{c_1}w^{-q} m^{-\lk q/2-1\rk}\log^{q} m-2\exp\lk-\tilde{c_2}u^2 2^{s_0}\rk$,
\begin{eqnarray*}
\sup_{f \in \mF} \lv\frac{1}{\sqrt{m}} \sum_{k=1}^m \lk \xi_k f\lk X_k\rk - \E \xi f \lk X\rk\rk \rv \le \widetilde{C} w u\lV \xi\rV_{L_q}\widetilde{\Lambda}_{s_0,\tilde{c_3}u}\lk\mF\rk.
\end{eqnarray*}
		\end{itemize}
\end{lemma}

\begin{remark}
\textbf{Part}~$\mathrm{(a)}$ of Lemma~\ref{multiplier} can be derived from the proof of Theorem 4.4 in \cite{mendelson2016upper}, which assumes $\xi$ to be sub-Gaussian. 
We found that with only minor adjustments, the result holds when $\xi$ is sub-exponential.
\textbf{Part}~$\mathrm{(b)}$ of Lemma~\ref{multiplier} follows from Theorem 1.9 in \cite{mendelson2016upper}.
\end{remark}
 
 \subsection{Proof of Theorem~\ref{mul inq}}\label{subsecion:proof for mul ineq}
 
To employ the multiplier processes in Lemma~\ref{multiplier}, we present the following lemma, which characterizes the geometric structure of the function class $\mathcal{F}$ in our setting.
 
 \begin{lemma}\label{geometry}
 For any $\M\in\mathcal{S}^n$, we have
  \begin{eqnarray}
 \lV\sum_{k=1}^{m}\left\langle\pp_k\pp_k^*-m\E\pp\pp^*,\M\right\rangle\rV_{L_q}\lesssim K^2\lk \sqrt{qm}\lV\M\rV_{F}+q\lV\M\rV_{op}\rk.
 \end{eqnarray}
 \end{lemma}
 
 \begin{proof}
 By Hanson-Wright inequality in \cite{rudelson2013hanson}, there exists universal constant $c>0$, such that
 for random variable 
 $$\sum_{k=1}^{m}\pp_k^*\M\pp_k=\begin{pmatrix}
 \pp_{1}^{*}& \cdots&\pp_{m}^{*} \\
\end{pmatrix}
\begin{pmatrix}
\M &   &   \\
	  & \ddots &   \\
	  &   & \M
\end{pmatrix}
\begin{pmatrix}
 \pp_{1}  \\
	 \vdots    \\
	  \pp_{m}
	  \end{pmatrix},$$
	   for any $t>0$, we have, 
\begin{eqnarray*}
\begin{aligned}
\mathbb{P}&\lk\lv\sum_{k=1}^{m}\pp_k^*\M\pp_k-m\E\pp^*\M\pp\rv > t \rk \\
&\quad \quad \quad \quad\quad \quad \le 2 \exp\lk-c\min\left\{\frac{t^2}{K^4m\lV \M\rV^{2}_{F}},\frac{t}{K^2\lV \M\rV_{op}}\right\}\rk.
\end{aligned}
\end{eqnarray*}
Then, we can obtain
\begin{equation}\label{E of hanson}
\begin{aligned}
\mathbb{E} \left| \sum_{k=1}^{m} \pp_k^* \pmb{M} \pp_k - m\, \mathbb{E}  \pp^* \pmb{M} \pp  \right|^q
&= \int_{0}^{\infty} q t^{q-1} \, \mathbb{P} \left( \left| \sum_{k=1}^{m} \pp_k^* \pmb{M} \pp_k - m\, \mathbb{E}  \pp^* \pmb{M} \pp  \right| > t \right) dt \\
&\le 2q \int_{0}^{\infty} t^{q-1} \exp\left( -c \frac{t^2}{K^4 m \|\pmb{M}\|_F^2} \right) dt \\
&\quad + 2q \int_{0}^{\infty} t^{q-1} \exp\left( -c \frac{t}{K^2 \|\pmb{M}\|_{op}} \right) dt \\
&= 2q\, K^{2q} m^{q/2} \|\pmb{M}\|_F^q \int_{0}^{\infty} x^{q-1} \exp(-c x^2) dx \\
&\quad + 2q\, K^{2q} \|\pmb{M}\|_{\mathrm{op}}^q \int_{0}^{\infty} x^{q-1} \exp(-c x) dx \\
&= 2q\, \Gamma\left( \frac{q}{2} \right) c^{q/2 - 1} K^{2q} m^{q/2} \|\pmb{M}\|_F^q \\
&\quad + 2q\, \Gamma(q)\, c^{q - 1} K^{2q} \|\pmb{M}\|_{op}^q.
\end{aligned}
\end{equation}
where $\Gamma\lk q\rk$ denotes the Gamma function. 
We outline a property of the Gamma function below.
Note that for any $q>0$,
 \begin{eqnarray}\label{gamma pro}
\Gamma\lk q+1\rk=\int_0^\infty \lk x^q e^{-\frac{x}{2}} \rk e^{-\frac{x}{2}} dx \le \lk 2q\rk^q e^{-q}\int_0^\infty e^{-\frac{x}{2}} dx = 2 \lk \frac{2q}{e} \rk^q,
	\end{eqnarray}
where we have used the fact that $x^qe^{-\frac{x}{2}}$ attains maximum at $x=2q$ as
 \begin{eqnarray*} 
 \frac{d}{dx} \lk x^qe^{-\frac{x}{2}} \rk = x^{q-1} e^{-\frac{x}{2}} \lk q-\frac{x}{2} \rk.	
 \end{eqnarray*}
 Thus when we substitute \eqref{gamma pro} into \eqref{E of hanson}, we obtain
 \begin{eqnarray}
 \begin{aligned}
 \lV\sum_{k=1}^{m}\left\langle\pp_k\pp_k^*-m\E\pp\pp^*,\M\right\rangle\rV_{L_q}&=\lk\E \lv  \sum_{k=1}^{m}\pp_k^*\M\pp_k-m\E\pp^*\M\pp\rv^{q}\rk^{1/q}\\
 &\lesssim K^2\lk \sqrt{qm}\lV\M\rV_{F}+q\lV\M\rV_{op}\rk.
 \end{aligned}
 \end{eqnarray}
 \end{proof}
 
Now, we are ready to proceed with the proof of Theorem~\ref{mul inq}.
We set $\Omega=\C^{n\times n}, X=\pp\pp^*$ and $\mathcal{F}=\left\{\lg\cdot,\M\rg:\M\in\mM\right\}$, where $\mM$ is a subset of $\mS^n$.
In our case later, we will take $\mM=\left\{\z\z^*:\z\in\mathbb{S}^{n-1}\right\}$.
By Lemma~\ref{multiplier}, it suffices to upper bound $\widetilde{\Lambda}_{s_0,u}\lk\mF\rk$ and invoke the probability bounds established therein.

By Lemma~\ref{geometry} and the definition of $\lV\,\cdot\,\rV_{\lk p\rk}$ norm, we have that
  \begin{eqnarray*}
  \begin{aligned}
 &\lV\left\langle\frac{1}{\sqrt{m}}\sum_{k=1}^{m}\lk\pp_k\pp_{k}^{*}-\E\pp\pp^{*}\rk,\M\right\rangle\rV_{\lk p\rk}\\
 &\quad\quad\quad= \sup\limits_{1 \le q \le p} \frac{\lV\left\langle \frac{1}{\sqrt{m}}\sum_{k=1}^{m}\lk\pp_k\pp_{k}^{*}-\E\pp\pp^{*}\rk,\M\right\rangle\rV_{L_q} }{\sqrt{q}}\\
 &\quad\quad\quad\lesssim K^2\lk \lV\M\rV_{F}+\sqrt{\frac{p}{m}}\lV\M\rV_{op}\rk,
 \end{aligned}
 \end{eqnarray*}
 and thus
  \begin{eqnarray*}
 \lV\left\langle\frac{1}{\sqrt{m}}\sum_{k=1}^{m}\lk\pp_k\pp_{k}^{*}-\E\pp\pp^{*}\rk,\M\right\rangle\rV_{\lk u^2 2^s\rk}&\lesssim K^2\lk \lV\M\rV_{F}+\frac{u2^{s/2}}{\sqrt{m}}\lV\M\rV_{op}\rk.
 \end{eqnarray*}
 Hence, by the definition of ${\Lambda}_{s_0,u}\lk \mF\rk$-functional, we can obtain
 \begin{equation} \label{gamma3}
 \begin{aligned}
{\Lambda}_{s_0,u}\lk \mF\rk
&\lesssim  K^2 \inf \sup_{\M\in \mM} \lk \sum_{s\ge s_0}2^{s/2}\lV\M - \pi_s \lk\M\rk\rV_{F} + \sum_{s\ge s_0}\frac{u2^s}{\sqrt{m}}\lV\M - \pi_s \lk\M\rk\rV_{op} \rk\\
&\lesssim K^2\lk\gamma_{s_0,2} \lk \mM,\lV\,\cdot\,\rV_{F}\rk+\frac{u}{\sqrt{m}}\gamma_{s_0,1}\lk\mM,\lV\,\cdot\,\rV_{op}\rk\rk,
\end{aligned}
\end{equation}
and then
 \begin{equation}\label{gamma4}
 \begin{aligned}
\widetilde{{\Lambda}}_{s_0,u}\lk \mF\rk&\lesssim K^2\lk\gamma_{s_0,2} \lk \mM,\lV\,\cdot\,\rV_{F}\rk+2^{s_0/2}\sup_{ \mM} \|\pi_{s_0}\lk\M\rk\|_{F}\rk\\
&+K^2\frac{u}{\sqrt{m}}\lk \gamma_{s_0,1}\lk \mM,\lV\,\cdot\,\rV_{op}\rk+ 2^{s_0}\sup_{\mM} \|\pi_{s_0}\lk\M\rk\|_{op}\rk.\\
\end{aligned}
\end{equation}

We now turn to our specific case, where $\mM=\left\{\z\z^*:\z\in\mathbb{S}^{n-1}\right\}$.
Thus  
\begin{equation*}
\sup_{\mM} \|\pi_{s_0}\lk\M\rk\|_{op}= \sup_{\mM} \|\pi_{s_0}\lk\M\rk\|_{F}=1.
\end{equation*}
By Lemma 3.1 in \cite{candes2011tight}, the  covering number $\mN\lk\mM, \lV\cdot\rV_{F},\epsilon \rk$ satisfies that
 \begin{eqnarray*}
 \mN\lk\mM, \lV\,\cdot\,\rV_{F},\epsilon \rk\le \lk\frac{9}{\epsilon}\rk^{2n+1}.
 \end{eqnarray*}
 Then by the Dudley integral (see, e.g., \cite[Theorem 11.17]{ledoux2013probability}), we have
\begin{equation*}
\begin{aligned}
\gamma_{s_0,2} \left( \mM,\lV\,\cdot\,\rV_{F} \right) 
&\le \gamma_{2} \left( \mM, \lV\,\cdot\,\rV_{F} \right) \\
&\lesssim \int_{0}^{1} \sqrt{ \log \mN\left( \mM, \lV\,\cdot\,\rV_{F}, \epsilon \right) } \, d\epsilon \\
&\le \int_{0}^{1} \sqrt{ \left(2n + 1\right) \cdot \log \left( \frac{9}{\epsilon} \right) } \, d\epsilon \lesssim \sqrt{n},
\end{aligned}
\end{equation*}
and
\begin{equation*}
\begin{aligned}
\gamma_{s_0,1} \left( \mM, \lV\,\cdot\,\rV_{op} \right)
&\le \gamma_{1} \left( \mM, \lV\,\cdot\,\rV_{op} \right)\le \gamma_{1} \left( \mM, \lV\,\cdot\,\rV_{F} \right) \\
&\lesssim \int_{0}^{1}  \log \mN\left( \mM, \lV\,\cdot\,\rV_{F}, \epsilon \right)  \, d\epsilon \\
&\lesssim \int_{0}^{1} \left(2n + 1\right) \cdot \log \left( \frac{9}{\epsilon} \right) \, d\epsilon \lesssim n.
\end{aligned}
\end{equation*}
Finally, we select $s_0$ sufficiently large such that $K^22^{s_0/2}\lesssim \sqrt{n}$ and $K^22^{s_0}\lesssim n$, and take $u$ and $w$ in Lemma~\ref{multiplier} to be of order 1, independent of other parameters.
With these choices and by ensuring $M \gtrsim_{K} n$, the proof is then complete.

\section{Small Ball Method and Lower Isometry Property}\label{small0}

The purpose of this section is to lower bound the parameters $\alpha$ and $\widetilde{\alpha}$ in Section \ref{Architecture} that satisfies the  \textbf{\textit{Sampling Lower Bound Condition}} (\textit{SLBC}) over different admissible sets.
We employ the small ball method and the lower isometry property to obtain lower bounds for these two parameters, respectively.

\subsection{Small Ball Method}\label{small_ball}

We present the following result, which establishes lower bounds for the \textit{SLBC} with respect to $\lV\,\cdot\,\rV_F$.

	\begin{lemma}\label{small ball1}
		Suppose that $\left\{\pp_{k}\right\}_{k=1}^{m}$ satisfy Assumption~\ref{sample}.
		There exist positive constants $L,c,C$, depending only on $K$ and $\mu$, such that if $m \ge L n$, the following holds with probability at least $1 - e^{-c m}$:   
for all $\M\in \mathcal{E}_{\text{ncvx}}$ or all $\M\in \mathcal{E}_{\text{cvx,1}}$, one has
\begin{equation*}
\sum_{k=1}^{m}\lv\lg\pp_k\pp_k^*,\M\rg\rv^2\ge C_1m \lV\M\rV^{2}_{F};
\end{equation*}
	\end{lemma}	

    \begin{remark}
    We make some remarks on Lemma~\ref{small ball1}.
    \begin{enumerate}
\item
Lemma~\ref{small ball1} provides lower bounds for the parameter $\alpha$ over admissible sets $\mathcal{E}_{\text{ncvx}}$ and $\mathcal{E}_{\text{cvx,1}}$, establishing that $\alpha\gtrsim_{K,\mu} m$ in both cases, i.e., up to a constant depending only on $K$ and $\mu$.
\item 
The result also holds for asymmetric sampling of the form $\left\{\pmb{a}_k\pmb{b}^*_k\right\}_{k=1}^{m}$, where $\left\{\pmb{a}_k\right\}_{k=1}^{m}$ and $\left\{\pmb{b}_k\right\}_{k=1}^{m}$ are formed from independent copies of $\pp\in\C^n$ satisfying the conditions in Remark \ref{remark1}.
\item A similar formulation of Lemma~\ref{small ball1} can be found in \cite[Lemma 3]{krahmer2020complex}, where it is proved for a different set and by an analysis different from ours, namely using the covering number analysis instead of our empirical chaos process approach (see Lemma~\ref{chaos1} below).
\end{enumerate}
\end{remark}

A standard and effective approach for establishing such lower bounds is the small ball method—a widely used probabilistic technique for deriving high-probability lower bounds on nonnegative empirical processes; see, e.g., \cite{mendelson2015learning,tropp2015convex, kueng2017low, krahmer2020complex, krahmer2021convex, chinot2022robustness,huang2025low}.

The proof relies on several auxiliary results.
We begin with the first, which states the small ball method \cite{mendelson2015learning,tropp2015convex} tailored to our setting.
For brevity, we omit its proof, which can be found in~\cite[Proposition 5.1]{tropp2015convex}.

\begin{proposition}[\cite{tropp2015convex}]\label{small ball}
Let matrix set $\mM\subset \mathcal{S}^n$ and $\left\{\pp_k\right\}_{k=1}^{m}$ be independent copies of a random vector $\pp$ in $\C^{n}$. 
For $u>0$, let the small ball function be
    \begin{eqnarray}\label{sbf}     
     \mathcal{Q}_{u}\lk \mM;\pp\pp^*\rk=\inf_{\M\in\mM}\mathbb{P}\lk\lv\pp\pp^*,\pmb{M}\rv\ge u\rk 
	         \end{eqnarray}
	         and the supremum of Rademacher empirical process be
	          \begin{eqnarray}	\label{Rade}    
	            \mathcal{W}_{m}\lk \mM;\pp\pp^*\rk
	            =\E\sup_{\pmb{M}\in\mM}\lv\frac{1}{\sqrt{m}}\sum_{k=1}^{m}\varepsilon_{k}\lg\pp_k\pp_k^*,\mM\rg\rv,
	      	         \end{eqnarray} 
where $\{\varepsilon_{k}\}_{k=1}^{m}$ is a Rademacher sequence independent of everything else.        
		         
	Then for any $u>0$ and $t>0$, with probability at least $1-\exp\lk-2t^{2}\rk$,
	\begin{eqnarray}
	\begin{aligned}
	     \inf_{\M\in\mM}
	&\lk\sum_{k=1}^{m}\lv\lg\pp_k\pp_k^*,\M\rg\rv^2\rk^{1/2}\\		
	 &\quad\quad\quad\ \ge u \sqrt{m} \mathcal{Q}_{2u}\lk \mM;\pp\pp^*\rk-2 \mathcal{W}_{m}\lk \mM;\pp\pp^*\rk-u t.
	 \end{aligned}
	    \end{eqnarray}  
\end{proposition}

To employ the preceding proposition, one should obtain a lower bound for the small ball function and an upper bound for the supremum of the Rademacher empirical process.
The following lemma provides the latter.
This result can be interpreted as a Rademacher-type empirical chaos process, generalizing Theorem 15.1.4 in \cite{talagrand2014upper}.

\begin{lemma}\label{chaos1} 
Let $\pp\in\C^n$ be a random vector whose entries $\left\{\varphi_j\right\}_{j=1}^{n}$ are i.i.d., mean 0, variance 1, and $K$-sub-Gaussian.
For any matrix set $\mM\subset\mathcal{S}^{n}$ that satisfies $\mM=-\mM$, we have 
         \begin{eqnarray}\label{rade0}
         \begin{aligned}
 	 \mathcal{W}_{m}&\lk \mM;\pp\pp^*\rk\\
	& \le C_1 K^2\lk\gamma_{2}\lk\mM,\lV\,\cdot\,\rV_{F}\rk +  \frac{\gamma_{1}\lk\mM,\lV\,\cdot\,\rV_{op}\rk }{\sqrt{m}}\rk+C_2\sup\limits_{\M\in\mM}\text{Tr}\lk\M\rk,
	 \end{aligned}
	  \end{eqnarray}
      where $C_1,C_2>0$ are absolute constants.
 \end{lemma}  

\begin{proof}
We have that 
\begin{equation}\label{99}
\begin{aligned}
\sqrt{m} \, \mathcal{W}_m(\mathcal{M}; \pp\pp^*) 
&= \mathbb{E} \sup_{\pmb{M} \in \mathcal{M}} \sum_{k=1}^{m} \varepsilon_k \left\langle \pp_k \pp_k^*, \pmb{M} \right\rangle \\
&\leq \mathbb{E}_{\varepsilon} \mathbb{E}_{\pp} \sup_{\pmb{M} \in \mM} \left\langle \sum_{k=1}^{m} \varepsilon_k \lk \pp_k \pp_k^* - \mathbb{E}_{\pp} \pp \pp^* \rk, \pmb{M} \right\rangle\\
&\quad + \mathbb{E}_{\varepsilon} \mathbb{E}_{\pp} \sup_{\pmb{M} \in \mM }\left\langle \sum_{k=1}^{m} \varepsilon_k \mathbb{E}_{\pp} \pp \pp^*, \pmb{M} \right\rangle \\
&\leq 2 \,\mathbb{E}_{\pp} \sup_{\pmb{M} \in \mM} \left\langle \sum_{k=1}^{m} \left(\pp_k \pp_k^* - \mathbb{E}_{\pp} \pp \pp^* \right), \pmb{M} \right\rangle \\
&\quad + \mathbb{E}_{\varepsilon} \sup_{\pmb{M} \in \mM} \left\langle \sum_{k=1}^{m} \varepsilon_k \pmb{I}_n, \pmb{M} \right\rangle
\end{aligned}
\end{equation}
The first line is due to $\mM=-\mM$.
In the second inequality, we have used Giné--Zinn symmetrization principle \cite[Lemma 6.4.2]{vershynin2018high} and $\E_{\pp}\pp\pp^{*}=\pmb{I}_n$. 
By adapting the proof of Theorem 15.1.4 in~\cite{talagrand2014upper} to the empirical setting and generalizing it to the sub-Gaussian case, we can obtain the following bound:
     \begin{equation}\label{chaos0}
\begin{aligned}
\mathbb{E}_{\pp} \sup_{\pmb{M} \in \mathcal{M}} 
\left\langle 
\sum_{k=1}^{m} \left( \pp_k \pp_k^* - \mathbb{E}_{\pp} \pp_k \pp_k^* \right),\, \pmb{M} 
\right\rangle
&\;\lesssim\; 
K^2 \sqrt{m} \, \gamma_2\left( \mathcal{M}, \|\cdot\|_F \right) \\
&\quad\quad+\, 
K^2 \, \gamma_1\left( \mathcal{M}, \|\cdot\|_{\mathrm{op}} \right).
\end{aligned}
\end{equation}
      For the second term on the last line of \eqref{99}, we have that
	     \begin{eqnarray}\label{epsilon}
  \begin{aligned}
 	  \E_{\varepsilon}\sup_{\M\in\mM}\lg\sum_{k=1}^{m}\varepsilon_{k}\pmb{I}_n,\M \rg
	  &=\E_{\varepsilon}\sup_{\M\in\mM}\sum_{k=1}^{m}\varepsilon_{k}\text{Tr}\lk\M\rk\\
	  &\le\E_{\varepsilon}\lv\sum_{k=1}^{m}\varepsilon_{k}\rv\sup_{\M\in\mM}\text{Tr}\lk\M\rk\\
	    &\lesssim \sqrt{m} \sup_{\M\in\mM}\text{Tr}\lk\M\rk.
	\end{aligned}
	  \end{eqnarray}
	In the last line, we have used $\E_{\varepsilon}\lv\sum\limits_{k=1}^{m}\varepsilon_{k}\rv\lesssim\sqrt{m}$.
	 Thus, by \eqref{chaos0} and \eqref{epsilon}, we have finished the proof.
     \end{proof}     
     
     \begin{remark}
     We make the following observations regarding Lemma~\ref{chaos1}.
        \begin{enumerate}
        \item Lemma~\ref{chaos1} can also be proved via the multiplier processes in Lemma~\ref{multiplier} with multiplier $\xi$ chosen as a Rademacher random variable, though we obtain it more directly from a classical result on empirical chaos process in \cite{talagrand2014upper}.
\item
     In \cite{maly2023robust}, Maly has proved that
     \begin{eqnarray}\label{K44}
 	 \mathcal{W}_m\lk\mathcal{M}; \pp\pp^*\rk \le C\lk\sqrt{\mathcal{R}_0}\gamma_{2}\lk\mM,\lV\,\cdot\,\rV_{F}\rk +  \frac{\gamma_{1}\lk\mM,\lV\,\cdot\,\rV_{op}\rk }{\sqrt{m}}\rk,
		  \end{eqnarray}
where the factor $\mathcal{R}_0$ is defined by \small{$\mathcal{R}_0:=\sup\limits_{\M\in\mM}\frac{\lV\M\rV^2_{*}}{\lV\M\rV^2_{F}}$} and $C>0$ is a constant dependent only on $K$.
This factor reduces the sharpness of the estimation of $\mathcal{W}_m\lk\mathcal{M}; \pp\pp^*\rk$ in many cases of interest.
For instance, if $\mM:=\left\{\M\in\mathcal{S}^n:\text{rank}\lk\M\rk\le r,\lV\M\rV_F=1\right\}$, then $\mathcal{R}_0=r$. By the Dudley integral together with the covering number bound in Lemma 3.1 of \cite{candes2011tight}, we bound that
\begin{eqnarray*}
\gamma_{2}\lk\mM,\lV\,\cdot\,\rV_{F}\rk\lesssim \sqrt{rn}\quad \text{and}\quad\gamma_{1}\lk\mM,\lV\,\cdot\,\rV_{op}\rk\lesssim rn.
\end{eqnarray*}
Consequently, \eqref{K44} is of order $r^{3/2}\sqrt{n}$, whereas \eqref{rade0} is only of order $\sqrt{rn}$ when $m \gtrsim_K rn$.
We can also provide a detailed comparison between \eqref{rade0} and \eqref{K44}, and observe that
     \begin{eqnarray}
       \begin{aligned}
\sqrt{\mathcal{R}_0}\cdot\gamma_{2}\lk\mM,\lV\,\cdot\,\rV_{F}\rk&=\sup\limits_{\M\in\mM}\frac{\lV\M\rV_{*}}{\lV\M\rV_{F}}\cdot\gamma_{2}\lk\mM,\lV\,\cdot\,\rV_{F}\rk\\
	 &\gtrsim\sup\limits_{\M\in\mM}\frac{\lV\M\rV_{*}}{\lV\M\rV_{F}}\cdot\text{diam}\lk\mM\rk
	 \ge \sup_{\X\in\mM}\text{Tr}\lk\M\rk.
	\end{aligned}
	  \end{eqnarray}
Since $\mathcal{R}_0\ge1$, our bound \eqref{rade0} is a substantial improvement over \eqref{K44}. 
    \end{enumerate}
     \end{remark}     

The next proposition provides a lower bound for the small ball function, obtained by refining the analysis in~\cite{krahmer2020complex}.
\begin{proposition}\label{small ball function}    
Assume that $\pp$  is a random vector satisfies the conditions in Assumption~\ref{sample}.
For any  matrix set $\mM\subset\mathbb{S}_F$, we have 
\begin{equation}\label{small1}
\mathcal{Q}_{u}\lk\mathcal{M};\, \pp\pp^{*} \rk 
\ge 
C_0 \frac{\min\left\{ \mu^2,\, 1 \right\}}{K^8 + 1},
\end{equation}
where $0<u\le\sqrt{\frac{\min\left\{\mu,1\right\}}{2}}$ and $C_0>0$ is an absolute constant.
 \end{proposition}  

 \begin{proof}
 See Appendix~\ref{proof:small ball function}.
 \end{proof}

We are now fully equipped to proceed with the proof of Lemma~\ref{small ball1}.

\subsubsection{Proof of Lemma~\ref{small ball1}}
In this subsection, we set $\mM:=\left\{\z\z^{*}:\z\in\mathbb{S}^{n-1}\right\}$.
By Lemma~\ref{chaos1},  we can obtain that
\begin{eqnarray}\label{up for W}
 \mathcal{W}_{m}\lk \mM;\pp\pp^*\rk
 \le C_1K^2\lk\sqrt{n}+\frac{n}{\sqrt{m}}\rk+C_2.
	\end{eqnarray}	
	Here, we have used $\gamma_{2}\lk\mM,\lV\,\cdot\,\rV_{F}\rk\lesssim \sqrt{n}$ and $\gamma_{1}\lk\mM,\lV\,\cdot\,\rV_{op}\rk \lesssim n$, as we have established in Section~\ref{mul inq} and 
	$\sup\limits_{\z\in\mathbb{S}^{n-1}}\text{Tr}\lk\z\z^*\rk=1$.
	Therefore, we can get
\begin{align}
\mathcal{W}_{m}\left(\mathcal{E}_{\text{ncvx}} \cap \mathbb{S}_F;\, \pp\pp^* \right) 
&\le \mathbb{E} \left\| \frac{1}{\sqrt{m}} \sum_{k=1}^{m} \varepsilon_k \pp_k \pp_k^* \right\|_{\mathrm{op}}  \|\pmb{M}\|_* \notag \\
&\le \sqrt{2} \, \mathcal{W}_m\left( \mathcal{M};\, \pp\pp^* \right) \notag \\
&\le \sqrt{2}C_1K^2\lk\sqrt{n}+\frac{n}{\sqrt{m}}\rk+2\sqrt{2}C_2.\label{eq:rademacher_bound}
\end{align}
In the second line we have used \textbf{Part}~$(\mathrm{a})$ of Proposition~\ref{pro low-rank}.

Now we set {\small$u=\frac{1}{2}\sqrt{\frac{\min\left\{\mu,1\right\}}{2}},t=\frac{\sqrt{m}C_0\min\left\{ \mu^2,\, 1 \right\}}{2\lk K^8 + 1\rk}$}.
By Proposition~\ref{small ball function}, we have
\begin{eqnarray*}
\mathcal{Q}_{2u}\lk \mathcal{E}_{\text{ncvx}} \cap \mathbb{S}_F;\pp\pp^*\rk	\ge C_0 \cdot \frac{\min\left\{ \mu^2,\, 1 \right\}}{K^8 + 1}.
\end{eqnarray*}   
Then, by Proposition~\ref{small ball}, with probability at least $1- e^{-cm}$, where {\small$c=\frac{C_0^2\min\left\{\mu^4,1\right\}}{2\lk K^8+1\rk^2}$}, we obtain for all $\M\in\mathcal{E}_{\text{ncvx}}$,
 \begin{eqnarray}
\sum_{k=1}^{m}\lv\lg\pp_k\pp_k^*,\M\rg\rv^2
       \ge \widetilde{C} m \frac{\min\left\{\mu^{6},1\right\}}{K^{16}+1}\lV\M\rV_{F}^{2},
	\end{eqnarray}	
	provided that $m\ge L n$ for some sufficiently large constant $L>0$ depending only on $K$ and $\mu$. 
    
	We can establish the similar result for $\mathcal{E}_{\text{cvx,1}}$, where the difference lies in bounding $\mathcal{W}_{m}\left( \mathcal{E}_{\text{cvx,1}} \cap \mathbb{S}_F;\, \pp\pp^* \right)$ using \textbf{Part}~$(\mathrm{b})$ of Proposition~\ref{pro low-rank}.

\subsection{Lower Isometry Property}\label{ell_1}

To identify the parameter $\widetilde{\alpha}$ in Section~\ref{Architecture} that satisfies the \textit{SLBC} with respect to $\lV\,\cdot\,\rV_{*}$, we follow the idea of the lower isometry property in \cite{candes2013phaselift,krahmer2020complex}.

	\begin{lemma}\label{approximate}
		Suppose that $\{\pp_{k}\}_{k=1}^{m}$ are independent copies of a random vectors $\pp\in \C^n$, whose entries $\left\{\varphi_j\right\}_{j=1}^{n}$ are i.i.d., mean 0, variance 1, and $K$-sub-Gaussian.
		Then there exist positive constants $L,c$, depending only on $K$, such that if $m\ge L n$, the following holds with probability at least $1- 2 e^{-cm}$:
        for all $\M \in  \mathcal{E}_{\text{cvx,2}}$, we have
		\begin{eqnarray}
		 \sum_{k=1}^{m}\lv\lg\pp_k\pp_k^*,\M\rg\rv^2\ge \frac{1}{36} m \lV\M\rV^{2}_{*}.
		\end{eqnarray}
	\end{lemma}

\begin{remark}
  Some remarks on Lemma~\ref{approximate} are given as follows.
    \begin{enumerate}
\item
Lemma~\ref{approximate} provides a lower bound for the parameter $\widetilde{\alpha}$, indicating that $\widetilde{\alpha}\ge\frac{1}{36} m$.
\item Notably, the validity of Lemma~\ref{approximate} does not rely on the fourth-moment condition $\E\lk\lv\varphi\rv^4\rk=1+\mu$ with $\mu>0$, as stated in Assumption~\ref{sample}.
\item Lemma~\ref{approximate} can be deduced from \cite[Lemma 4]{krahmer2020complex}.
For completeness, we provide a full proof below.
\end{enumerate}
\end{remark}

\subsubsection{Proof of Lemma~\ref{approximate}}

By Theorem 4.6.1 in \cite{vershynin2018high}, for any $0\le\delta\le1$,
there exist positive constants $\widetilde{L}$ and $\tilde{c}$ dependent on $K$ and $\delta$, such that if $m\ge \widetilde{L} n$, with probability at least $1-2e^{-\tilde{c}m}$, the following holds:
   \begin{eqnarray}\label{l_2}
 \lk 1-\delta\rk\lV\z\rV^2_{2}\le  \frac{1}{m}\sum_{k=1}^{m}\lv\lg\pp_k,\z\rg\rv ^2\le \lk1+\delta\rk\lV\z\rV^2_{2},\quad \forall \z\in\C^n.
          \end{eqnarray}	   

We set $\M\in \mathcal{E}_{\text{cvx,2}}$ has eigenvalue decomposition $\M = \sum\limits_{i=1}^{n} \lambda_i \lk\M\rk \xu_i \xu^*_i.$
We obtain
\begin{align*}
\sum_{k=1}^{m} \bigl|\langle \pp_k \pp_k^*, \M \rangle \bigr|
&\;\ge\; \sum_{k=1}^{m} \langle \pp_k \pp_k^*, \M \rangle \\
&= \sum_{k=1}^{m} \Bigl\langle \pp_k \pp_k^*, \sum_{i=1}^{n} \lambda_i(\M)\, \xu_i \xu_i^* \Bigr\rangle \\
&= \sum_{i=1}^{n} \lambda_i(\M) \left( \sum_{k=1}^{m} \bigl|\langle \pp_k, \xu_i \rangle \bigr|^2 \right).
\end{align*}
Proposition~\ref{one} states that $\M$ has at most one negative eigenvalue. 
If all eigenvalues $\lambda_i \lk\M\rk$ are positive and if we choose $\delta =\frac{1}{6}$ in \eqref{l_2}, then, on the event that \eqref{l_2} occurs, we obtain
\begin{equation}\label{eq1}
\sum_{k=1}^{m}\lv\lg\pp_k\pp_k^*,\M\rg\rv \ge \frac{5}{6} m \sum_{i=1}^{n} \lambda_i \lk\M\rk =\frac{5}{6} m\lV\M\rV_{*}.
\end{equation}
If $\lambda_n \lk\M\rk <0$,
since the elements in $\mathcal{E}_{\text{cvx,2}}$ satisfy $-\lambda_n \lk\M\rk \le \frac{1}{2}  \sum\limits_{i=1}^{n-1} \lambda_i \lk\M\rk$, we obtain
\begin{equation}\label{eq2}
\begin{split}
\sum_{k=1}^{m}\lv\lg\pp_k\pp_k^*,\M\rg\rv
 &\ge  \frac{5}{6} m \sum_{i=1}^{n-1} \lambda_i \lk\M\rk  +\frac{7}{6}  m\lambda_n \lk\M\rk\\
 &\ge  \frac{1}{4} m \sum_{i=1}^{n-1} \lambda_i \lk\M\rk\ge\frac{1}{6} m \lV\M\rV_{*}.
\end{split}
\end{equation}
In the last inequality, we have used 
\begin{equation*}
\lV\M\rV_{*}= \sum_{i=1}^{n-1} \lambda_i \lk\M\rk -\lambda_n \lk\M\rk \le  \frac{3}{2}\sum_{i=1}^{n-1} \lambda_i \lk\M\rk.
\end{equation*}

Hence, by combining (\ref{eq1}) and (\ref{eq2}) with the Cauchy–Schwarz inequality, we deduce that
\begin{eqnarray*}
		 \sum_{k=1}^{m}\lv\lg\pp_k\pp_k^*,\M\rg\rv^2\ge \frac{1}{m}\lk\sum_{k=1}^{m}\lv\lg\pp_k\pp_k^*,\M\rg\rv\rk^2\ge\frac{1}{36}m\lV\M\rV^{2}_{*}.
		\end{eqnarray*}

\section{Proofs of Main Results}\label{proof of main}

We adhere to the framework outlined in Section~\ref{Architecture} to prove Theorem~\ref{thm poisson} and Theorem~\ref{low-dose} for Poisson model, and Theorem~\ref{thm heavy} for heavy-tailed model.
We will identify distinct parameters $\alpha,\beta,\widetilde{\alpha}$, and $\widetilde{\beta}$ for the respective admissible sets.

\subsection{Key Properties of Poisson Noise}

We first present the following proposition, which demonstrates that the behavior of Poisson noise can be approximated by sub-exponential noise.

\begin{proposition}\label{psi_1 of poisson}
Let random variable 
\begin{eqnarray*}
\xi=\text{Poisson}\lk \lv\lg\pp,\x\rg\rv^{2}\rk-\lv\lg\pp,\x\rg\rv^{2},
\end{eqnarray*}
where the entries $\left\{\varphi_j\right\}_{j=1}^{n}$ of random vector $\pp$ are independent, mean-zero and $K$-sub-Gaussian.
Then we have
 \begin{eqnarray*} 
\lV\xi\rV_{\psi_1}\lesssim \max\left\{1, K\lV\x\rV_{2} \right\}.
 \end{eqnarray*}
\end{proposition}

\begin{proof}
See Appendix \ref{psi_1}.
\end{proof}

Proposition \ref{psi_1 of poisson} provides an upper bound on the sub-exponential norm of $\xi$. 
However, in the low-energy regime where $\lV\x\rV_2\ll1/K$, we have $\lV\xi\rV_{\psi_1}\gtrsim 1$, which prevents the Poisson model analysis from capturing the decay in noise level as the signal energy diminishes.
Thus, we also present the following proposition, which characterizes the $L_{4}$ norm of $\xi$.
The underlying idea is that, in the low energy regime, the Poisson-type noise $\xi$ is more prone to deviating from its mean and thus becomes more susceptible to generating outliers, which makes it reasonable to model it as heavy-tailed noise.

\begin{proposition}\label{L_q of poisson}
Let random variable 
\begin{eqnarray*}
\xi=\text{Poisson}\lk \lv\lg\pp,\x\rg\rv^{2}\rk-\lv\lg\pp,\x\rg\rv^{2},
\end{eqnarray*}
where the entries $\left\{\varphi_j\right\}_{j=1}^{n}$ of random vector $\pp$ are independent, mean-zero and $K$-sub-Gaussian.
Then we have
 \begin{eqnarray*}
\lV\xi\rV_{L_{4}}\lesssim\max\left\{\lk K\lV\x\rV_2\rk^{1/2},K\lV\x\rV_2\right\}.
 \end{eqnarray*}
\end{proposition}

\begin{proof}
See Appendix \ref{L_q}.
\end{proof}

\subsection{Proof of Theorem \ref{thm poisson}}

We first focus on the analysis of the NCVX-LS estimator.
In this case, the admissible set is $\mathcal{E}_{\text{ncvx}}:=\left\{ \z\z^*-\x\x^*:\z,\x\in\C^{n}\right\}$.
By Lemma~\ref{small ball1}, for the \textit{SLBC} with respect to $\lV\,\cdot\,\rV_{F}$, we conclude that the parameter in~\eqref{non} satisfies $$\alpha\gtrsim_{K,\mu} m$$ 
with probability at least $1- \mathcal{O}\lk e^{-c_1m}\rk$, assuming $m\gtrsim_{K,\mu} n$.
		By \textbf{Part}~$\mathrm{(a)}$ of Corollary~\ref{mul inq for subexponential}, for the \textit{NUBC} with respect to $\lV\,\cdot\,\rV_{F}$, with probability at least $1- \mathcal{O}\lk e^{-c_2n}\rk$, one has for all $\M \in \mathcal{E}_{\text{ncvx}}$
\begin{eqnarray*}
\begin{aligned}  
	       \lv  \sum_{k=1}^{m}\xi_k\lg\pp_k\pp_k^*,\M\rg \rv
		&=\lv\sum_{k=1}^{m}\xi_k\lg\pp_k\pp_k^*-\E\,\xi\pp\pp^*,\M\rg\rv\\
		&\lesssim_{K} \lV\xi\rV_{\psi_1}\sqrt{mn}\lV\M\rV_{F}\\
		&\lesssim_{K} \max\left\{1, K\lV\x\rV_{2} \right\} \sqrt{mn}\lV\M\rV_{F},
		\end{aligned}
		\end{eqnarray*}
			provided $m\gtrsim_{K} n$.
Here, in the first line we have used $\E\,\xi\pp\pp^*=\pmb{0}$
and in the third line we have used Proposition \ref{psi_1 of poisson}.
Therefore, for the parameter in~\eqref{noise}, we have $$\beta\lesssim_{K} \max\left\{1, K\lV\x\rV_{2} \right\}\sqrt{mn}.$$
Then, by~\eqref{error1}, we can obtain the estimation error for the NCVX-LS estimator is
	                \begin{eqnarray}
		\textbf{dist}\lk\z_\star,\x\rk \lesssim_{K,\mu}\min\left\{  \max\left\{K, \frac{1}{\lV\x\rV_{2}} \right\}\cdot \sqrt{\frac{n}{m}},\,\max\left\{1, \sqrt{K\lV\x\rV_{2}} \right\}\cdot\lk\frac{n}{m}\rk^{1/4}\right\}.
				\end{eqnarray}
 
We next turn our attention to the CVX-LS estimator.
In this case, we take into account two admissible sets $\mathcal{E}_{\text{cvx,1}}$ and $\mathcal{E}_{\text{cvx,2}}$.
For $\mathcal{E}_{\text{cvx,1}}$, our argument follows the NCVX-LS estimator, and therefore we have 
\begin{eqnarray*}
\alpha\gtrsim_{K,\mu} m\quad\text{and}\quad\beta\lesssim_{K} \max\left\{1, K\lV\x\rV_{2} \right\} \sqrt{mn}.
\end{eqnarray*}
We next analyze $\mathcal{E}_{\text{cvx,2}}$.
By Lemma~\ref{approximate}, for the \textit{SLBC} with respect to $\lV\,\cdot\,\rV_{*}$,
we obtain that the parameter in \eqref{non_0} satisfies $$\widetilde{\alpha}\ge\frac{1}{36} m$$ with probability at least $1- 2 e^{-c_3m}$, provided $m\gtrsim_{K}n$.
		By \textbf{Part}~$\mathrm{(b)}$ of Corollary~\ref{mul inq for subexponential} and Proposition \ref{psi_1 of poisson}, for the \textit{NUBC} with respect to $\lV\,\cdot\,\rV_{*}$, with probability at least $1- \mathcal{O}\lk e^{-c_4n}\rk$, one has for all $\M \in \mathcal{E}_{\text{cvx,2}}$				
\begin{eqnarray*}
\begin{aligned}
		\lv \sum_{k=1}^{m}\xi_k\lg\pp_k\pp_k^*,\M\rg \rv
		&\lesssim_{K} \lV\xi\rV_{\psi_1}\sqrt{mn}\lV\M\rV_{*}\\
		&\lesssim_{K}  \max\left\{1, K\lV\x\rV_{2} \right\} \sqrt{mn}\lV\M\rV_{*},
		\end{aligned}
		\end{eqnarray*}
		provided $m\gtrsim_{K} n$.
Thus, for the parameter in~\eqref{noise_0} we have 
\begin{eqnarray*}
\widetilde{\beta}\lesssim_{K}  \max\left\{1, K\lV\x\rV_{2}\right\} \sqrt{mn}.
\end{eqnarray*}
Finally, by~\eqref{error2} and~\eqref{error3}, we can obtain the estimation error for the CVX-LS estimator is
              \begin{eqnarray}
		\lV\Z_{\star}-\x\x^*\rV_{F} \lesssim_{K,\mu} \max\left\{1,K\lV\x\rV_{2} \right\}\sqrt{\frac{n}{m}},
				\end{eqnarray}
			and 
	                \begin{eqnarray}
		\textbf{dist}\lk\z_\star,\x\rk\lesssim_{K,\mu} \max\left\{K, \frac{1}{\lV\x\rV_{2}} \right\} \sqrt{\frac{n}{m}}.
				\end{eqnarray}
			
\subsection{Proof of Theorem~\ref{low-dose}}

The proof of Theorem~\ref{low-dose} is nearly identical to that of Theorem~\ref{thm poisson}, differing mainly in the choice of parameters $\beta$ and $\widetilde{\beta}$ for the case $\lV\x\rV_2\le 1/K$ and in the probability bounds, which no longer decay exponentially.

The upper bounds for the parameters $\alpha$ and $\widetilde{\alpha}$ are the same as those established in the proof of Theorem \ref{thm poisson}.
Following the argument in the proof of Theorem \ref{thm poisson}, by \textbf{Part}~$\mathrm{(a)}$ of Corollary~\ref{mul inq for heavy}, with probability at least $1-c_5\frac{\log^{4}}{m} m-2\exp\lk-c_6n\rk$,
\begin{eqnarray*}
\begin{aligned}  
	    \lv  \sum_{k=1}^{m}\xi_k\lg\pp_k\pp_k^*,\M\rg \rv
		&\lesssim_{K} \lV\xi\rV_{L_4}\sqrt{mn}\lV\M\rV_{F}\\
		&\lesssim_{K} \max\left\{\sqrt{K\lV\x\rV_{2}}, K\lV\x\rV_{2} \right\} \sqrt{mn}\lV\M\rV_{F}\\
        &\lesssim_{K} \sqrt{K\lV\x\rV_{2}}\cdot \sqrt{mn}\lV\M\rV_{F},
		\end{aligned}
		\end{eqnarray*}
			provided $m\gtrsim_{K} n$.
Here, the second inequality follows from Proposition~\ref{L_q of poisson}, and the third inequality is due to $\lV\x\rV_2\le1/K$. 
Therefore, we have $$\beta\lesssim_{K} \sqrt{K\lV\x\rV_{2}}\cdot\sqrt{mn}.$$
Similarly, by \textbf{Part}~$\mathrm{(b)}$ of Corollary~\ref{mul inq for heavy}, we can also  obtain $\widetilde{\beta}\lesssim_{K} \sqrt{K\lV\x\rV_{2}}\cdot\sqrt{mn}.$
Thus, by \eqref{error1}, for the NVCX-LS estimator, we can obtain
	                \begin{eqnarray}
		\textbf{dist}\lk\z_\star,\x\rk \lesssim_{K,\mu} \min\left\{\sqrt{\frac{K}{\lV\x\rV_2}}\cdot\sqrt{\frac{n}{m}},\,\lk K\lV\x\rV_{2}\rk^{1/4}\cdot\lk\frac{n}{m}\rk^{1/4}\right\}.
				\end{eqnarray}
And by \eqref{error2} and \eqref{error3}, for the CVX-LS estimator, we can deduce that    
\begin{eqnarray}
		\lV\Z_{\star}-\x\x^*\rV_{F} \lesssim_{K,\mu} \sqrt{K\lV\x\rV_{2}}\cdot\sqrt{\frac{n}{m}},
				\end{eqnarray}
and 
 \begin{eqnarray}
		\textbf{dist}\lk\z_\star,\x\rk \lesssim_{K,\mu} \sqrt{\frac{K}{\lV\x\rV_{2}}}\cdot\sqrt{\frac{n}{m}}.
				\end{eqnarray}

\subsection{Proof of Theorem~\ref{thm heavy}}

The proof of Theorem~\ref{thm heavy}  follows a similar structure to that of Theorem \ref{thm poisson}. 
For the NCVX-LS estimator, we also have that $$\alpha\gtrsim_{K,\mu} m$$
holds with probability at least $1- \mathcal{O}\lk e^{-c_7m}\rk$, assuming $m\gtrsim_{K,\mu} n$.
By \textbf{Part}~$\mathrm{(a)}$ of Corollary~\ref{mul inq for heavy}, with probability at least $1-c_8m^{-\lk q/2-1\rk}\log^{q} m-2\exp\lk-c_9n\rk$,
we have 
\begin{eqnarray*}
\beta \lesssim _{K,q} \lV\xi\rV_{L_q}\sqrt{mn}
\end{eqnarray*}
when provided $m\gtrsim_K n$.
Therefore, by \eqref{error1}, we can obtain
	                \begin{eqnarray}
		\textbf{dist}\lk\z_\star,\x\rk \lesssim_{K,\mu,q} \min\left\{\frac{\lV\xi\rV_{L_q}}{\lV\x\rV_2}\cdot\sqrt{\frac{n}{m}},\,\sqrt{\lV\xi\rV_{L_q}}\cdot\lk\frac{n}{m}\rk^{1/4}\right\}.
				\end{eqnarray}
				
For the CVX-LS estimator,  applying Lemma~\ref{approximate} together with \textbf{Part}~$\mathrm{b}$ of Corollary~\ref{mul inq for heavy}, we similarly obtain 
\begin{eqnarray*}
\widetilde{\alpha} \ge \frac{1}{36}m\quad\text{and}\quad\widetilde{\beta} \lesssim_{K,q} \lV\xi\rV_{L_q}\sqrt{mn},
\end{eqnarray*}
with the same probability bounds as that established for the NCVX-LS estimator.
Thus by \eqref{error2} and \eqref{error3}, we can deduce that    
\begin{eqnarray}
		\lV\Z_{\star}-\x\x^*\rV_{F} \lesssim_{K,\mu,q}\lV\xi\rV_{L_q}\cdot\sqrt{\frac{n}{m}},
				\end{eqnarray}
and 
 \begin{eqnarray}
		\textbf{dist}\lk\z_\star,\x\rk \lesssim_{K,\mu,q} \frac{\lV\xi\rV_{L_q}}{\lV\x\rV_2}\cdot\sqrt{\frac{n}{m}}.
				\end{eqnarray}

\section{Minimax Lower Bounds}\label{minimax}

The goal of this section is to establish the minimax lower bounds stated in Theorem~\ref{minimax poisson} and Theorem~\ref{minimax heavy}.
The core idea is to follow the general framework presented in \cite{tsybakov2008nonparametric}, while refining the analysis in \cite{chen2017solving}. 
Specifically, we construct a finite set of well-separated hypotheses and apply a Fano-type minimax lower bound to derive the desired results. 
Since the hypotheses can be constructed in the real domain, it suffices to restrict our attention to the case where $\x\in\R^n$ and $\left\{\pp_k\right\}_{k=1}^m\overset{\text{i.i.d.}}{\sim}\mathcal{N}\lk\pmb{0},\pmb{I}_n\rk$.

For any two probability measures $\mathcal{P}$ and $\mathcal{Q}$, we denote by $\text{KL}\lk \mathcal{P}\|\mathcal{Q}\rk$ the Kullback-Leibler (KL) divergence between them:
 \begin{equation}\label{KL0}
		\text{KL}\lk \mathcal{P}\|\mathcal{Q}\rk:=\int\log\lk\frac{d \mathcal{P} }{d \mathcal{Q}}\rk d \mathcal{P}.
                \end{equation}
Below, we gather some results that will be used.
The first result provides an upper bound for the KL divergence between two Poisson-distributed datasets.
  
\begin{lemma}\label{KL1}
Fix a family of design vectors $\left\{\pp_k\right\}_{k=1}^m$. 
Let $\mathbb{P}\lk\y \mid \z\rk$ be the likelihood of $y_{k}\overset{\text{ind.}}{\sim}\text{Poisson}\lk\lv\lg\pp_{k},\z \rg\rv^{2}\rk$ conditional on $\left\{\pp_k\right\}_{k=1}^m$, where $k=1,2,\cdots,m$.
Then for any $\z,\x\in \R^n$, one has
 \begin{eqnarray}\label{KL ineq}
 \text{KL}\lk \mathbb{P}\lk\y \mid \z\rk \|\mathbb{P}\lk\y \mid \x\rk\rk
 \le\sum_{k=1}^m \lv \pp_k^{\top} \lk \z - \x \rk \rv^2\lk8+
        2\frac{\lv \pp_k^{\top} \lk \z - \x \rk \rv^2}{\lv \pp_k^{\top}  \x \rv^2}\rk.
 \end{eqnarray}
\end{lemma}

\begin{proof}
Note that the KL divergence between two Poisson distributions with rates $\lambda_1$ and $\lambda_0$ satisfies 
 \begin{eqnarray*}
 \begin{aligned}
 \text{KL}\lk \text{Poisson}\lk\lambda_1\rk \|\text{Poisson}\lk\lambda_0\rk\rk
 &=\lambda_0-\lambda_1+\lambda_1\log\lk\frac{\lambda_1}{\lambda_0}\rk\\
 &\le \lambda_0-\lambda_1+\lambda_1\lk\frac{\lambda_1}{\lambda_0}-1\rk\\
 &=\frac{\lk\lambda_1-\lambda_0\rk^2}{\lambda_0}.
 \end{aligned}
 \end{eqnarray*}
  Thus, by the definition of the KL divergence and triangle inequality, we can further bound
 \begin{equation*}
    \begin{aligned}
        \text{KL}\lk \mathbb{P}\lk\y \mid \z\rk \|\mathbb{P}\lk\y \mid \x\rk\rk
        &\le\sum_{k=1}^m\frac{\lk\lv\lg\pp_k,\z\rg\rv^2-\lv\lg\pp_k,\x\rg\rv^2\rk^2}{\lv\lg\pp_k,\x\rg\rv^2}\\
        &\le\sum_{k=1}^m \lv \pp_k^{\top} \lk \z - \x \rk \rv^2\frac{\lk2\lv \pp_k^{\top}  \x  \rv+\lv \pp_k^{\top} \lk \z - \x \rk \rv\rk^2}{\lv \pp_k^{\top}  \x \rv^2}\\
        &\le\sum_{k=1}^m \lv \pp_k^{\top} \lk \z - \x \rk \rv^2\lk8+
        2\frac{\lv \pp_k^{\top} \lk \z - \x \rk \rv^2}{\lv \pp_k^{\top}  \x \rv^2}\rk.
    \end{aligned}
\end{equation*}

\end{proof}

The second result provides an upper bound for the KL divergence between two Gaussian-distributed datasets.

\begin{lemma}\label{KL2}
Fix a family of design vectors $\left\{\pp_k\right\}_{k=1}^m$. 
Let $\mathbb{P}\lk\y \mid \z\rk$ be the likelihood of $y_{k}\overset{\text{ind.}}{\sim} \lv\lg\pp_{k},\z \rg\rv^{2}+\xi_k$ conditional on $\left\{\pp_k\right\}_{k=1}^m$, where $\left\{\xi_k\right\}_{k=1}^{m}\overset{\text{i.i.d.}}{\sim} \mathcal{N}\lk 0,\sigma^2\rk$ and $k=1,2,\cdots,m$.
Then for any $\z,\x\in \R^n$, one has
 \begin{eqnarray}\label{KL ineq for Gaussian}
 \text{KL}\lk \mathbb{P}\lk\y \mid \z\rk \|\mathbb{P}\lk\y \mid \x\rk\rk
 \le\frac{1}{\sigma^2}\sum_{k=1}^m \lv \pp_k^{\top} \lk \z - \x \rk \rv^2\lk4\lv \pp_k^{\top}  \x  \rv^2+
        \lv \pp_k^{\top} \lk \z - \x \rk \rv^2\rk.
 \end{eqnarray}
\end{lemma}

\begin{proof}
The KL divergence between two Gaussian distributions $\mathcal{N}\lk \mu_1,\sigma^2\rk$ and $\mathcal{N}\lk \mu_2,\sigma^2\rk$ satisfies 
 \begin{eqnarray*}
 \text{KL}\lk \mathcal{N}\lk \mu_1,\sigma^2\rk \|\mathcal{N}\lk \mu_2,\sigma^2\rk\rk
 =\frac{1}{2\sigma^2}\lk\mu_1-\mu_2\rk^2.
 \end{eqnarray*}
  Thus we can further bound that
 \begin{equation*}
    \begin{aligned}
        \text{KL}\lk \mathbb{P}\lk\y \mid \z\rk \|\mathbb{P}\lk\y \mid \x\rk\rk
        &\le\frac{1}{2\sigma^2}\sum_{k=1}^m\lk\lv\lg\pp_k,\z\rg\rv^2-\lv\lg\pp_k,\x\rg\rv^2\rk^2\\
        &\le\frac{1}{2\sigma^2}\sum_{k=1}^m \lv \pp_k^{\top} \lk \z - \x \rk \rv^2\lk2\lv \pp_k^{\top}  \x  \rv+\lv \pp_k^{\top} \lk \z - \x \rk \rv\rk^2\\
        &\le\frac{1}{\sigma^2}\sum_{k=1}^m \lv \pp_k^{\top} \lk \z - \x \rk \rv^2\lk4\lv \pp_k^{\top}  \x  \rv^2+
        \lv \pp_k^{\top} \lk \z - \x \rk \rv^2\rk.
    \end{aligned}
\end{equation*}

\end{proof}

The quantities \eqref{KL ineq} and \eqref{KL ineq for Gaussian} in Lemma~\ref{KL1} and Lemma~\ref{KL2} turn out to be crucial in controlling the information divergence between different hypotheses.
To this end, we provide the following lemma, proved by modifying the argument in \cite{chen2017solving}, and which will be used to derive upper bounds for \eqref{KL ineq} and \eqref{KL ineq for Gaussian}.

\begin{lemma}\label{up for KL}
Suppose that $\left\{\pp_k\right\}_{k=1}^m \overset{\text{i.i.d.}}{\sim} \mathcal{N}\lk \pmb{0}, \pmb{I}_n \rk$, where $m,n$ are sufficiently large and $m \ge Ln$ for some sufficiently large constant $L > 0$.
Consider any $\x \in \R^n \setminus \{\pmb{0}\}$.  
There exists a collection $\mathcal{T}$ containing $\x$ with cardinality $\lv \mathcal{T} \rv = \exp \lk n / 200 \rk$, such that all $\pmb{\z}^{(i)} \in \mathcal{T}$ are distinct and satisfy the following properties:
\begin{itemize}
    \item[$\mathrm{(a)}$] With probability at least 
\begin{equation}\label{eq:pro for KL}
1 - \frac{3}{\log m}-5\exp \lk -\Omega \lk \frac{n}{\log m} \rk \rk-\exp \lk -\Omega \lk \frac{n^2}{m \log^2 n} \rk \rk,
\end{equation}
for all $\z^{(i)}, \z^{(j)} \in \mathcal{T}$,
   \begin{eqnarray}\label{distance}
    \frac{1}{\sqrt{8}} - (2n)^{-1/2} \le\lV \z^{(i)} - \z^{(j)} \rV_2 \le \frac{3}{2} + n^{-1/2},
    \end{eqnarray}
    and for all $\z \in \mathcal{T}\setminus \{\x\}$,
    \begin{eqnarray}\label{KL for poisson1}
    \frac{\lv \pp_k^{\top} \lk \z - \x \rk \rv^2}{\lv \pp_k^{\top} \x \rv^2} \le \lk 2 + 25600 \frac{m^2\log^3 m}{n^2} \rk\frac{\lV \z - \x \rV_2^2}{\lV \x \rV_2^2},\quad 1\le k\le m;
    \end{eqnarray}
    \item[$\mathrm{(b)}$] If $\frac{m}{n}\le\widetilde{L}\log m$ for some universal constant $\widetilde{L} > 0$, then
    with probability at least $1 - \frac{3}{\log m}-5\exp \lk -\Omega \lk \frac{m}{\log^4 m} \rk \rk$,
for all $\z^{(i)}, \z^{(j)} \in \mathcal{T}$, \eqref{distance} holds and for all $\z \in \mathcal{T}\setminus \{\x\}$,
    \begin{eqnarray}\label{KL for poisson2}
    \frac{\lv \pp_k^{\top} \lk \z - \x \rk \rv^2}{\lv \pp_k^{\top} \x \rv^2} \le \lk 2 + 16\log^5 m\rk\frac{\lV \z - \x \rV_2^2}{\lV \x \rV_2^2},\quad 1\le k\le m;
    \end{eqnarray}
    \item[$\mathrm{(c)}$] 
    With probability at least $1 - \frac{1}{\log m}-2\exp \lk -\Omega \lk n \rk \rk$,
for all $\z^{(i)}, \z^{(j)} \in \mathcal{T}$, \eqref{distance} holds and for all $\z \in \mathcal{T}\setminus \{\x\}$, 
    \begin{eqnarray}\label{KL for heavy1}
    \lv \pp_k^{\top} \lk \z - \x \rk \rv^2 \le 16 \log m \lV \z - \x \rV_2^2,\quad 1\le k\le m.
    \end{eqnarray}
\end{itemize}
\end{lemma}
\begin{proof}
See Appendix \ref{pf_KL}.
\end{proof}
\begin{remark} 
 
From \eqref{distance}, we observe that any two hypotheses in $\mathcal{T}$ are located around $\x$ while remaining well separated by a distance on the order of 1.
\textbf{Part}~$\mathrm{(a)}$ will be used to establish an upper bound for \eqref{KL ineq} in the proof of \textbf{Part}~$\mathrm{(a)}$ of Theorem~\ref{minimax poisson}, while
\textbf{Part}~$\mathrm{(b)}$ will be used in the proof of \textbf{Part}~$\mathrm{(b)}$ of the same theorem.
Finally, \textbf{Part}~$\mathrm{(c)}$ will be invoked to derive an upper bound for \eqref{KL ineq for Gaussian} in the proof of Theorem~\ref{minimax heavy}.
\end{remark}

\subsection{Proof of Theorem~\ref{minimax poisson}}\label{proof of minimax poisson}

We first  prove \textbf{Part}~$(\mathrm{a})$ of Theorem~\ref{minimax poisson}.
Define $\pmb{\Phi} := \left[\pp_1,\pp_2,\cdots,\pp_m\right]^{\mathrm{T}}$, and let $\mathcal{E}_1$ denote the event $\mathcal{E}_1:=\left\{\lV\pmb{\Phi}\rV_{op}\le\sqrt{2m}\right\}$.
By \cite[Theorem 4.6.1]{vershynin2018high}, $\mathcal{E}_1$ holds with probability at least  $1-2\exp\lk-\Omega \lk m\rk\rk$.
Let $\mathcal{E}_2$ be the event under which \textbf{Part}~$\mathrm{(a)}$ of Lemma~\ref{up for KL} holds.
Now, conditioning on the events $\mathcal{E}_1$ and $\mathcal{E}_2$, 
Lemma~\ref{KL1} together with \eqref{KL for poisson1} of Lemma~\ref{up for KL} implies that the KL divergence satisfies
\begin{equation*}
    \begin{aligned}
        \text{KL}\lk \mathbb{P}\lk\y \mid \z\rk \|\mathbb{P}\lk\y \mid \x\rk\rk
            &\le \sum_{k=1}^{m} \lv \pp_k^{\top} \lk \z - \x \rk \rv^2\lk8+2
        \frac{\lv \pp_k^{\top} \lk \z - \x \rk \rv^2}{\lv \pp_k^{\top}  \x \rv^2}\rk\\
        &\le 20m\lV\z - \x\rV^2_2+ 51200\frac{m^3\log^3 m}{n^2}\frac{\lV\z-\x\rV_2^4}{\lV\x\rV_2^2}.
    \end{aligned}
\end{equation*}
We rescale the hypotheses in $\mathcal{T}$ of Lemma~\ref{up for KL} by the substitution:\ 
$\z\leftarrow\x+\delta\lk\z-\x\rk$.
In such a way, we have that
\begin{equation*}
\lV\z^{\lk i\rk}-\x\rV_2 \asymp\delta\quad\text{and}\quad \lV\z^{\lk i\rk}-\z^{\lk j\rk}\rV_2 \asymp\delta,\quad \forall\ \z^{\lk i\rk},\z^{\lk j\rk}\in\mathcal{T}\setminus \{\x\}\ \text{with}\ \z^{\lk i\rk}\neq\z^{\lk j\rk}.
\end{equation*}

By \cite[Theorem 2.7]{tsybakov2008nonparametric}, if the the conditional KL divergence obeys
\begin{equation}\label{KL up}
\frac{1}{\lv\mathcal{T}\rv-1}\sum\limits_{\z^{\lk i\rk}\in\mathcal{T}\setminus \{ \x\}}\text{KL}\lk \mathbb{P}\lk\y \mid \z^{\lk i\rk}\rk \|\mathbb{P}\lk\y \mid \x\rk\rk\le\frac{1}{10}\log\lk\lv
\mathcal{T}\rv-1\rk,
\end{equation}
then the Fano-type minimax lower bound asserts that
\begin{equation*}
\inf\limits_{\widehat{\x}}\sup\limits_{\x\in\mathcal{T}}\E\left[\lV\widehat{\x}-\x\rV_2\mid \left\{\pp_k\right\}\right]\gtrsim \min\limits_{\substack{\z^{\lk i\rk},\z^{\lk j\rk}\in\mathcal{T}\\\z^{\lk i\rk}\neq\z^{\lk j\rk}}}
\lV\z^{\lk i\rk}-\z^{\lk j\rk}\rV_2.
\end{equation*}
Since $\lv\mathcal{T}\rv=\exp\lk n/200\rk$, \eqref{KL up} would follow from 
\begin{equation}\label{eq:KL poisson}
20\lV\z - \x\rV^2_2+ 51200\frac{m^2\log^3 m}{n^2}\frac{\lV\z-\x\rV_2^4}{\lV\x\rV_2^2}\le\frac{n}{2000m},\quad \forall\ \z\in\mathcal{T}.
\end{equation}

In the real domain, we have that $\textbf{dist}\lk\z,\x\rk=\min\left\{\lV\z-\x\rV_2,\lV\z+\x\rV_2\right\}$.
\textbf{Part}~$(\mathrm{a})$ of Lemma~\ref{up for KL} implies that if we set $\delta\le \frac{1}{12}\lV\x\rV_2$, then all the hypotheses $\z^{\lk i\rk}$ are around $\x$ at a distance about $\delta$ that is smaller than $\frac{1}{2}\lV\x\rV_2$, thus for hypotheses $\z^{\lk i\rk}$, we have 
$\textbf{dist}\lk\z^{\lk i\rk},\x\rk=\lV\z^{\lk i\rk}-\x\rV_2$, which implies for any estimator, we have
$\textbf{dist}\lk\widehat{\x},\x\rk=\lV\widehat{\x}-\x\rV_2$.
To meet the  condition \eqref{eq:KL poisson} and $\delta\le \frac{1}{12}\lV\x\rV_2$, we choose $\delta^2$ as
\begin{equation*}
\min\left\{\frac{1}{144}\lV\x\rV_2^2,\frac{\frac{n}{4000m}}{10+3\sqrt{\frac{\log^3 m}{\lV\x\rV^2_2}\cdot\frac{m}{n}}}\right\}.
\end{equation*}
Thereby, we can obtain
\begin{equation}\label{minimax result for poisson 1}
\inf\limits_{\widehat{\x}}\sup\limits_{\x\in\mathcal{T}}\E\left[\textbf{dist}\lk\widehat{\x},\x\rk\mid \left\{\pp_k\right\}\right]\gtrsim \delta\asymp \min\left\{\lV\x\rV_2,\frac{\sqrt{\frac{n}{m}}}{1+\frac{\log^{3/4} m}{\sqrt{\lV\x\rV_2}}\cdot\lk \frac{m}{n}\rk^{1/4}}\right\}.
\end{equation}
To ensure that the probability \eqref{eq:pro for KL} tends to 1, we impose $\frac{m}{n^2}\le\frac{\widetilde{L}}{\log^3 m}$ for some universal constant $L>0$.

We turn to prove \textbf{Part}~$(\mathrm{b})$ of Theorem~\ref{minimax poisson}.
Let $\mathcal{E}_3$ be the event that \textbf{Part}~$\mathrm{(b)}$ of Lemma~\ref{up for KL} holds.
Now, conditioning on the events $\mathcal{E}_1$ and $\mathcal{E}_3$, 
Lemma~\ref{KL1} together with \eqref{KL for poisson2} of Lemma~\ref{up for KL} implies that \eqref{KL up} follows from 
\begin{equation}\label{eq:KL poisson2}
            20\lV\z - \x\rV^2_2+ 32\log^5 m\frac{\lV\z-\x\rV_2^4}{\lV\x\rV_2^2}\le\frac{n}{2000m},\quad \forall\ \z\in\mathcal{T}.
\end{equation}
If $\lV\x\rV_2=o\lk \frac{\sqrt\frac{n}{m}}{\log^{5/2}m}\rk$, we set 
\begin{equation*}
\delta\asymp\sqrt{\lV\x\rV_2}\cdot\frac{\lk\frac{n}{m}\rk^{1/4}}{\log^{5/4}m}.
\end{equation*}
Then the condition \eqref{eq:KL poisson2} holds and we have $\lV\x\rV_2\ll\delta$.
Thus, for any $\z\in\mathcal{T}\setminus \{ \x\}$, we have 
\begin{equation*}
\begin{aligned}
\textbf{dist}\lk\z,\x\rk&=\min\left\{\lV\z-\x\rV_2,\lV\z+\x\rV_2\right\}\\
&\ge \min\left\{\lV\z-\x\rV_2,\lV\z-\x\rV_2-2\lV\x\rV_2\right\}\\
&=\lV\z-\x\rV_2-2\lV\x\rV_2\asymp\lV\z-\x\rV_2,
\end{aligned}
\end{equation*}
which implies that
\begin{equation}\label{minimax result for poisson 2}
\inf\limits_{\widehat{\x}}\sup\limits_{\x\in\mathcal{T}}\E\left[\textbf{dist}\lk\widehat{\x},\x\rk\mid \left\{\pp_k\right\}\right]\gtrsim \delta\asymp \sqrt{\lV\x\rV_2}\cdot\frac{\lk\frac{n}{m}\rk^{1/4}}{\log^{5/4}m}.
\end{equation}

\subsection{Proof of Theorem~\ref{minimax heavy}}

We follow the steps in the proof of Theorem~\ref{minimax poisson}.
Let $\mathcal{E}_4$ be the event under which \textbf{Part}~$\mathrm{(c)}$ of Lemma~\ref{up for KL} holds.
Conditioning on the event $\mathcal{E}_1$ and $\mathcal{E}_4$,
Lemma~\ref{KL2} together with \textbf{Part}~$\mathrm{(c)}$ of Lemma~\ref{up for KL} implies that, in this case, the conditional KL divergence satisfies
\begin{equation*}
    \begin{aligned}
        \text{KL}\lk \mathbb{P}\lk\y \mid \z\rk \|\mathbb{P}\lk\y \mid \x\rk\rk
            &\le\frac{1}{\sigma^2}\sum_{k=1}^m \lv \pp_k^{\top} \lk \z - \x \rk \rv^2\lk4\lv \pp_k^{\top}  \x  \rv^2+
        \lv \pp_k^{\top} \lk \z - \x \rk \rv^2\rk\\
        &\le \frac{8}{\sigma^2}m\log m\lV\z - \x\rV^2_2\lV\x\rV_2^2+ \frac{32}{\sigma^2}m\log m\lV\z - \x\rV_2^4.
    \end{aligned}
\end{equation*}
We rescale the hypotheses by the substitution:\ 
$\z\leftarrow\x+\delta\lk\z-\x\rk$.
By \cite[Theorem 2.7]{tsybakov2008nonparametric} and noting that $\lv\mathcal{T}\rv=\exp\lk n/200\rk$,
we can obtain the Fano-type minimax lower bound provided that the following inequality holds
\begin{equation}\label{eq:KL gaussian}
8\log m\lV\z - \x\rV^2_2\lV\x\rV_2^2+ 32\log m\lV\z - \x\rV_2^4\le\frac{\sigma^2n}{2000m},\quad \forall\ \z\in\mathcal{T}.
\end{equation}

For \textbf{Part}~$(\mathrm{a})$ of Theorem~\ref{minimax heavy}, in order to satisfy condition~\eqref{eq:KL gaussian} and ensure that all hypotheses $\z^{(i)}$ obey $\textbf{dist}\lk\z^{\lk i\rk},\x\rk=\lV\z^{\lk i\rk}-\x\rV_2$, we choose $\delta^2$ as
\begin{equation*}
\min\left\{\frac{1}{144}\lV\x\rV_2^2,\frac{\frac{n}{4000m}}{8\log m\lV\x\rV^2_2/\sigma^2+\sqrt{\frac{2\log{m}}{125\sigma^2}\cdot\frac{n}{m}}}\right\}.
\end{equation*}
Thus, we can obtain
\begin{equation}\label{minimax result for gaussian 1}
\inf\limits_{\widehat{\x}}\sup\limits_{\x\in\mathcal{T}}\E\left[\textbf{dist}\lk\widehat{\x},\x\rk\mid \left\{\pp_k\right\}\right]\gtrsim \delta\asymp 
\min\left\{\lV\x\rV_2,\frac{\sqrt{\frac{n}{m}}}{\lV\x\rV_2\sqrt{\log m}/\sigma+\lk\frac{\log m}{\sigma^2}\rk^{1/4}\cdot\lk \frac{n}{m}\rk^{1/4}}\right\}.
\end{equation}

For \textbf{Part}~$(\mathrm{b})$ of Theorem~\ref{minimax heavy}, since $\lV\x\rV_2=o\lk \sqrt{\sigma}\cdot\frac{\lk\frac{n}{m}\rk^{1/4}}{\log^{1/4}m}\rk$, we set 
\begin{equation*}
\delta\asymp\sqrt{\sigma}\cdot\frac{\lk\frac{n}{m}\rk^{1/4}}{\log^{1/4}m}.
\end{equation*}
Thus, condition~\eqref{eq:KL gaussian} holds and we obtain $\lV\x\rV_2\ll\delta$,
which further implies that for any $\z^{(i)}\in\mathcal{T}\setminus \{\x\}$, we have $\textbf{dist}\lk\z^{\lk i\rk},\x\rk\asymp\lV\z^{\lk i\rk}-\x\rV_2$.
Finally, we can obtain
\begin{equation}\label{minimax result for gaussian 2}
\inf\limits_{\widehat{\x}}\sup\limits_{\x\in\mathcal{T}}\E\left[\textbf{dist}\lk\widehat{\x},\x\rk\mid \left\{\pp_k\right\}\right]\gtrsim \delta\asymp \sqrt{\sigma}\cdot\frac{\lk\frac{n}{m}\rk^{1/4}}{\log^{1/4}m}.
\end{equation}

\section{Numerical Simulations}\label{exp}

In this section, we carry out a series of numerical simulations to confirm the validity of our theory.
In particular, we demonstrate the stable performance of the NCVX-LS and CVX-LS estimators vis-à-vis Poisson noise and heavy-tailed noise.

\subsection{Numerical Performance for Poisson Model}\label{subsec:num for poisson}

We investigate the numerical performance of the NCVX-LS and CVX-LS estimators for Poisson model \eqref{poisson}.
We will use the relative mean squared error (MSE) and the mean absolute error (MAE) to measure performance.
Since a solution is only unique up to the global phase, we compute the distance modulo a global phase
term and define the relative MSE and MAE as
\begin{equation*}\label{eq:MSE}
\text{MSE}:=\inf\limits_{\lv c\rv=1}\frac{\lV c\z_{\star}-\x\rV^2_2}{\lV\x\rV^2_2}\quad \text{and}\quad \text{MAE}:=\inf\limits_{\lv c\rv=1}\lV c\z_{\star}-\x\rV_2.
\end{equation*}

\begin{figure}[htbp]
    \centering
    \includegraphics[width=1\textwidth]{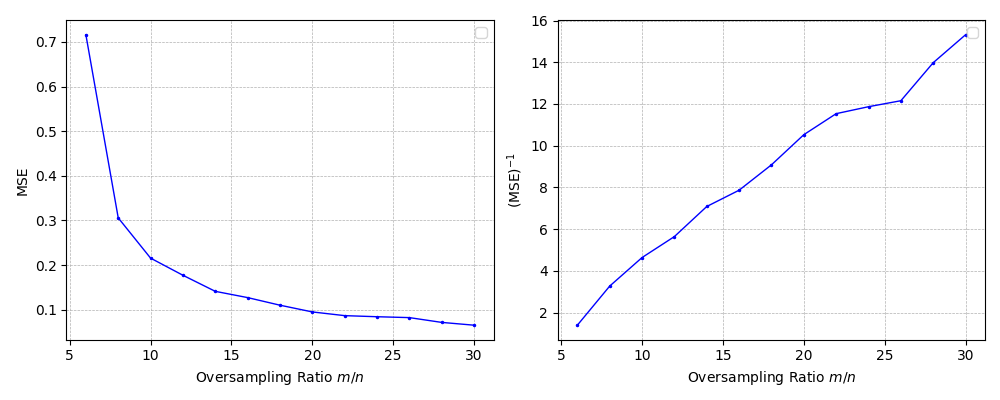}  
    \caption{Poisson: \textsc{NCVX-LS} with $m/n$.
}
    \label{fig1}
\end{figure}

\begin{figure}[htbp]
    \centering
    \includegraphics[width=1\textwidth]{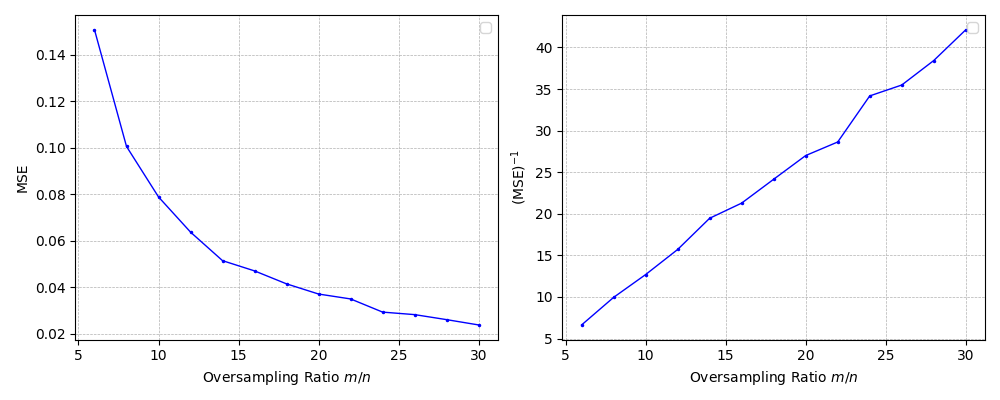}  
    \caption{Poisson: \textsc{CVX-LS} with $m/n$.
}
    \label{fig3}
\end{figure}

In the first experiment, we examine the performance of the NCVX-LS and CVX-LS estimators as the oversampling ratio $r := m/n$ increases under Poisson noise.
The NCVX-LS estimator is solved using the Wirtinger Flow (WF) algorithm (see \cite{candes2015phase4}).
The CVX-LS estimator is implemented in Python using MOSEK; to obtain an approximation $\z_{\star}$, we extract its largest rank-1 component as described in Section~\ref{setup}.
The test signal $\x \in \mathbb{C}^n$ is randomly generated and normalized to unit $\ell_2$-norm, i.e., $\lV\x \rV_2 = 1$;
we set $n=32$ for NCVX-LS and $n=16$ for CVX-LS, since the convex formulation incurs higher memory costs.
The sampling vectors are independently drawn from $\mathcal{CN}\lk\pmb{0}, \pmb{I}_n\rk$.
We vary the oversampling ratio $r$ from 6 to 30 in increments of 2.
For each value of $r$, the experiment is repeated 50 times and the average relative MSE is reported.

Figures~\ref{fig1} and \ref{fig3} plot the relative MSE of the NCVX-LS and CVX-LS estimators against the oversampling ratio.
The results show that the relative MSE decreases inversely with $r$, while its reciprocal grows nearly linearly in $r$.
Since $\lV\x \rV_2 = 1$, this empirical trend corroborates our theoretical prediction that, in the high-energy regime, the estimation error scales linearly with $\sqrt{n/m}$.

We examine the performance of the NCVX-LS estimator as the signal energy increases under Poisson noise.
The algorithm employs the truncated spectral initialization from \cite{chen2017solving} together with the iterative refinement method of \cite{candes2015phase4}.
The test signal $\x \in \mathbb{C}^n$ is randomly generated with length $n = 10$, normalized to unit $\ell_2$-norm, and then scaled by a factor $\alpha$ ranging from 0.01 to 1 in increments of 0.01.
The oversampling ratio is fixed at $r = 40$.
For each $\alpha$, the experiment is repeated 50 times with independently generated noise and measurement matrices, and the average MAE is reported.

Figure~\ref{fig2} plots the MAE against $\sqrt{\alpha}$. 
The results show that when $\sqrt{\alpha} \in (0,0.4)$, the MAE grows approximately linearly with $\sqrt{\alpha}$. 
Beyond the threshold $\sqrt{\alpha} \approx 0.4$, the MAE stabilizes within a narrow band between 0.13 and 0.15.
This empirical behavior aligns with our theoretical findings: witg a fixed oversampling ratio, the estimation error of the NCVX-LS estimator grows proportionally to $\sqrt{\lV\x\rV_2}$ in the low-energy regime, consistent with the minimax lower bound, whereas in the high-energy regime, the error becomes nearly independent of the signal energy.

\begin{figure}[htbp]
    \centering
    \includegraphics[width=0.75\textwidth]{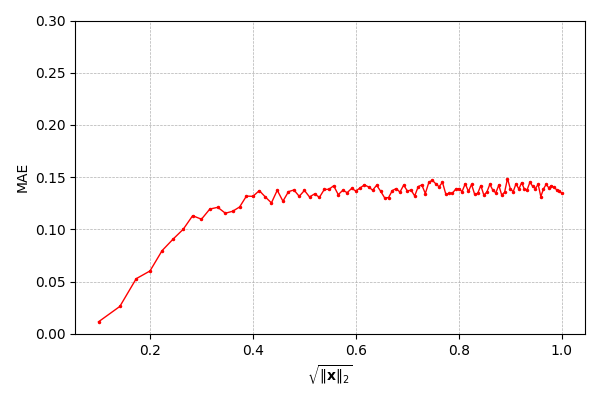}  
    \caption{Poisson: \textsc{NCVX-LS} with $\sqrt{\lV\x\rV_2}$.
}
    \label{fig2}
\end{figure}

\subsection{Numerical Performance for Heavy-tailed Model}

We investigate the numerical performance of the NCVX-LS and CVX-LS estimators for hevay-tailed model \eqref{heavy}.
Performance is measured using the relative MSE and MAE defined in Section~\ref{subsec:num for poisson}.
To model heavy-tailed corruption, we add independent additive noise to each measurement,
drawn from a Student's $t$-distributions with degrees of freedom (DoF) $\nu$, which will be specified subsequently.
The Student’s $t$-distribution is symmetric with heavier tails than the Gaussian distribution, and the tail heaviness is controlled by $\nu$: smaller $\nu$ produces heavier tails and more extreme outliers, 
while $\nu \to \infty$ recovers the standard normal distribution $\mathcal{N}\lk 0,1\rk$.

\begin{figure}[htbp]
    \centering
    \includegraphics[width=1\textwidth]{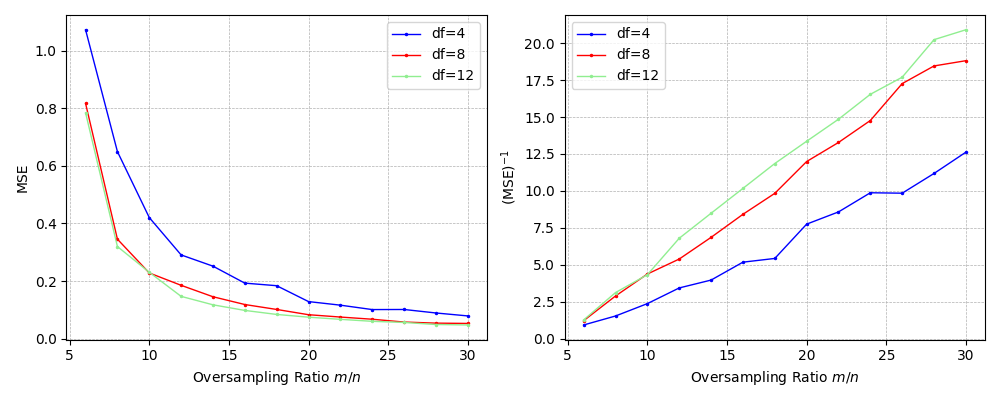}  
    \caption{Hevay-tail: \textsc{NCVX-LS} with $m/n$.}
    \label{fig4}
\end{figure}

\begin{figure}[htbp]
    \centering
    \includegraphics[width=1\textwidth]{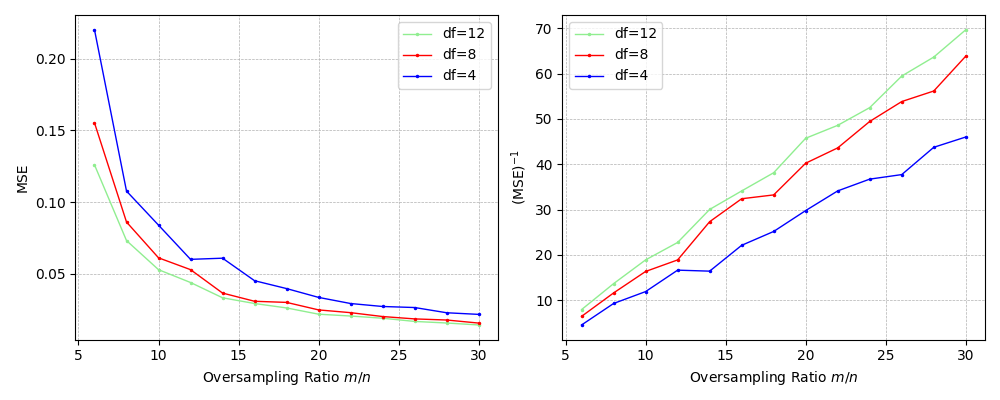}  
    \caption{Hevay-tail: \textsc{CVX-LS} with $m/n$.}
    \label{fig5}
\end{figure}
We investigate the performance of the NCVX-LS and CVX-LS estimators as the oversampling ratio $r$ increases under heavy-tailed noise.
The NCVX-LS estimator is solved using truncated spectral initialization \cite{chen2017solving} followed by WF iterations \cite{candes2015phase4}, while the CVX-LS estimator is implemented in Python with MOSEK.
The ratio $r$ ranges from 6 to 30 in increments of 2. 
In each trial, the true signal $\x$ is randomly generated and normalized to unit $\ell_2$-norm;
we set $n=32$ for NCVX-LS and $n=16$ for CVX-LS.
Independent sampling vectors are drawn from $\mathcal{CN}\lk\pmb{0},\pmb{I}_n\rk$ and heavy-tailed noise is generated from Student's $t$-distributions with $\nu \in \left\{4, 8, 12\right\}$.
For each combination of $r$ and $\nu$, the experiment is repeated 50 times, and the average relative MSE across trials is reported.

Figures~\ref{fig4} and \ref{fig5} show that the relative MSE decreases as the oversampling ratio increases, and its reciprocal grows approximately linearly with $r$.
This empirical trend is consistent with our theoretical prediction that the estimation error of both estimators scales as $\sqrt{n/m}$ in the high-energy regime.
Moreover, the estimation error decreases with increasing $\nu$: extremely heavy-tailed noise (small $\nu$) may destabilize the estimators, whereas lighter-tailed noise (larger $\nu$) improves accuracy, reflecting their robustness.

We also examine the performance of the NCVX-LS estimator as the signal energy increases under heavy-tailed noise. 
We solve the NCVX-LS estimator using the WF method with a prior-informed initialization. 
To mitigate the high sensitivity of the truncated spectral initialization to heavy-tailed noise in the low-energy regime, we initialize the algorithm at $s\x$, where the scaling factor $s\in[0.8,1.2]$ is randomly selected.
The test signal $\x \in \mathbb{C}^n$ is randomly generated with length $n = 10$, normalized to unit $\ell_2$-norm, and then scaled by a factor $\alpha$ ranging from 0.01 to 0.5 in increments of 0.01 and from 0.5 to 1.2 in increments of 0.03. The oversampling ratio is fixed at $r = 40$. 
For each $\alpha$, the experiment is repeated 50 times with independently generated noise drawn from a Student's $t$-distribution with $\nu = 8$, and the average MAE is reported.

Figure~\ref{fig6} plots the MAE against $\alpha$. 
The results show that when $\alpha \in (0,0.5)$, the MAE remains within the range of approximately 0.35 to 0.45. 
Beyond the threshold $\alpha \approx 0.5$, the MAE decreases as $\alpha$ continues to grow. 
This behavior reflects the experimental trend: with a fixed oversampling ratio, the estimation error of the NCVX-LS estimator remains relatively stable in the low-energy regime, whereas in the high-energy regime, it gradually decreases as the signal energy increases.

\begin{figure}[htbp]
    \centering
    \includegraphics[width=0.75\textwidth]{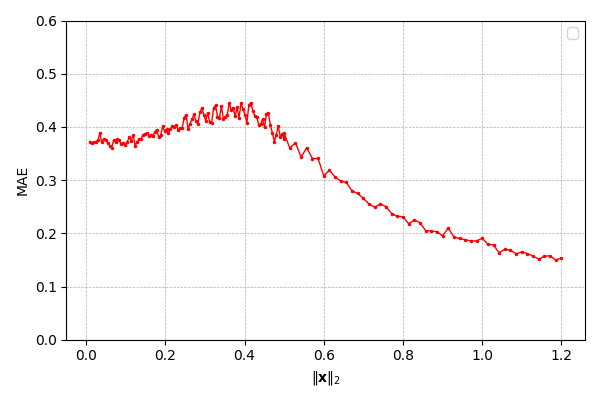}  
    \caption{Hevay-tail: \textsc{NCVX-LS} with $\lV\x\rV_2$.}
    \label{fig6}
\end{figure}

\section{Further Illustrations}\label{app}

In this section, we extend our analytical framework to three additional problems: sparse phase retrieval, low-rank PSD matrix recovery, and random blind deconvolution. 
We further derive the corresponding error bounds to characterize their stable performance of LS-type estimators in these settings.

\subsection{Sparse Phase Retrieval}\label{sub spr}

We first formulate the sparse phase retrieval problem. 
Specifically, we consider applying the NCVX-LS estimator to recover an $s$-sparse signal $\x\in\C^n$ and investigate its stable performance under the given noise settings.
Therefore, we modify the constraint set in the NCVX-LS estimator \eqref{model1} as follows:
          \begin{equation} \label{model3}
		\begin{array}{ll}
			\text{minimize}   & \quad\lV\ppp\lk\z\rk-\y\rV_2\\
			\text{subject to} & \quad \z\in\Sigma_{s}^{n}.\\				 
		\end{array}
	\end{equation}
Here, $\ppp\lk\z\rk$ denotes the phaseless operator as previously defined, $\y$ represents either Poisson model~\eqref{poisson} or heavy-tailed model \eqref{heavy}, and $\Sigma_{s}^{n}:=\left\{\lV\z\rV_{0}\le s:\z\in\C^{n}\right\}$ denotes the set of $s$-sparse signals in $\C^n$.
We refer to \eqref{model3} as the sparse NCVX-LS estimator.

	The following theorem addresses sparse phase retrieval under the Poisson model~\eqref{poisson}.
    
	 \begin{theorem}\label{sparse pr_poisson}
	 Let $\x$ be an $s$-sparse signal. 
     Suppose that $\left\{\pp_k\right\}_{k=1}^{m}$ satisfy Assumption \ref{sample} and the Poisson model \eqref{poisson} satisfies the distribution in Assumption~\ref{noise0} $\mathrm{(a)}$.
     Then there exist universal constants $L,\widetilde{L},c_1,c_2,C_1,C_2> 0$ that depend only on $K$ and $\mu$ such that the following holds:
	 	\begin{itemize}
		\item[$\mathrm{(a)}$]  
If $m\ge L s\log\lk\frac{en}{s}\rk$, then with probability at least $1-\mathcal{O}\lk e^{-c_1 s\log\lk en/s\rk}\rk$, the sparse NCVX-LS estimator satisfies the following error bound uniformly for all $\x\in\Sigma_{s}^{n}$, 
\begin{align}
\textbf{dist}\lk\z_\star,\x\rk
\le C_1 \min \bigg\{ 
    & \max\left\{ K, \frac{1}{\lV\x\rV_2} \right\} \cdot \sqrt{\frac{s \log(en/s)}{m}}, \notag \\
    & \max\left\{ 1, \sqrt{K\lV\x\rV_2} \right\} \cdot \lk \frac{s \log(en/s)}{m} \rk^{1/4} 
\bigg\}.
\end{align}
	\item[$\mathrm{(b)}$]
Let $\Gamma_s:=\left\{\x\in\Sigma_{s}^{n}:\lV\x\rV_2\le\frac{1}{K}\right\}$.
If $m\ge \widetilde{L} s\log\lk\frac{en}{s}\rk$, then with probability at least $1-\mathcal{O}\lk \frac{\log^4m}{m}\rk-\mathcal{O}\lk e^{-c_2 s\log\lk en/s\rk}\rk$, the sparse NCVX-LS estimator satisfies the following error bound uniformly for all $\x\in\Gamma_s$,
\begin{align}
\textbf{dist}\lk\z_\star,\x\rk
\le C_2 \min \bigg\{ 
     &\sqrt{\frac{K}{\lV\x\rV_2}} \cdot \sqrt{\frac{s \log(en/s)}{m}}, \notag \\
     &\left( K \lV\x\rV_2 \right)^{1/4} \cdot \lk \frac{s \log(en/s)}{m} \rk^{1/4} 
\bigg\}.
\end{align}
	\end{itemize}  
\end{theorem}
We provide some comments on Theorem~\ref{sparse pr_poisson}.
\textbf{Part}~$\mathrm{(a)}$ of Theorem~\ref{sparse pr_poisson} establishes that the sparse NCVX-LS estimator attains an error bound of {\small$\mO\lk\sqrt{\frac{s\log\lk en/s\rk}{m}}\rk$} in the high-energy regime.
This rate appears to be minimax optimal, since a matching lower bound of the same order can be obtained in this regime by adapting the proof of Theorem~\ref{minimax poisson}.
In contrast, \textbf{Part}~$\mathrm{(b)}$ of Theorem~\ref{sparse pr_poisson} demonstrates that, in the low-energy regime, the sparse NCVX-LS estimator achieves an error bound
{\small
$
\mO\lk\lV\x\rV_{2}^{1/4}\cdot\lk\frac{s\log\lk en/s\rk}{m}\rk^{1/4}\rk,
$}
which decays with the signal energy.
These results seem to be the first theoretical guarantee for sparse phase retrieval under Poisson noise, thereby establishing the provable performance of the proposed estimator.

We also provide the following theorem for sparse phase retrieval under heavy-tailed model \eqref{heavy}.

	 \begin{theorem}\label{sparse pr_heavy}
	 Let $\x$ be an $s$-sparse signal. 
     Suppose that $\left\{\pp_k\right\}_{k=1}^{m}$ satisfy Assumption~\ref{sample} and the heavy-tailed model~\eqref{heavy} satisfies the conditions in Assumption~\ref{noise0} $\mathrm{(b)}$ with $q>2$.
		 Then there exist universal constants $L,c,C> 0$ dependent only on $K,\mu$ and $q$ such that when provided
         $m\ge Ls\log\lk\frac{en}{s}\rk$, with probability at least 
         $$1-\mathcal{O}\lk m^{\lk q/2-1\rk}\log^{q} m\rk-\mathcal{O}\lk e^{-c s\log\lk en/s\rk}\rk,$$ simultaneously for all signals $\x\in\Sigma_{s}^{n}$, the sparse NCVX-LS estimates obey
 \begin{eqnarray}
		\textbf{dist}\lk\z_{\star},\x\rk\le C\min\left\{\frac{\lV\xi\rV_{L_q}}{\lV\x\rV_{2}}\cdot\sqrt{\frac{s\log\lk en/s\rk}{m}},\,\sqrt{\lV\xi\rV_{L_q}}\cdot\lk\frac{s\log\lk en/s\rk}{m}\rk^{1/4}\right\}.
					\end{eqnarray}	
	
\end{theorem}

We discuss Theorem~\ref{sparse pr_heavy} and its relation to existing work.
In particular, \cite{lecue2015minimax} analyzed the same sparse NCVX-LS estimator under i.i.d., mean-zero, sub-Gaussian noise and derived an error bound {\small$\widetilde{\mO}\lk\frac{ \lV\xi\rV_{\psi_2}}{\lV\x\rV_2} \cdot\sqrt{\frac{s\log\lk en/s\rk}{m}}\rk$}.
For i.i.d. Gaussian noise $\mathcal{N}\lk0,\sigma^2\rk$, with sufficiently large signal energy, they showed that no estimator can achieve a smaller error than {\small$\Omega\lk\frac{\sigma}{\lV\x\rV_2}\cdot\sqrt{\frac{s\log\lk en/s\rk}{m}}\rk$}, establishing the minimax lower bound.
Subsequent work~\cite{cai2016optimal,wu2023nearly} considered independent, centered sub-exponential noise and proposed convergent algorithms attaining nearly minimax optimal rate {\small$\mO\lk\frac{\lV\xi\rV_{\psi_1}}{\lV\x\rV_2}\cdot \sqrt{\frac{s\log n}{m}}\rk$}.
Theorem~\ref{sparse pr_heavy} extends these results to the heavy-tailed model~\eqref{heavy}. 
Under suitable assumptions, the sparse NCVX-LS estimator achieves the minimax optimal rate
{\small
$\mO\lk\frac{\lV\xi\rV_{L_q}}{\lV\x\rV_{2}}\cdot\sqrt{\frac{s\log \lk en/s\rk}{m}}\rk$
} in the high-energy regime,
matching the best-known results in \cite{lecue2015minimax,cai2016optimal,wu2023nearly}.
In the low-energy regime, it achieves {\small$\mO\lk\sqrt{\lV\xi\rV_{L_q}}\cdot\lk\frac{s\log\lk en/s\rk}{m}\rk^{1/4}\rk$},
which also appears to be minimax optimal, as a matching lower bound can be established by adapting the proof of Theorem~\ref{minimax heavy}.

\subsection{Low-Rank PSD Matrix Recovery}\label{sub lr}
             
We focus on the recovery of low-rank PSD matrices.
Specifically, we investigate the use of the CVX-LS estimator for recovering a rank-$r$ PSD matrix $\X \in \mS^{n}$ and analyze its stable performance under two different observation models.
The observation vector $\y$ is considered under the following two models:
Poisson observation model
\begin{equation}\label{poisson PSD}
	y_{k}\overset{\text{ind.}}{\sim}\text{Poisson}\lk\lg\pp_{k}\pp_{k}^{*},\X \rg\rk,\quad k=1,\cdots,m,
	\end{equation}
	and  heavy-tailed observation model
	\begin{equation}\label{heavy PSD}
	y_{k}=\lg\pp_{k}\pp_{k}^{*},\X \rg+\xi_{k},\quad k=1,\cdots,m,
	\end{equation}
    where $\left\{\xi_k\right\}_{k=1}^m$ are i.i.d., heavy-tailed noise variables.
We recall that the CVX-LS estimator is given by
 \begin{equation}\label{model4}
		\begin{array}{ll}
			\text{minimize}   & \quad\lV\mA\lk\Z\rk-\y\rV_2\\
			\text{subject to} & \quad \Z\in\mS_{+}^{n},\\				 
		\end{array}
	\end{equation}
    where $\mS_+^n$ denotes the cone of PSD matrices in $\C^{n\times n}$, and $\mA(\Z)$ is the linear measurement operator given by $\mA\lk\Z\rk:=\left\{\lg\pp_k\pp_k^*,\Z\rg\right\}_{k=1}^{m}$.

We present the following theorem for low-rank PSD matrix recovery under the Poisson observation model~\eqref{poisson PSD}.
	 
	 \begin{theorem}\label{lr_poisson}
	 Let $\X$ be a rank-$r$ PSD matrix.
    Suppose that $\left\{\pp_k\right\}_{k=1}^{m}$ satisfy Assumption~\ref{sample}, and the observations follow the Poisson model in~\eqref{poisson PSD}.
     Then there exist some universal constants $L,\widetilde{L},c_1,c_2,C_1,C_2 > 0$ dependent only on $K$ and $\mu$ such that the following holds:
	 	\begin{itemize}
		\item[$\mathrm{(a)}$]  
If $m\ge L rn$, then with probability at least $1-\mathcal{O}\lk e^{-c_1 rn}\rk$, the CVX-LS estimator satisfies, simultaneously for all rank-$r$ PSD matrices $\X$, the following estimate:
 \begin{eqnarray}\label{dist2_lr}
\lV\Z_{\star}-\X\rV_{F}
\le C_1 \max\left\{1, K\sqrt{\lV\X\rV_{*}}\right\}\cdot \sqrt{\frac{rn}{m}}.
\end{eqnarray}

		\item[$\mathrm{(b)}$]
        Let $\Gamma^r:=\left\{\X\in\mS_+^n:\lV\X\rV_*\le\frac{1}{K^2}\right\}$.
		 If $m\ge \widetilde{L} rn$, then with probability at least $1-\mathcal{O}\lk \frac{\log^4m}{m}\rk-\mathcal{O}\lk e^{-c_2 rn}\rk$, the CVX-LS estimator satisfies, simultaneously for all rank-$r$ PSD matrices $\X\in\Gamma^r$, the following estimate:
 \begin{eqnarray}
		\lV\Z_{\star}-\X\rV_{F}\le C_2 K^{1/2}\lV\X\rV_*^{1/4}\cdot\sqrt{\frac{rn}{m}}.
					\end{eqnarray}	
		\end{itemize}
\end{theorem}

Theorem~\ref{lr_poisson} states that, in the high-energy regime ($\lV\X\rV_{*}\ge\frac{1}{K^2}$),
the CVX-LS estimator achieves the error bound {\small$\mO\lk \sqrt{\lV\X\rV_{*}}\cdot\sqrt{\frac{rn}{m}}\rk$}.
In the low-energy regime ($\lV\X\rV_{*}\le\frac{1}{K^2}$), it yields {\small$\mO\lk\lV\X\rV_*^{1/4}\cdot\sqrt{\frac{rn}{m}}\rk$}, which decreases as the nuclear norm of $\X$ diminishes.
Although related work, such as \cite{cao2015poisson,mcrae2021low} on matrix completion and \cite{zhang2021low} on tensor completion with Poisson observations, has achieved notable advances, differences in problem formulation render their results not directly comparable to ours.

We then state the following theorem, which characterizes the recovery of low-rank PSD matrices under the heavy-tailed observation model \eqref{heavy PSD}.

	 \begin{theorem}\label{lr_heavy}
	 Let $\X$ be a rank-$r$ PSD matrix.
     Suppose that $\left\{\pp_k\right\}_{k=1}^{m}$ satisfy Assumption~\ref{sample} and the observations follow the heavy-tailed model in~\eqref{heavy PSD} where $\left\{\xi_k\right\}_{k=1}^{m}$ satisfy the conditions in Assumption~\ref{noise0} $\mathrm{(b)}$ with $q>2$.
		 Then there exist universal constants $L,c,C> 0$ dependent only on $K,\mu$ and $q$ such that when provided that 
         $m\ge Lrn$, with probability at least 
          \begin{eqnarray}
          1-\mathcal{O}\lk m^{\lk q/2-1\rk}\log^{q} m\rk-\mathcal{O}\lk e^{-crn} \rk,
          \end{eqnarray}
         simultaneously for all rank-$r$ PSD matrices $\X$, the estimates obtained from the CVX-LS estimator satisfy
 \begin{eqnarray}\label{dist3_lr}
\lV\Z_{\star}-\X\rV_{F}
\le C \lV\xi\rV_{L_q}\cdot \sqrt{\frac{rn}{m}}.
\end{eqnarray}
\end{theorem}

Theorem~\ref{lr_heavy} shows that the CVX-LS estimator achieves the minimax optimal error bound {\small$\mO\lk\lV\xi\rV_{L_q}\cdot \sqrt{\frac{rn}{m}}\rk$}, matching the minimax lower bounds derived in \cite{candes2011tight,cai2015rop}.
Previous work, such as \cite{lecue2018regularization,fan2021shrinkage} addressed low-rank matrix recovery under heavy-tailed noise via LS-type estimators, attaining bounds comparable to ours—the former through regularization and the latter via a shrinkage mechanism to mitigate the effect of heavy-tailed observations.
Similarly, \cite{yu2024low} studied a related problem using robust estimation with the Huber loss and obtained comparable performance. 
In contrast, our CVX-LS estimator requires neither regularization nor data preprocessing, yet still achieves minimax optimal guarantees, thereby offering a conceptually simpler and more direct optimization procedure.
Investigations of low-rank matrix recovery under heavy-tailed noise in various problem settings have also been conducted in \cite{elsener2018robust,wang2025robust,shen2025computationally}.

\subsection{Random Blind Deconvolution}\label{sub bd}

We consider a special case of random blind deconvolution.
Suppose we aim to recover a pair of unknown signals $\x, \pmb{h} \in \C^{n}$ from a collection of $m$ nonlinear measurements given by 
\begin{equation}\label{heavy BD} 
y_k = \pmb{b}_k^* \x \h^* \pmb{a}_k + \xi_k, \quad k = 1, \dots, m, 
\end{equation} 
where $\left\{\pmb{a}_k\right\}_{k=1}^{m}$ and $\left\{\pmb{b}_k\right\}_{k=1}^{m}$ are known sampling vectors, and $\left\{\xi_k\right\}_{k=1}^{m}$ denotes the additive noise.
The goal is to accurately recover both $\x$ and $\pmb{h}$ from the bilinear measurements in \eqref{heavy BD}. 
This problem of solving bilinear systems arises in various domains, with blind deconvolution being a particularly notable application \cite{ahmed2013blind,ling2015self}.

To address the non-convexity inherent in the problem, a popular strategy is to lift the bilinear system to a higher-dimensional space.
Specifically, we consider the following constrained LS estimator: 
\begin{equation} \label{model5}
\begin{aligned}
	\underset{\Z \in \C^{n \times n}}{\text{minimize}} 
	&\quad \lV \mathcal{B}\lk\Z\rk -\y\rV_2 \\
	\text{subject to} 
	&\quad \lV \Z \rV_* \le \lV\x\rV_2\cdot\lV\h\rV_2,
\end{aligned}
\end{equation}
where $\mathcal{B}\lk\Z\rk$ is the linear measurement operator $\mathcal{B}\lk\Z\rk:=\left\{\lg\pmb{a}_k\pmb{b}_k^*,\Z\rg\right\}_{k=1}^{m}$, and $\lV\x\rV_2\cdot\lV\h\rV_2$ is the nuclear norm of $\x\h^*$. 
We consider the setting in which both $\left\{\pmb{a}_k\right\}_{k=1}^{m}$ and $\left\{\pmb{b}_k\right\}_{k=1}^{m}$ are random sub-Gaussian sampling vectors \cite{cai2015rop,charisopoulos2021low,chen2023convex}, while the observations $\y:=\left\{y_k\right\}_{k=1}^{m}$ are contaminated by heavy-tailed noise $\left\{\xi_k\right\}_{k=1}^{m}$.
Another common setting considers $\left\{\pmb{a}_k\right\}_{k=1}^{m}$ as random Gaussian sampling vectors, while $\left\{\pmb{b}_k\right\}_{k=1}^{m}$ consists of the first $n$ columns of the unitary discrete Fourier transform (DFT) matrix $\pmb{F}\in\C^{m\times m}$ obeying $\pmb{F}\pmb{F}^*=\pmb{I}_m$ \cite{li2019rapid,ma2019implicit,krahmer2021convex,chen2023convex,kostin2025robust}; this setting is beyond the scope of the present work.

The following theorem establishes the performance of the constrained LS estimator~\eqref{model5} under heavy-tailed noise.

\begin{theorem}\label{thm:bd}
Suppose that $\left\{\pmb{a}_k\right\}_{k=1}^{m}$ and $\left\{\pmb{b}_k\right\}_{k=1}^{m}$ are all independent copies of a random vector $\pp\in\C^n$ whose entries $\left\{\varphi_j\right\}_{j=1}^{n}$ are i.i.d., mean 0, variance 1, and $K$-sub-Gaussian,
and the noise term $\left\{\xi_k\right\}_{k=1}^{m}$ in~\eqref{heavy BD} satisfies the conditions in Assumption \ref{noise0} $\mathrm{(b)}$ with $q>2$.
Then there exist universal constants $L,c,C> 0$ dependent only on $K$ and $q$ such that when provided that $m\ge Ln$, with probability at least 
\begin{eqnarray*}
1-\mathcal{O}\lk m^{-\lk q/2-1\rk}\log^{q} m\rk-\mathcal{O}\lk e^{-cn}\rk,
\end{eqnarray*}
simultaneously for all $\x,\h\in\C^n$, the output $\Z_\star$ of the constrained LS estimator satisfies
\begin{eqnarray}
\lV\Z_\star- \x\h^*\rV_F \le C \lV\xi\rV_{L_q}\cdot\sqrt{\frac{n}{m}}.
\end{eqnarray}
\end{theorem}

Theorem~\ref{thm:bd} shows that the constrained LS estimator achieves the error bound
{\small$\mO\lk\lV\xi\rV_{L_q}\cdot \sqrt{\frac{n}{m}}\rk$}.
This rate is optimal up to a logarithmic factor, as implied by the minimax lower bound established in \cite{chen2023convex}. 
Compared to the estimation results in \cite[Theorem 3]{chen2023convex}, 
Theorem~\ref{thm:bd} extends the noise model from sub-Gaussian to heavy-tailed distributions and reduces the required number of samples from $m=\mO\lk n\log^6 m\rk$ to the optimal $m=\mO\lk n\rk$, while also improving the estimation error.

\section{Discussion}\label{discussion}

This paper investigates the stable performance of the NCVX-LS and CVX-LS estimators for phase retrieval in the presence of Poisson and heavy-tailed noise.
We have demonstrated, that both estimators achieve the minimax optimal rates in the high-energy regime for these two noise models. 
In the Poisson setting, the NCVX-LS estimator further achieves an error rate that decreases with the signal energy in the low-energy regime, remaining optimal with respect to the oversampling ratio. 
Similarly, in the heavy-tailed setting, the NCVX-LS estimator achieves a minimax optimal rate in the low-energy regime.
We also extend our analysis framework to some related problems, including sparse phase retrieval, low-rank PSD matrix recovery, and random blind deconvolution.

Moving forward, our findings suggest several directions for further investigation.
For the Poisson model \eqref{poisson}, the gap in the low-energy regime between our upper bound for both the NCVX-LS estimator and the minimax lower bound $\Omega\lk\sqrt{\lV\x\rV_2}\cdot\lk\frac{n}{m}\rk^{1/4}\rk$ could potentially be closed. 
Our analysis suggests that employing robust estimators capable of handling heavy-tailed noise with a finite 
 $L_2$-norm rather than a finite $L_4$-norm would allow this gap to be closed.
Moreover, developing efficient algorithms to compute the NCVX-LS estimator and achieve the optimal error rate in the low-energy regime also represents a promising research direction.
For the heavy-tailed model~\eqref{heavy}, an interesting question is whether optimal error rates can be achieved when the noise has only a finite $q$-th moment ($1 \le q \le 2$) or even no finite expectation. 
Addressing this case may require additional assumptions on the noise (e.g., symmetry or structural properties), as well as robust estimators or suitable data preprocessing.
Furthermore, beyond sub-Gaussian sampling, it would be of interest to extend the current analysis to more realistic measurement schemes, such as coded diffraction patterns (CDP) or short-time Fourier transform (STFT) sampling.
We leave these questions for future work.

\section*{Acknowledgments}
G.H. was supported by the Qiushi Feiying Program of Zhejiang University. 
This work was carried out while he was a visiting PhD student at UCLA. 
S.L. was supported by NSFC under grant number U21A20426. 
D.N. was partially supported by NSF DMS 2408912.

\appendix

\section{Auxiliary Proofs}\label{app dis}

\subsection{Proof of Proposition \ref{dis1}}\label{app dis1}

 We choose $\varphi_0:=\mbox{Phase}\lk\z_\star^{*} \x\rk$ and set $\widetilde{\x}:=e^{i\varphi_0}\x$, then  $\lg\z_\star^{*},\widetilde{\x}\rg\ge 0$ 
 and we have 
 \begin{eqnarray*}
 \begin{aligned}
 \textbf{dist}^2\lk\z_\star^{*},\x\rk&=\min_{\varphi\in\lz0,2\pi\rk}\lV e^{i\varphi}\z_\star-\x\rV_{2}^{2}\\
 &=\lV e^{i\varphi_0}\z_\star-\x\rV_{2}^{2}=\lV\z_\star\rV^{2}_{2}+\lV\widetilde{\x}\rV^{2}_{2}-2\lg\z_\star,\widetilde{\x}\rg.
 \end{aligned}
  \end{eqnarray*} 
We also obtain that
\begin{eqnarray*}
\begin{aligned}
\lV\z_\star\z_\star^{*}-\x\x^{*}\rV^{2}_{F}
&=\lV\z_\star\rV^{4}_{2}+\lV\x\rV^{4}_{2}-2\lv\lg\z_\star,\x\rg\rv^{2}
=\lV\z_\star\rV^{4}_{2}+\lV\widetilde{\x}\rV^{4}_{2}-2\lv\lg\z_\star,\widetilde{\x}\rg\rv^{2}\\
&=\lk\sqrt{\lV\z_\star\rV^{4}_{2}+\lV\widetilde{\x}\rV^{4}_{2}}-\sqrt{2}\lg\z_\star,\widetilde{\x}\rg\rk\cdot\lk\sqrt{\lV\z_\star\rV^{4}_{2}+\lV\widetilde{\x}\rV^{4}_{2}}+\sqrt{2}\lg\z_\star,\widetilde{\x}\rg\rk\\
&\ge\frac{1}{2}\lk\lV\z_\star\rV^{2}_{2}+\lV\widetilde{\x}\rV^{2}_{2}-2\lg\z_\star,\widetilde{\x}\rg\rk\cdot\lk\lV\z_\star\rV^{2}_{2}+\lV\widetilde{\x}\rV^{2}_{2}+2\lg\z_\star,\widetilde{\x}\rg\rk\\
&\ge\frac{1}{4}\textbf{dist}^2\lk\z_\star,\x\rk\cdot\lk\lV\z_\star\rV_{2}+\lV\widetilde{\x}\rV_{2}\rk^{2}.
\end{aligned}
\end{eqnarray*}
In the third and fourth lines, we have used the Cauchy-Schwarz inequality.
Since 
\begin{eqnarray*}
\lk\lV\z_\star\rV_{2}+\lV\widetilde{\x}\rV_{2}\rk^{2}\ge\max\left\{\textbf{dist}^2\lk\z_\star,\x\rk,\lV\x\rV_{2}^{2}\right\},
\end{eqnarray*}
we have finished the proof.

\subsection{Proof of Proposition~\ref{one}}\label{alg}

Let $\M\in\mathcal{E}_{\text{cvx}}$, by the definition of $\mathcal{E}_{\text{cvx}}$, we can find a rank-$1$ matrix $\x\x^* \in \mS^{n}_{+}$ such that 
\begin{equation}\label{PSD}
\x\x^* + \M \in \mathcal{S}^n_{+}.
\end{equation}
Suppose now by contradiction that $\M$ has $2$ (strictly) negative eigenvalues with corresponding eigenvectors $ \pmb{z}_1, \pmb{z}_{2} \in \mathbb{C}^n$.  
We can find a vector $\xu \in \text{span} \left\{ \pmb{z}_1,\pmb{z}_{2} \right\} \backslash \left\{ 0 \right\}$ such that $ \langle  \xu,\x \rangle =0 $. 
This implies that we have
\begin{equation*}
\xu \lk  \x\x^* + \M \rk \xu^* =  \xu^*\M\xu < 0,
\end{equation*} 
which is a contradiction to \eqref{PSD}.

\subsection{Proof of Proposition~\ref{pro low-rank}}\label{proof:app low-rank}

The proof of \textbf{Part}~$\mathrm{(a)}$ follows from the observation that the elements in $\mathcal{E}_{\text{ncvx}}$ have a rank at most 2.
For \textbf{Part}~$\mathrm{(b)}$, as every element $\M\in\mathcal{E}_{\text{cvx,1}}$ satisfies 
\begin{eqnarray*}
 \frac{1}{2}  \sum_{i=1}^{n-1} \lambda_i \lk\M\rk< -\lambda_n \lk\M\rk ,
\end{eqnarray*}
we have that
       \begin{eqnarray*}
\lV\M\rV_{*}=\sum_{i=1}^{n-1} \lambda_i \lk\M\rk -\lambda_n \lk\M\rk
                  \le -3 \lambda_n \lk\M\rk\le 3\lV\M\rV_{F}.
\end{eqnarray*}

\subsection{Proof of Proposition \ref{small ball function}}\label{proof:small ball function}

By the Paley–Zygmund inequality (see e.g., \cite{de2012decoupling}), we have that for any $\M\in\mS^n$,
\begin{equation*}
     \mathbb{P}\lk\lv\pp^*\M\pp\rv^2\ge\frac{\E\lv\pp^*\M\pp\rv^2}{2}\rk
     \ge\frac{\lk\E\lv\pp^*\M\pp\rv^2\rk^2}{\E\lv\pp^*\M\pp\rv^4}.
 \end{equation*}
By Lemma 9 in \cite{krahmer2020complex} and $\E\lk\pp^2\rk=0$, we can obtain for any $\M\in\mS^n$,
\begin{equation}\label{eq:E^2}
\begin{aligned}
  \E\lv\pp^*\M\pp\rv^2&=\lk\text{Tr}\lk\M\rk\rk^2+\lz\E\lk\lv\pp\rv^4\rk-1\rz\sum_{i=1}^n\M^2_{i,i}+\sum_{i\neq j}\lv\M_{i,j}\rv^2\\
  &\ge \lk\text{Tr}\lk\M\rk\rk^2+\min\left\{\mu,1\right\}\cdot\lV\M\rV_F^2.
  \end{aligned}
 \end{equation}
 The second line follows from $\E\lk\lv\pp\rv^4\rk\ge1+\mu$.
 Setting $q=4, m=1$ in Lemma~\ref{geometry}, we obtain
 \begin{equation*}
 \lV\pp^*\M\pp-\E\pp^*\M\pp\rV_{L_4}\lesssim K^2\lV\M\rV_F.
 \end{equation*}
 Therefore, the triangle inequality yields that
 \begin{equation}\label{eq:E^4}
 \begin{aligned}
\E\lv\pp^*\M\pp\rv^4&\lesssim\E\lv\pp^*\M\pp-\E\pp^*\M\pp\rv^4+\lk\E\pp^*\M\pp\rk^4\\
&\lesssim K^8\lV\M\rV^4_F+\lk\text{Tr}\lk\M\rk\rk^4,
\end{aligned}
 \end{equation}
 where we have used $\E\pp^*\M\pp=\text{Tr}\lk\M\rk$.
Hence, for $0<u\le\sqrt{\frac{\min\left\{\mu,1\right\}}{2}}$, we have
 \begin{equation*}
 \begin{aligned}
 \mathcal{Q}_{u}\lk\mM ;\pp\pp^{*}\rk&\ge
 \inf_{\M\in\mM}\mathbb{P}\lk\lv\pp^*\M\pp\rv^2\ge\frac{\E\lv\pp^*\M\pp\rv^2}{2}\rk \\
 &\gtrsim\frac{\min\left\{\mu^2,1\right\}\cdot\lV\M\rV_F^4+\lk\text{Tr}\lk\M\rk\rk^4}{K^8\lV\M\rV_F^4+\lk\text{Tr}\lk\M\rk\rk^4}\\
 &\ge\frac{\min\left\{\mu^2,1\right\}}{ K^8+1}.
 \end{aligned}
 \end{equation*}
 In the first inequality, we have used $\lV\M\rV_F=1$ and \eqref{eq:E^2}, and in the second inequality we have used \eqref{eq:E^2} and \eqref{eq:E^4}.
 
\subsection{Proof of Proposition \ref{psi_1 of poisson}}\label{psi_1}
We record some facts that will be used.

\begin{fact}\label{fact1}
For $x\in\lz0,\frac{1}{2}\rz$, we have $\frac{1}{1-x}\le e^{2x}$.
\end{fact}

\begin{fact}\label{fact2}
Let $f\lk x\rk=\frac{e^x-1-x}{x^2}$. 
Then $f\lk x\rk$ is monotonically increasing on $\R$.
\end{fact}

\begin{fact}\label{fact3}
Let $Z\sim \text{Poisson}\lk\lambda\rk$.
The moment generating function of $Z$ is
$$M_{Z}\lk t\rk=e^{\lambda\lk e^t-1\rk}.$$
\end{fact}

\begin{fact}\label{fact4}
There exists a constant $C_0\ge 1$ such that
\begin{eqnarray*}
\lV  \pp^*\x\rV_{\psi_2}\le C_0 K\lV\x\rV_2.
\end{eqnarray*}
\end{fact}

Fact~\ref{fact1} and Fact~\ref{fact2} can be verified by differentiation; 
Fact~\ref{fact3} follows from the probability density function of the Poisson distribution;
Fact~\ref{fact4} follows directly from Lemma 3.4.2 in \cite{vershynin2018high}.
We omit the details here.

We denote $X= \lv\pp^*\x\rv$ and then $\xi=\text{Poisson}\lk X^{2}\rk-X^2$.
Clearly, we have $\E\lk \xi\rk=0$.
By Fact~\ref{fact4} and Proposition 2.5.2 in \cite{vershynin2018high}, for any $p\ge 1$ we have
 \begin{eqnarray}\label{E^p}
 \E\lv X\rv^p\le \lk C_0K\lV\x\rV_2\sqrt{p}\rk^p.
  \end{eqnarray}
Given that  $\xi\mid X=\lambda\sim\text{Poisson}\lk \lambda^{2}\rk-\lambda^2$, Fact~\ref{fact3} yields
\begin{eqnarray*}
\E \lk e^{\theta \xi}\mid X=\lambda\rk= e^{\lk e^\theta-1-\theta\rk \lambda^2}:= e^{ g\lk\theta\rk \lambda^2}.
\end{eqnarray*}
Therefore, applying the law of total expectation and using Taylor expansion, we obtain
 \begin{equation}\label{EY}
 \begin{aligned}
 \E \lk e^{\theta \xi}\rk&=\E\lk e^{ g\lk\theta\rk X^2}\rk=1+\sum_{p=1}^{\infty}\frac{g\lk\theta\rk^p\E\lk X^{2p}\rk}{p!}\\
  &\le 1+ \sum_{p=1}^{\infty}\frac{g\lk\theta\rk^pC_0^{2p}K^{2p}\lV\x\rV_2^{2p}\lk 2p\rk^p}{p!}\\
  &\le 1+ \sum_{p=1}^{\infty}\frac{g\lk\theta\rk^pC_0^{2p}K^{2p}\lV\x\rV_2^{2p}\lk 2p\rk^p}{\lk\frac{p}{e}\rk^p}\\
  &=1+\sum_{p=1}^{\infty}\lz 2e g\lk\theta\rk C_0^2 K^2 \lV\x\rV_2^{2} \rz^p\\
  &=\frac{1}{1-2e g\lk\theta\rk C_0^2 K^2 \lV\x\rV_2^{2}}\\
  &\le e^{4e g\lk\theta\rk C_0^2 K^2 \lV\x\rV_2^{2}}
 \end{aligned}
 \end{equation}
provided $2e g\lk\theta\rk C_0^2 K^2 \lV\x\rV_2^{2}\le\frac{1}{2}$.
Here, in the second line we have used~\eqref{E^p}, the third line employs the inequality $\lk\frac{p}{e}\rk^p\le p!$, and in the last line we invoke Fact~\ref{fact1}.
 
To bound the sub-exponential norm of $\xi$, we apply Proposition 2.7.1 from \cite{vershynin2018high}, which requires identifying a sufficiently small constant $T_0$ such that
\begin{equation*}
  \E \lk e^{\theta \xi}\rk\le e^{T_0^2\theta^2},\quad \forall \lv\theta\rv\le\frac{1}{T_0}.
\end{equation*}
 By~\eqref{EY}, this condition is satisfied if
  \begin{equation}\label{eq:g(theta)}
  4e g\lk\theta\rk C_0^2 K^2 \lV\x\rV_2^{2}\le T_0^2\theta^2,\quad \forall \lv\theta\rv\le\frac{1}{T_0}.
  \end{equation}
By Fact~\ref{fact2}, $\frac{g\lk\theta\rk}{\theta^2}$ is monotonically increases on $\lz-\frac{1}{T_0},\frac{1}{T_0}\rz$, thus \eqref{eq:g(theta)} holds if
\begin{equation*}
\frac{g\lk 1/T_0\rk}{\lk1/T_0\rk^2}\cdot\frac{4eC_0^2K^2\lV\x\rV_2^2}{T_0^2}=\sum_{p=0}^{\infty}\frac{1}{T_0^p\lk p+2\rk!}\cdot\frac{4eC_0^2K^2\lV\x\rV_2^2}{T_0^2}\le1.
\end{equation*}
We finish the proof by choosing $T_0=\max\left\{2,2\sqrt{e}C_0K\lV\x\rV_2\right\}$.

\subsection{Proof of Proposition \ref{L_q of poisson}}\label{L_q}

Recall that $X= \lv\pp^*\x\rv$. 
Conditioned on $X=\lambda$, then we obtain
 \begin{eqnarray}\label{EY^4}
 \begin{aligned}
 \E\lk\lv \xi\rv^4\mid X=\lambda\rk&=\E\lk\lv \text{Poisson}\lk \lambda^{2}\rk-\lambda^2\rv^4\rk\\
 &=\E\lk \text{Poisson}\lk \lambda^{2}\rk^4\rk-4\lambda^2\E\lk \text{Poisson}\lk \lambda^{2}\rk^3\rk\\
 &\quad +6\lambda^4\E\lk \text{Poisson}\lk \lambda^{2}\rk^2\rk
  -4\lambda^6\E\lk \text{Poisson}\lk \lambda^{2}\rk\rk+\lambda^8.
 \end{aligned}
  \end{eqnarray}
By direct calculation, we have
\begin{eqnarray*}
\mathbb{E}\left( \text{Poisson}\lk\lambda^{2}\rk \right) &=& \lambda^2, \\
\mathbb{E}\left( \text{Poisson}\lk\lambda^{2}\rk^2 \right) &=& \lambda^2 + \lambda^4, \\
\mathbb{E}\left( \text{Poisson}\lk\lambda^{2}\rk^3 \right) &=& \lambda^2 + 3\lambda^4 + \lambda^6, \\
\mathbb{E}\left( \text{Poisson}\lk\lambda^{2}\rk^4 \right) &=& \lambda^2 + 7\lambda^4 + 6\lambda^6 + \lambda^8.
\end{eqnarray*}
Substitute the above equations into \eqref{EY^4}, we have that
\begin{eqnarray*}
\E\lk\lv \xi\rv^4\mid X=\lambda\rk=\lambda^2+3\lambda^4.
\end{eqnarray*}
Now, by the law of total expectation and \eqref{E^p}, we obtain
 \begin{eqnarray*}
\E\lk\lv\xi\rv^4\rk=E\lk X^2\rk+3E\lk X^4\rk\le\lk \sqrt{2} C_0K\lV\x\rV_2\rk^2+3\lk2 C_0K\lV\x\rV_2\rk^4.
\end{eqnarray*}
Finally, we could further bound that
\begin{eqnarray*}
\lV\xi\rV_{L_4}\le\lk \sqrt{2} C_0K\lV\x\rV_2\rk^{1/2}+3 C_0K\lV\x\rV_2\lesssim\max\left\{\lk K\lV\x\rV_2\rk^{1/2},K\lV\x\rV_2\right\}.
\end{eqnarray*}

\section{Proof of Lemma~\ref{up for KL}}\label{pf_KL}

Our analysis primarily follows the approach in \cite[Lemma 7.1]{chen2017solving}, with some refinements. 
We first prove \textbf{Part}~$(\mathrm{a})$, while \textbf{Part}~$(\mathrm{b})$ and \textbf{Part}$~(\mathrm{c})$ follow by similar arguments.
We begin by constructing a set $\mathcal{T}_{1}$ that satisfying \eqref{distance} in \textbf{Part}~$(\mathrm{a})$, with exponentially many vectors near $\x$ that are approximately equally separated.
The construction of $\mathcal{T}_{1}$ follows a standard random packing argument. 
Specifically, let
\begin{eqnarray*}
	\z=\left[z_{1},\cdots,z_{n}\right]^{\top},\quad z_{l}=x_{l}+\frac{1}{\sqrt{2n}}g_{l},\quad1\le l\le n,
\end{eqnarray*}
where $g_{l} \overset{\text{ind.}}{\sim} \mathcal{N}\left(0,1\right)$.
The set $\mathcal{T}_{1}$ is then obtained by generating $T_{1}=\exp\left(\frac{n}{20}\right)$
independent copies $\z^{(i)}$ ($1\le i\le T_{1}$) of $\z$.
For all $\z^{(i)},\z^{(j)}\in\mathcal{T}_{1}$, concentration inequality (see \cite[Theorem 5.1.4]{vershynin2018high}), together with a union bound over all $\binom{T{1}}{2}$ pairs, imply that
\begin{eqnarray}\label{eq:dis}
	\begin{array}{ll}
	1/2-n^{-1/2} &\le\lV \z^{(i)}-\z^{(j)}\rV_2 \le 3/2+n^{-1/2},\quad\ \forall i\neq j\\
	1/{\sqrt{8}}- (2n)^{-1/2} &\le\lV \z^{(i)}-\x\rV_2 \leq 3/\sqrt{8}+ (2n)^{-1/2}, \quad1\le i\leq T_{1}
	\end{array}
\end{eqnarray}
with probability at least $1-2\exp\left(-\frac{n}{40}\right)$. 

We then show that many vectors in $\mathcal{T}_{1}$ satisfy \eqref{KL for poisson1} in \textbf{Part}$~(\mathrm{a})$.
By the rotation invariance of Gaussian vectors, we may assume without loss of generality that $\x=\left[a,0,\cdots,0\right]^{\top}$
for some $a>0$.
For any given $\z$ with $\pmb{r}:=\z-\x$,
by letting $\pp_{\perp}:=\left[\varphi_{2},\cdots,\varphi_{n}\right]^{\top}$, and $\pmb{r}_{\perp}:=\left[r_{2},\cdots,r_{n}\right]^{\top}$,
we derive
\begin{align}
	\frac{ |\pp^{\top}\pmb{r} |^{2}}{ |\pp^{\top}\x |^{2}}  
	\le  \frac{2|\varphi_{1}r_{1}|^{2}+2|\pp_{\perp}^{\top}\pmb{r}_{\perp}|^{2}}{\left|\varphi_{1}\right|^{2}\lV \x\rV_2 ^{2}}
	\le \frac{2\lV\pmb{r}\rV_2^{2}}{\lV \x\rV_2 ^{2}}+\frac{2|\pp_{\perp}^{\top}\pmb{r}_{\perp}|^{2}}{\left|\varphi_{1}\right|^{2}\lV \x\rV^{2}}.
	\label{eq:key-quantity}
\end{align}
Our analysis next focuses on deriving an upper bound for
$\frac{2|\pp_{\perp}^{\top}\pmb{r}_{\perp}|^{2}}{\left|\varphi_{1}\right|^{2}}$.
The motivation for the above decomposition is that $|\pp_{\perp}^{\top}\pmb{r}_{\perp}|^{2}$ and
$\left|\varphi_{1}\right|^{2}$ are independent, which makes the ratio more convenient to handle. 
Before we proceed with our analysis, we present two facts on the magnitudes of $\pp_{k}^{\top}\x$ ($1\leq k\leq m$).

\begin{fact}\label{fact5}
For any given $\x$ and any sufficiently large $m$, with probability at least $1-\frac{2}{\log m}$,
\begin{eqnarray*}\label{eq:LB}
\min_{1\le k\leq m}\lv\pp_{k}^{\top}\x\rv\geq\frac{1}{ m\log m}\lV\x\rV_2.
\end{eqnarray*}
\end{fact}

\begin{proof}
We have that
\begin{align*}
	 \mathbb{P}\left\{ \min_{1\le k\leq m}\lv\pp_{k}^{\top}\x\rv\geq\frac{1}{ m\log m}\lV\x\rV_2 \right\}  
	 &= \left(\mathbb{P}\left\{ \lv\pp_{k}^{\top}\x\rv \geq \frac{1}{m\log m} \lV \x \rV_2
	\right\} \right)^{m} \\
	&\ge\left(1-\frac{2}{\sqrt{2\pi}}\frac{1}{m\log m}\right)^{m}\\
    &\ge e^{-\frac{2}{\log m}}\ge 1-\frac{2}{\log m}.
\end{align*}
\end{proof}

\begin{fact}\label{fact6}
For any given $\x$, with probability at least $1-\exp\lk- \Omega\Big( \frac{n^2} {m \log^2 m }  \Big) \rk$,
\begin{eqnarray*}\label{eq:UB}
\sum_{k=1}^{m}\mathbbm{1}_{\left\{ \lv\pp_{k}^{\top}\x\rv\le\frac{n\lV \x\rV_2 }{40m\log m}\right\}  }> \frac{n}{25\log m} := t_0.
\end{eqnarray*}
\end{fact}

\begin{proof}
Since 
\begin{eqnarray*}
\mathbb{E}\left[\mathbbm{1}_{\left\{ \lv\pp_{k}^{\top}\x\rv\le\frac{n\lV \x\rV_2 }{40m\log m}\right\} }\right]
\le\frac{2}{\sqrt{2\pi}}\frac{n}{40m\log m}
\le\frac{n}{25m\log m},
\end{eqnarray*}
by Hoeffding inequality \cite[Theorem 2.6.2]{vershynin2018high}, we have 
\begin{eqnarray*}
 	&& \mathbb{P}\bigg\{ \sum_{k=1}^{m}\mathbbm{1}_{\left\{ \lv\pp_{k}^{\top}\x\rv\le\frac{n\lV \x\rV_2 }{10m\log m}\right\}  }> \frac{n}{25\log m} \bigg\} \\
 	&& \quad \le \mathbb{P}\bigg\{ \frac{1}{m}\sum_{k=1}^{m}\left(\mathbbm{1}_{\left\{ \lv\pp_{k}^{\top}\x\rv\le\frac{n\lV \x\rV_2 }{40m\log m}\right\} }
		- \mathbb{E}\left[\mathbbm{1}_{\left\{ \lv\pp_{k}^{\top}\x\rv\le\frac{n\lV \x\rV_2 }{40m\log m}\right\}}
        \right]\right)>\frac{n}{50m\log m}\bigg\} \\
	&& \quad \le \exp\lk- \Omega\Big( \frac{n^2} {m \log^2 m }  \Big) \rk.
\end{eqnarray*}
\end{proof}

To simplify presentation, we reorder $\{\pp_{k}\}_{k=1}^m$ such that
\begin{eqnarray*}
	(m\log m)^{-1}\lV \x\rV_2\le\left|\pp_{1}^{\top}\x\right|\le\left|\pp_{2}^{\top}\x\right|\leq\cdots\le\left|\pp_{m}^{\top}\x\right|.
\end{eqnarray*}
In the sequel we construct hypotheses conditioned on the events in Fact~\ref{fact5} and Fact~\ref{fact6}.
To proceed, let $\pmb{r}_{\perp}^{(i)}$ denote the vector obtained by removing the first entry of $\z^{(i)}-\boldsymbol{x}$, 
and introduce the indicator variables
\begin{equation}
	\xi_{k}^{i}:=\begin{cases}
	\mathbbm{1}_{\left\{ \left|\pp_{k,\perp}^{\top}\pmb{r}_{\perp}^{(i)}\right|\leq\frac{1}{m}\sqrt{\frac{n-1}{2n}}\right\} },\quad & 1\le k\leq t_0,\\
	\mathbbm{1}_{\big\{ \left|\pp_{k,\perp}^{\top}\pmb{r}_{\perp}^{(i)}\right|\leq\sqrt{\frac{2\left(n-1\right)\log m}{n}}\big\} }, & k>t_0,
\end{cases}
	\label{eq:ind_xi}
\end{equation}
where $t_0=\frac{n}{25\log m}$ as before. 
The idea behind dividing $\left\{\pp_{k,\perp}^{\top}\pmb{r}_{\perp}^{(i)}\right\}_{i=1}^m$ into two groups is that, by Fact~\ref{fact5}, it becomes more difficult to upper bound $\frac{|\pp_{k,\perp}^{\top}\pmb{r}^{(i)}_{\perp}|^{2}}{\left|\varphi_{k,1}\right|^{2}}$ when $\lv\varphi_{k,1}\rv$ is small.
Therefore, in this case, we should impose a stricter control on $|\pp_{k,\perp}^{\top}\pmb{r}^{(i)}_{\perp}|$.  

For any $\z^{(i)}\in\mathcal{T}_1$, the indicator variables in \eqref{eq:ind_xi} obeying $\prod\limits_{k=1}^m\xi_k^i=1$ ensure \textbf{Part}~$(\mathrm{b})$ when $n$ is sufficiently large.
To see this, note that for the first group of indices,  by $\xi_k^i=1$ and \eqref{eq:dis} one has
\begin{equation*}
	\left|\pp_{k,\perp}^{\top}\pmb{r}_{\perp}^{(i)}\right|\le\frac{1}{m}\sqrt{\frac{n-1}{2n}} \le\frac{3}{m}\lV\pmb{r}^{(i)}\rV,
	\quad1\le k\le t_0,
\end{equation*}
This taken collectively with \eqref{eq:key-quantity} and Fact \ref{fact5} yields
\begin{eqnarray*}
	\frac{\left|\pp_{k}^{\top}\pmb{r}^{(l)}\right|^{2}}{\left|\pp_{k}^{\top}\x\right|^{2}} 
	 \le  \frac{2 \|\pmb{r}^{(i)}\| ^{2}}{\lV \x\rV^{2}} + \frac{\frac{9}{m^{2}}\lV\pmb{r}^{(i)}\rV ^{2}}{\frac{1}{m^2\log^2m}\lV \x \rV_2^2}
	\le \frac{ (2+ 9 \log^2m)\lV \pmb{r}^{(i)}\rV ^{2} }{ \lV \x \rV_2 ^2}  ,\quad1\le k\leq t_0.
\end{eqnarray*}
For the second group of indices, since $\xi_k^i=1$, it follows from \eqref{eq:dis} that
\begin{equation}\label{second}
	\left|\pp_{k,\perp}^{\top}\pmb{r}_{\perp}^{(i)}\right|\le\sqrt{\frac{2\left(n-1\right)\log m}{n}}
	\leq 4\sqrt{\log m}\lV \pmb{r}^{(i)}\rV_2 ,\quad k=t_0+1,\cdots,m,
\end{equation}
Substituting the above inequality together with Fact~\ref{fact6} into \eqref{eq:key-quantity} yields
\begin{eqnarray*}
	\frac{\left|\pp_{k}^{\top}\pmb{r}^{(i)}\right|^{2}}{\left|\pp_{k}^{\top}\x\right|^{2}} 
    & \le & \frac{2 \lV \pmb{r}^{(i)} \rV_2^{2}}{\lV\x\rV_2^{2}} 
	+ \frac{16\lV \pmb{r}^{(i)}\rV ^{2}\log m}{\lV\x\rV_2^{2}n^2/1600m^2\log^{2}m}\\
	&\le&\frac{\lk2+25600\frac{m^2\log^3m}{n^2}\rk\lV\pmb{r}^{(i)}\rV_2^{2}}{\lV \x\rV_2^{2}} ,\quad k\ge t_0+1.
\end{eqnarray*}
Thus, \eqref{KL for poisson1} holds for all $1\le k\le m$.  
It remains to ensure the existence of exponentially many vectors satisfying $\prod\limits_{k=1}^m \xi_k^i = 1$.

The first group of indicators is quite restrictive: for each $k$, only a fraction $O(1/m)$ of the equations satisfy $\xi_k^i=1$. 
Fortunately, since $T_1$ is exponentially large, even $T_1/m^{t_0}$ remains exponentially large under our choice of $t_0=\frac{n}{25\log m}$.
By the calculations in \cite[pp.~871–872]{chen2017solving}, with probability exceeding $1-3\exp\left(-\Omega\left(t_0\right)\right)$,
the first group satisfies
\begin{eqnarray*}
	\sum_{i=1}^{T_{1}}\prod_{k=1}^{t_0}\xi_{k}^{i} 
    & \ge & \frac{1}{2}\frac{T_{1}}{\lk2\pi\rk^{t_0/2}\lk1+4\sqrt{t_0/n}\rk^{t_0/2}}\left(\frac{1}{\sqrt{2\pi}m}\right)^{t_0}\\
    &\ge&\frac{1}{2}T_{1}\frac{1}{\lk e^2m\rk^{t_0}}\\
    &\ge&\frac{1}{2}\exp\lz\left(\frac{1}{20}-\frac{t_0\left(2+\log m\right)}{n}\right)n\rz\\
	&\ge& \frac{1}{2}\exp\left(\frac{1}{100}n\right).
\end{eqnarray*}
In light of this, we define $\mathcal{T}_{2}$ as the collection of all $\z^{(i)}$ satisfying $\prod\limits_{i=k}^{t_0}\xi_{k}^{i}=1$.
Its size is at least $T_{2}\ge\frac{1}{2}\exp\left(\frac{1}{100}n\right)$
based on the preceding argument. 
For notational simplicity, we assume the elements of $\mathcal{T}_{2}$ are indexed as $\z^{(j)}$ for $1\le j\le T_{2}$.

We next turn to the second group, examining how many vectors $\z^{(j)}$ in $\mathcal{T}_{2}$ further satisfy $\prod\limits_{k=t_0+1}^{m}\xi_{k}^{j}=1$. 
The construction of $\mathcal{T}_2$ depends only on $\left\{\pp_k\right\}_{1\le k\le t_0}$ and is independent of the remaining vectors $\left\{\pp_k\right\}_{k>t_0}$.
The following argument is therefore carried out conditional on $\mathcal{T}_2$ and $\{\pp_k\}_{1 \le k \le t_0}$. 
By Bernstein inequality \cite[Theorem 2.8.1]{vershynin2018high}, we obtain 
\begin{eqnarray*}
	\mathbb{P}\left\{ \left|\pp_{k,\perp}^{\top}\pmb{r}_{\perp}^{(j)}\right|>\sqrt{\frac{2\left(n-1\right)\log m}{n}}\right\} \le\frac{2}{m^2}.
\end{eqnarray*}
for sufficiently large $n$.
Then by the union bound, we obtain 
\begin{align*}
 	\mathbb{E}&\left[\sum\limits_{j=1}^{T_{2}}\left(1-\prod\limits_{k=t_0+1}^{m}\xi_{k}^{j}\right)\right]\\
	&\quad= \sum\limits_{j=1}^{T_{2}}\mathbb{P}\bigg\{ \exists k \text{ }(t_0<k\le m):
    \text{ }\left|\pp_{k,\perp}^{\top}\pmb{r}_{\perp}^{(j)}\right|>\sqrt{\frac{2\left(n-1\right)\log m}{n}} \bigg\} \\
 	& \quad\le \text{}\sum_{j=1}^{T_{2}}\sum_{k=t_0+1}^{m}\mathbb{P}\left\{ \left|\pp_{k,\perp}^{\top}\pmb{r}_{\perp}^{(j)}\right|>\sqrt{\frac{2\left(n-1\right)\log m}{n}}\right\}\\
    &\quad\le\text{ }T_2m\frac{2}{m^2}=\frac{2T_{2}}{m}.
\end{align*}
This combined with Markov's inequality gives
\begin{eqnarray*}
	\sum\limits_{j=1}^{T_{2}}\left(1-\prod\limits_{k=t_0+1}^{m}\xi_{k}^{j}\right)\leq\frac{\log m}{m}\cdot T_{2}
\end{eqnarray*}
with probability $1- \frac{1}{\log m}$. 
The above inequalities implies that for sufficiently large $m$, there exist at least
\begin{eqnarray*}
	\left(1-\frac{\log m}{m}\right)T_{2}\geq\frac{1}{2}\left(1-\frac{\log m}{m}\right)\exp\left(\frac{1}{100}n\right)
	\ge \exp\left(\frac{n}{200}\right)
\end{eqnarray*}
vectors in $\mathcal{T}_{2}$ satisfying $\prod\limits_{j=t_0+1}^m\xi_i^l=1$. 
We finally choose $\mathcal{T}$ to be the set consisting of all these vectors.

The proof of \textbf{Part}~$(\mathrm{b})$ parallels that of \textbf{Part}~$(\mathrm{a})$, with a few differences.
First, Fact~\eqref{fact6} must be replaced by the following Fact~\eqref{fact7}, since a different choice of 
$t_0$ is required in our proof.
\begin{fact}\label{fact7}
For any given $\x$, with probability at least $1-\exp\lk- \Omega\Big( \frac{m} {\log^4 m }  \Big) \rk$,
\begin{eqnarray*}
\sum_{k=1}^{m}\mathbbm{1}_{\left\{ \lv\pp_{k}^{\top}\x\rv\le\frac{\lV \x\rV_2 }{40\log^2 m}\right\}  }> \frac{m}{25\log^2 m} := t_0.
\end{eqnarray*}
\end{fact}
\noindent The proof of Fact~\eqref{fact7} is similar to that of Fact~\eqref{fact6}.
Second, because of our choice of $t_0$, to use the analysis of the first group in \textbf{Part}~$(\mathrm{a})$, we must impose the restriction
\begin{eqnarray*}
t_0\log m/n=\frac{m}{40n\log m}\le\widetilde{L}
\end{eqnarray*}
for some $\widetilde{L}>0.$
The remaining analysis is identical to that in \textbf{Part}~$(\mathrm{a})$.

The proof of \textbf{Part}~$(\mathrm{c})$ parallels the analysis of the second group in \textbf{Part}~$(\mathrm{a})$, and does not rely on Fact~\eqref{fact5} or Fact~\eqref{fact6}. 
We therefore omit the details.

\section{Proofs for Sparse Phase Retrieval}

Following the framework in Section~\ref{Architecture} for analyzing the NCVX-LS estimator \eqref{model1}, we define the admissible set as
\begin{eqnarray*}
		\mathcal{E}^{s}_{\text{ncvx}}:=\left\{ \z\z^*-\x\x^*:\z,\x\in\Sigma_{s}^{n}\right\}.
	\end{eqnarray*}
It remains to verify that, with high probability, both the \textit{SLBC} and \textit{NUBC} with respect to $\lV\,\cdot\,\rV_{F}$ hold uniformly over this set, providing lower and upper bounds for parameters $\alpha$ and $\beta$, respectively.

\subsection{Upper Bounds for \textit{NUBC}}\label{NUBC for spr}

We provide upper bounds for the \textit{NUBC} with respect to $\lV\,\cdot\,\rV_F$, as stated in the following lemma.

\begin{lemma}\label{mul inq for spr}
Suppose that $\{\pp_{k}\}_{k=1}^{m}$ and $\{\xi_{k}\}_{k=1}^{m}$ satisfy conditions in Theorem~\ref{mul inq}.
				\begin{itemize}
		\item[$(\mathrm{a})$] If $\xi$ is sub-exponential, then there exist positive constants $c_1,C_1,L$ dependent only on $K$ such that if $m\ge L s\log\lk en/s\rk$, with probability at least $1-2\exp\lk-c_1s\log\lk en/s\rk\rk$, for all $\M\in\mathcal{E}^{s}_{\text{ncvx}}$,
     \begin{eqnarray*}
\lv\left\langle\sum_{k=1}^{m}\lk\xi_k\pp_k\pp_{k}^{*}-\E\xi\pp\pp^{*}\rk,\M\right\rangle\rv
		\le C_1 \lV\xi\rV_{\psi_1}\sqrt{ms\log\lk en/s\rk}\lV\M\rV_F;
				\end{eqnarray*}

		\item[$(\mathrm{b})$] If $\xi\in L_q$ for some $q>2$, 
		then there exist positive constants $c_2,c_3,C_2,\widetilde{L}$ dependent only on $K$ and $q$ such that if $m\ge \widetilde{L} s\log\lk en/s\rk$,
        with probability at least $1-c_2m^{-\lk q/2-1\rk}\log^{q} m-2\exp\lk-c_3s\log\lk en/s\rk\rk$, for all $\M\in\mathcal{E}^{s}_{\text{ncvx}}$,
     \begin{eqnarray*}
	\lv\left\langle\sum_{k=1}^{m}\lk\xi_k\pp_k\pp_{k}^{*}-\E\xi\pp\pp^{*}\rk,\M\right\rangle\rv
	\le C_2 \lV\xi\rV_{L_q}\sqrt{ms\log\lk en/s\rk}\lV\M\rV_F.
				\end{eqnarray*}
		\end{itemize}
\end{lemma}

\begin{proof}
Similar to the proof of Theorem~\ref{mul inq}, we use the multiplier processes in Lemma~\ref{multiplier}.
The only distinctions lies in the parameter $\widetilde{\Lambda}_{s_0,u}\lk\mF\rk$, where 
\begin{eqnarray*}
\mF:=\left\{\left\langle\frac{1}{\sqrt{m}}\sum_{k=1}^{m}\lk\pp_k\pp_k^*-\E\pp\pp^*\rk,\M\right\rangle:\M\in \mathcal{E}^{s}_{\text{ncvx}}\cap\mathbb{S}_{F}\right\}.
\end{eqnarray*}
To upper bound $\widetilde{\Lambda}_{s_0,u}\lk\mF\rk$, by Lemma~\ref{geometry} and following the proof of Theorem~\ref{mul inq}, it suffices to evaluate the $\gamma_{2}$-functional and $\gamma_{1}$-functional with respect to the set $\mathcal{E}^{s}_{\text{ncvx}}\cap\mathbb{S}_{F}$.

Since all elements of $\mathcal{E}^{s}_{\text{ncvx}}$ have rank at most 2, Lemma 3.1 in \cite{candes2011tight} implies the following bound on the covering number of $\mathcal{E}^{s}_{\text{ncvx}}\cap\mathbb{S}_{F}$:
 \begin{eqnarray*}
 \begin{aligned}
 \mN\lk \mathcal{E}^{s}_{\text{ncvx}}\cap\mathbb{S}_{F},\lV\,\cdot\,\rV_{F},\epsilon\rk
 \le\sum_{k=1}^{s} \binom{n}{k}\cdot \lk\frac{9}{\epsilon}\rk^{2\lk2s+1\rk}\le\lk\frac{en}{s}\rk^{s}\cdot\lk\frac{9}{\epsilon}\rk^{6s}.
 \end{aligned}
 \end{eqnarray*} 
Therefore, by Dudley integral (\cite[Theorem 11.17]{ledoux2013probability}), we obtain 
 \begin{eqnarray*}
\begin{aligned}
		\gamma_{2}\lk\mathcal{E}^{s}_{\text{ncvx}}\cap\mathbb{S}_{F},\lV\,\cdot\,\rV_{F}\rk
		&\le C \sqrt{6s}\lk\sqrt{\log \lk \frac{en}{s}\rk}+\int_{0}^{1} \sqrt{\log \lk \frac{9}{\epsilon}\rk}d\epsilon\rk\\
        &\le\widetilde{C}\sqrt{s\log\lk \frac{en}{s}\rk}.
	\end{aligned}
\end{eqnarray*} 
Similarly, we further bound $\gamma_{1}\lk\mathcal{E}^{s}_{\text{ncvx}}\cap\mathbb{S}_{F},\lV\,\cdot\,\rV_{op}\rk\lesssim s\log\lk en/s\rk$.
By ensuring that $m \gtrsim_{K} s\log\lk en/s\rk$, the proof is complete.
\end{proof}

\subsection{Lower Bounds for \textit{SLBC}}

We provide lower bounds for the \textit{SLBC} with respect to $\lV\,\cdot\,\rV_F$, as stated in the following lemma.

	\begin{lemma}\label{small ball for spr}
		Suppose that the sampling vectors  $\left\{\pp_{k}\right\}_{k=1}^{m}$  satisfy Assumption \ref{sample}.
		Then there exist positive constants $L,c,C$, depending only on $K$ and $\mu$ such that if $m\ge L s\log\lk en/s\rk$, 
         with probability at least $1- e^{-cm}$, for all $\M \in \mathcal{E}^{s}_{\text{ncvx}}$: 
		\begin{eqnarray*}
		 \sum_{k=1}^{m}\lv\lg\pp_k\pp_k^*,\M\rg\rv^2\ge Cm \lV\M\rV^{2}_{F}.
		\end{eqnarray*}
	\end{lemma}	

\begin{proof}
The proof follows the same strategy as in Lemma~\ref{small ball1}, employing the small ball method.
Using the upper bounds on the $\gamma_{2}\lk\mathcal{E}^{s}_{\text{ncvx}}\cap\mathbb{S}_{F},\lV\,\cdot\,\rV_{F}\rk$ and $\gamma_{1}\lk\mathcal{E}^{s}_{\text{ncvx}}\cap\mathbb{S}_{F},\lV\,\cdot\,\rV_{F}\rk$ established in the proof of Lemma~\ref{mul inq for spr}, together with Lemma~\ref{chaos1}, we obtain
\begin{align*}
\mathcal{W}_{m}\lk \mathcal{E}^{s}_{\text{ncvx}}\cap\mathbb{S}_{F};\pp\pp^* \rk \notag
&\le C K^2\sqrt{m}\lk \sqrt{\frac{s\log\lk en/s\rk}{m}} + \frac{s\log\lk en/s\rk}{m} \rk.
\end{align*}
We choose $u=\frac{1}{2}\sqrt{\frac{\min\left\{\mu,1\right\}}{2}}$, 
by proposition~\ref{small ball function} we have
\begin{eqnarray*}
\mathcal{Q}_{2u}\lk \mathcal{E}^{s}_{\text{ncvx}}\cap\mathbb{S}_{F};\pp\pp^*\rk	\gtrsim \frac{\min\left\{ \mu^2,\, 1 \right\}}{K^8 + 1}.
\end{eqnarray*} 
This completes the proof by Proposition~\ref{small ball}, provided that $m \gtrsim_{K,\mu} s\log \lk en/s\rk$.
 \end{proof}

\subsection{Proofs of Theorem~\ref{sparse pr_poisson} and Theorem~\ref{sparse pr_heavy}}

We follow the argument presented in Section~\ref{proof of main}.
We first prove \textbf{Part}~$(\mathrm{a})$ of Theorem~\ref{sparse pr_poisson}.
By \textbf{Part}~$(\mathrm{a})$ of Lemma~\ref{mul inq for spr} and  Proposition~\ref{psi_1 of poisson}, we have 
\begin{eqnarray*}
\beta\lesssim_K \max\left\{1,K\lV\x\rV_2\right\}\cdot\sqrt{ms\log\lk en/s\rk}.
\end{eqnarray*}
Moreover, Lemma~\ref{small ball for spr} yields $\alpha \gtrsim_{K,\mu} m$.
Hence, \textbf{Part}~$(\mathrm{a})$ of Theorem~\ref{sparse pr_poisson} is established by \eqref{error1} in Section~\ref{Architecture}.
Similarly, by \textbf{Part}~$(\mathrm{b})$ of Lemma~\ref{mul inq for spr} along with Proposition~\ref{L_q of poisson} and  the condition $\x\in\Gamma_s$, we obtain
\begin{eqnarray*}
\beta\lesssim_K \sqrt{K\lV\x\rV_2}\cdot\sqrt{ms\log\lk en/s\rk}.
\end{eqnarray*}
Combining with the lower bound  $\alpha \gtrsim_{K,\mu} m$, we can establish \textbf{Part}~$(\mathrm{b})$ of Theorem~\ref{sparse pr_poisson}.

To prove Theorem~\ref{sparse pr_heavy}, we invoke \textbf{Part}~$(\mathrm{b})$ of Lemma~\ref{mul inq for spr}, which yields
\begin{eqnarray*}
\beta\lesssim_{K,q} \lV\xi\rV_{L_q}\cdot\sqrt{ms\log\lk en/s\rk}.
\end{eqnarray*}
Combined with $\alpha \gtrsim_{K,\mu} m$, the proof is complete.

\section{Proofs for Low-Rank PSD Matrix Recovery}

We follow the framework outlined in Section~\ref{Architecture} for analyzing the CVX-LS estimator \eqref{model2}. 
In the setting of recovering low-rank PSD matrix, we define the admissible set as
\begin{eqnarray*}\label{E_psd}
		\mathcal{E}^{r}_{\text{cvx}}:=\left\{\Z-\X:\Z,\X\in\mS_{+}^{n}\  \text{and}\ \X\ \text{is rank-}r \right\}.
	\end{eqnarray*}
We begin with the following proposition, which asserts that any matrix in $\mathcal{E}^{r}_{\text{cvx}}$ has at most $r$ negative eigenvalues.	
\begin{proposition}\label{r}
Suppose that $\pmb{M}\in \mathcal{E}^{r}_{\text{cvx}}$. 
Then $\pmb{M}$ has at most $r$ strictly negative eigenvalue.
\end{proposition}

\begin{proof}

By the definition of $\mathcal{E}^{r}_{\text{cvx}}$, for any $\pmb{M}\in \mathcal{E}^{r}_{\text{cvx}}$, we can find a rank-$r$ matrix $\X \in \mS^{n}_{+}$ such that $\X + \M \in \mathcal{S}^n_{+}.$
If $\M$ has $r+1$ (strictly) negative eigenvalues with corresponding eigenvectors $ \pmb{z}_1,\cdots, \pmb{z}_{r+1} \in \mathbb{C}^n$, one could choose a nonzero vector $\xu$ in their span orthogonal to $\X$, i.e.,  $ \langle  \xu\xu^*,\X \rangle =0 $, yielding $\xu \lk  \X + \M \rk \xu^* =  \xu^*\M\xu < 0$, contradicting the PSD condition.

\end{proof}

Unlike the two-part partition used for $\mathcal{E}_{\text{cvx}}$ in Section~\ref{Architecture}, a more refined partitioning strategy is required to handle $\mathcal{E}^{r}_{\text{cvx}}$.
We restate that for a matrix $\M\in \mathcal{S}^n$, we denote its eigenvalues by $\left\{ \lambda_i\lk \M\rk \right\}_{i=1}^n$, arranged in decreasing order.
By Proposition \ref{r}, the eigenvalues of $\M$ satisfies that  $\lambda_i\lk \M\rk \ge 0$ for all $i \in \left[n - r\right]$.
We first divide $\mathcal{E}^{r}_{\text{cvx}}$ into $r+1$ disjoint parts: 
\begin{equation*}
\mathcal{E}^{r;k}_{\text{cvx}} := \left\{
\M \in \mathcal{E}^r_{\text{cvx}} :
\begin{array}{l}
\text{for } i \in [n-k],\quad \lambda_i(\M) > 0 \\[3pt]
\text{for } i \in [n]\setminus[n-k],\quad \lambda_i(\M) \le 0
\end{array}
\right\},\quad k=0,1,\cdots,r.
\end{equation*}
We can see that $\mathcal{E}^{r;0}_{\text{cvx}}$ is the positive definite cone in $\mathcal{S}^n$.
For each $\mathcal{E}^{r;k}_{\text{cvx}}$, we divide it into two parts:
an approximately low-rank subset
\begin{eqnarray*}
		\mathcal{E}^{r;k}_{\text{cvx,1}}:= \left\{\M\in  \mathcal{E}^{r;k}_{\text{cvx}}:   -\sum_{i=n-k+1}^{n} \lambda_{i} \lk\M\rk > \frac{1}{2}  \sum_{i=1}^{n-k} \lambda_i \lk\M\rk   \right\},
			\end{eqnarray*}
            and an almost PSD subset
\begin{eqnarray*}
		\mathcal{E}^{r;k}_{\text{cvx,2}}:= \left\{\M\in  \mathcal{E}^{r;k}_{\text{cvx}}: -\sum_{i=n-k+1}^{n} \lambda_{i} \lk\M\rk \le \frac{1}{2}  \sum_{i=1}^{n-k} \lambda_i \lk\M\rk  \right\}.
	\end{eqnarray*}
Now, we let
\begin{eqnarray*}
		\mathcal{E}^{r}_{\text{cvx,1}}:= \bigcup_{k=0}^{r}	\mathcal{E}^{r;k}_{\text{cvx,1}}		
		\quad \text{and}\quad \mathcal{E}^{r}_{\text{cvx,2}}:= \bigcup_{k=0}^{r}	\mathcal{E}^{r;k}_{\text{cvx,2}}.	
		\end{eqnarray*}

The following proposition states that the elements in $\mathcal{E}^{r}_{\text{cvx,1}}$ are approximately low-rank.

\begin{proposition}\label{pro low-rank_0}
 For all $\M\in\mathcal{E}^{r}_{\text{cvx,1}}$, we have $ \lV\M\rV_{*}\le 3\sqrt{r}\lV\M\rV_{F}$.
\end{proposition}

\begin{proof}
For every $k=0,1,\cdots,r$, the element $\M\in\mathcal{E}^{r;k}_{\text{cvx,1}}$ satisfies that
\begin{eqnarray*}
 \frac{1}{2}  \sum_{i=1}^{n-k} \lambda_i \lk\M\rk< -\sum_{i=n-k+1}^{n} \lambda_{i} \lk\M\rk.
\end{eqnarray*}
Thus we have that
       \begin{eqnarray*}
       \begin{aligned}
\lV\M\rV_{*}&=\sum_{i=1}^{n-k} \lambda_i \lk\M\rk -\sum_{i=n-k+1}^{n} \lambda_{i} \lk\M\rk\\
&\le -3\sum_{i=n-k+1}^{n} \lambda_{i} \lk\M\rk
\le 3\sqrt{k}\lV\M\rV_{F}\le3\sqrt{r}\lV\M\rV_{F}.      
                  \end{aligned}
\end{eqnarray*}
\end{proof}

\subsection{Upper Bounds for \textit{NUBC}}

We provide uppers bounds for the \textit{NUBC}, as stated in the following lemma.

\begin{lemma}\label{mul inq for psd}
Suppose that $\left\{\pp_{k}\right\}_{k=1}^{m}$ and $\left\{\xi_{k}\right\}_{k=1}^{m}$ satisfy the conditions in Theorem~\ref{mul inq}.
\begin{itemize}
\item If $\xi$ is sub-exponential, then there exist positive constants $c,C_1,C_2,L$ dependent only on $K$ such that, when provided $m\ge L n$, with probability at least $1-2\exp\lk-cn\rk$, the following holds:
				\begin{itemize}
		\item[$\mathrm{(a)}$] 
		For  all $\M\in \mathcal{E}^r_{\text{cvx,1}} $, one has
 \begin{eqnarray*}
 \lv\left\langle\sum_{k=1}^{m}\lk\xi_k\pp_k\pp_{k}^{*}-\E\xi\pp\pp^{*}\rk,\M\right\rangle\rv
	\le C_1 \lV\xi\rV_{\psi_1}\sqrt{mrn} \lV\M\rV_{F};
    \end{eqnarray*}
	
		\item[$\mathrm{(b)}$] 
		  For all $\M\in \mathcal{E}^r_{\text{cvx,2}}$, one has
        \begin{eqnarray*}
	\lv\left\langle\sum_{k=1}^{m}\lk\xi_k\pp_k\pp_{k}^{*}-\E\xi\pp\pp^{*}\rk,\M\right\rangle\rv
	\le C_2 \lV\xi\rV_{\psi_1}\sqrt{mn}\lV\M\rV_{*}.
    \end{eqnarray*}

		\end{itemize}
        
       \item If $\xi\in L_q$ for some $q>2$, then there exist positive constants $c_1,c_2,C_3,C_4,\widetilde{L}$ dependent only on $K$ and $q$ such that, when provided $m\ge \widetilde{L} n$, with probability at least $1-c_1m^{-\lk q/2-1\rk}\log^{q} m-2\exp\lk-c_2n\rk$,
the following holds:
				\begin{itemize}
		\item[$\mathrm{(c)}$] 
		For all $\M\in \mathcal{E}^r_{\text{cvx,1}} $, one has
  \begin{eqnarray*}
  \lv\left\langle\sum_{k=1}^{m}\lk\xi_k\pp_k\pp_{k}^{*}-\E\xi\pp\pp^{*}\rk,\M\right\rangle\rv
	\le C_3 \lV\xi\rV_{L_q}\sqrt{mrn} \lV\M\rV_{F};
    \end{eqnarray*}
	
		\item[$\mathrm{(d)}$] 
		  For all $\M\in \mathcal{E}^r_{\text{cvx,2}} $, one has
        \begin{eqnarray*}
	\lv\left\langle\sum_{k=1}^{m}\lk\xi_k\pp_k\pp_{k}^{*}-\E\xi\pp\pp^{*}\rk,\M\right\rangle\rv
	\le C_4 \lV\xi\rV_{L_q}\sqrt{mn}\lV\M\rV_{*}.
    \end{eqnarray*}
    
        \end{itemize}
        \end{itemize}
\end{lemma}

\begin{proof}
The proof of \textbf{Part}~$\mathrm{(a)}$ follows from Theorem~\ref{mul inq} and Proposition~\ref{pro low-rank_0}, since we have that
\begin{eqnarray*}
\begin{aligned}
 \lv\left\langle\sum_{k=1}^{m}\lk\xi_k\pp_k\pp_{k}^{*}-\E\xi\pp\pp^{*}\rk,\M\right\rangle\rv
	&\le \lV\sum_{k=1}^{m}\lk\xi_k\pp_k\pp_{k}^{*}-\E\xi\pp\pp^{*}\rk\rV_{op}\cdot\lV\M\rV_F\\
    &\le C_1 \lV\xi\rV_{\psi_1}\sqrt{mrn} \lV\M\rV_{F}.
    \end{aligned}
    \end{eqnarray*}
The proof of \textbf{Part}~$\mathrm{(c)}$ is similar.
The proofs of \textbf{Part}~$\mathrm{(b)}$ and \textbf{Part}~$\mathrm{(d)}$ follow directly from Theorem~\ref{mul inq}.
\end{proof}

\subsection{Lower Bounds for \textit{SLBC}}

We establish lower bounds for the \textit{SLBC} to bound the parameters~$\alpha$ and~$\widetilde{\alpha}$ from below.
We first derive the \textit{SLBC} with respect to $\lV\,\cdot\,\rV_{F}$ over the admissible set $\mathcal{E}^{r}_{\text{cvx,1}}$.
The result is stated in the following lemma.

	\begin{lemma}\label{small ball for lr}
		Suppose that the $\left\{\pp_{k}\right\}_{k=1}^{m}$  satisfy Assumption \ref{sample}.
		Then there exist positive constants $L,c,C$ dependent only on $K$ and $\mu$ such that if $m\ge Lrn$, with probability at least $1- \mathcal{O}\lk e^{-cm}\rk$, the following holds  for all $ \M \in \mathcal{E}^r_{\text{cvx}}$,
		\begin{eqnarray*}
		 \sum_{k=1}^{m}\lv\lg\pp_k\pp_k^*,\M\rg\rv^2\ge Cm \lV\M\rV^{2}_{F}.
		\end{eqnarray*}
	\end{lemma}		

\begin{proof}

The proof is similar to that of Lemma~\ref{small ball1}, except that here it remains to establish
 \begin{eqnarray*}
 \mathcal{W}_{m}\lk \mathcal{E}_{\text{cvx,1}}\cap \mathbb{S}_{F};\pp\pp^*\rk\lesssim_K\sqrt{rm}\lk \sqrt{\frac{n}{m}}+\frac{n}{m}\rk.
\end{eqnarray*}
In fact, we have that
\begin{align*}
\mathcal{W}_{m}\lk \mathcal{E}_{\text{cvx,1}} \cap \mathbb{S}_{F}; \pp\pp^* \rk
&\le \mathbb{E} \lV \frac{1}{\sqrt{m}} \sum_{k=1}^{m} \varepsilon_{k} \pp_k \pp_k^* \rV_{op} \cdot \lV \M \rV_{*} \notag \\
&\le 3\sqrt{r} \cdot \mathcal{W}_{m}\lk\mM; \pp\pp^* \rk \notag \\
&\lesssim K^{2} \sqrt{rm} \lk\sqrt{\frac{n}{m}} + \frac{n}{m} \rk, \label{eq:W_bound_cvx1}
\end{align*}
where $\mM=\left\{\z\z^*:\z\in\mathbb{S}^{n-1}\right\}$.
Here, in the second inequality we have used Proposition \ref{pro low-rank_0}, and in the third inequality we have used \eqref{up for W} in Section~\ref{small_ball}.
\end{proof}

We then derive the \textit{SLBC} with respect to $\lV\,\cdot\,\rV_{op}$ over the admissible set $\mathcal{E}^{r}_{\text{cvx,2}}$. 

	\begin{lemma}\label{approximate for lr}
		Suppose that $\{\pp_{k}\}_{k=1}^{m}$ are independent copies of a random vectors $\pp$ whose entries $\left\{\varphi_j\right\}_{j=1}^{n}$ are i.i.d., mean 0, variance 1, and $K$-sub-Gaussian.
		Then there exist positive constants $L,c$ dependent only on $K$ such that if  $m\ge L n$, with probability at $1- 2e^{-cm}$, the following holds for all $ \M \in  \mathcal{E}^r_{\text{cvx,2}}$,
        \begin{eqnarray*}
		 \sum_{k=1}^{m}\lv\lg\pp_k\pp_k^*,\M\rg\rv^2\ge \frac{1}{36} m \lV\M\rV^{2}_{*}.
		\end{eqnarray*}
	\end{lemma}
	
	\begin{proof}
    The proof is similar to that of Lemma~\ref{approximate}.
Set $\M\in \mathcal{E}^r_{\text{cvx,2}}$, by Proposition \ref{r}, we know that $\M$ has at most $r$ negative eigenvalue. 
If $\M\in \mathcal{E}^{r;0}_{\text{cvx,2}}\subset\mathcal{E}^r_{\text{cvx,2}}$, then setting $\delta = \frac{1}{6}$ in \eqref{l_2} yields
$\sum\limits_{k=1}^{m}\lv\lg\pp_k\pp_k^*,\M\rg\rv \ge \frac{5}{6} m\lV\M\rV_{*}.$
If $\M\in \mathcal{E}^{r;k}_{\text{cvx,2}}$ where $k\in[r]$, 
since we have $-\sum\limits_{i=n-k+1}^{n} \lambda_{i} \lk\M\rk \le \frac{1}{2}  \sum\limits_{i=1}^{n-k} \lambda_i \lk\M\rk$, we obtain that
\begin{equation*}\label{eq22}
\begin{aligned}
\sum_{k=1}^{m}\lv\lg\pp_k\pp_k^*,\M\rg\rv
&\ge  \sum_{k=1}^{m}\lg\pp_k\pp_k^*,\M\rg
=\sum_{i=1}^{n} \lambda_i \lk\M\rk \lk \sum_{k=1}^{m}\lv\lg\pp_k,\xu_i\rg\rv^2\rk\\
 &\ge  \frac{5}{6} m \sum_{i=1}^{n-k} \lambda_i \lk\M\rk  +\frac{7}{6}  m\sum_{i=n-k+1}^{n}\lambda_k \lk\M\rk\\
 &\ge  \frac{1}{4} m \sum_{i=1}^{n-k} \lambda_i \lk\M\rk\ge\frac{1}{6} m \lV\M\rV_{*}.
\end{aligned}
\end{equation*}
In the last inequality, we have used 
\begin{equation*}
\lV\M\rV_{*}= \sum_{i=1}^{n-k} \lambda_i \lk\M\rk -\sum_{i=n-k+1}^{n}\lambda_n \lk\M\rk \le  \frac{3}{2}\sum_{i=1}^{n-k} \lambda_i \lk\M\rk.
\end{equation*}
The proof then follows from the Cauchy–Schwarz inequality.
\end{proof}

\subsection{Proof of Theorem~\ref{lr_poisson}}	

The proof relies on the following proposition to characterize the properties of Poisson noise.
\begin{proposition}\label{eq:psi_1 of poisson_lr}
Let random variable 
\begin{eqnarray*}
\xi=\text{Poisson}\lk \lg\pp\pp^*,\X\rg\rk-\lg\pp\pp^*,\X\rg,
\end{eqnarray*}
where $X\in\mathcal{S}^{+}_n$ and the entries $\left\{\varphi_j\right\}_{j=1}^{n}$ of random vector $\pp$ are independent, mean-zero and $K$-sub-Gaussian.
Then we have
\begin{itemize}
		\item[$\mathrm{(a)}$] 
$ \lV\xi\rV_{\psi_1}\lesssim \max\left\{1, K\sqrt{\lV\X\rV_*} \right\};$
 		\item[$\mathrm{(b)}$] 
$\lV\xi\rV_{L_{4}}\lesssim\max\left\{ \sqrt{K}\lV\X\rV_*^{1/4},K\sqrt{\lV\X\rV_*}\right\}.$
 \end{itemize}
\end{proposition}

\begin{proof}

We claim that there exists a constant $C_0\ge 1$ such that
\begin{eqnarray}\label{eq:psi_2 for PSD}
\lV  \sqrt{\lg\pp\pp^*,\X\rg}\rV_{\psi_2}\le C_0 K\sqrt{\lV\X\rV_*}.
\end{eqnarray}
    Since $\lV  \xi\rV^2_{\psi_2}=\lV \xi^2\rV_{\psi_1}$, we can obtain that
    \begin{align*}
\lV  \sqrt{\lg\pp\pp^*, \X \rg} \rV^2_{\psi_2}
&= \lV\lg \pp\pp^*, \X \rg \rV_{\psi_1} \nonumber 
\le \sum_{k=1}^m \lambda_k \lk \X\rk \lV \lg \pp\pp^*, \xu_k \xu_k^* \rg\rV_{\psi_1} \nonumber \\
&= \sum_{k=1}^m \lambda_k\lk \X\rk \lV\pp^* \xu_k \rV^2_{\psi_2} \nonumber 
\le C K^2 \sum_{k=1}^m \lambda_k\lk \X\rk=C K^2\lV\X\rV_{*}.
\end{align*}
The first inequality follows from the orthogonal decomposition of the PSD matrix 
$\X$.
The second inequality follows from Fact~\ref{fact4}.

The remaining proofs follow directly from Proposition~\ref{psi_1 of poisson} and Proposition~\ref{L_q of poisson}, provided that Fact~\ref{fact4} used in their proofs is adapted to the setting of \eqref{eq:psi_2 for PSD}.

\end{proof}

We now prove \textbf{Part}~$(\mathrm{a})$ of Theorem~\ref{lr_poisson}.
By Lemma~\ref{small ball for lr} we have $\alpha\gtrsim_{K,\mu}m$, and by Lemma~\ref{approximate for lr} it holds that
$\widetilde{\alpha}\ge \frac{1}{36}m$.
Moreover, by combining \textbf{Part}~$\mathrm{(a)}$ and \textbf{Part}~$\mathrm{(b)}$ of Lemma~\ref{mul inq for psd} with \textbf{Part}~$\mathrm{(a)}$ of Proposition~\ref{eq:psi_1 of poisson_lr}, we obtain
    \begin{equation*}
\beta\lesssim_K\max\left\{1, K\sqrt{\lV\X\rV_*} \right\}\cdot\sqrt{mrn}\quad \text{and}\quad\widetilde{\beta}\lesssim_K \max\left\{1, K\sqrt{\lV\X\rV_*} \right\}\cdot\sqrt{mn}.
\end{equation*}
Therefore, the estimation error can be bounded as
\begin{equation*}
\lV\Z_{\star}-\X\rV_F\le2\max\left\{\frac{\beta}{\alpha},\frac{\widetilde{\beta}}{\widetilde{\alpha}}\right\}\lesssim_{K,\mu}\max\left\{1, K\sqrt{\lV\X\rV_*} \right\}\cdot\sqrt{\frac{rn}{m}}.
\end{equation*}

Similarly, for \textbf{Part}~$\mathrm{(b)}$ of Theorem~\ref{lr_poisson}, by combining \textbf{Part}~$\mathrm{(c)}$ and \textbf{Part}~$\mathrm{(d)}$ of Lemma~\ref{mul inq for psd} with \textbf{Part}~$\mathrm{(b)}$ of Proposition~\ref{eq:psi_1 of poisson_lr}, we have
\begin{equation*}
\beta\lesssim_K \sqrt{K}\lV\X\rV_*^{1/4}\cdot\sqrt{mrn}\quad \text{and}\quad\widetilde{\beta}\lesssim_K \sqrt{K}\lV\X\rV_*^{1/4}\cdot\sqrt{mn}.
\end{equation*}
Therefore, the error bound becomes
\begin{equation*}
\lV\Z_{\star}-\X\rV_F\lesssim_{K,\mu}\sqrt{K}\lV\X\rV_*^{1/4}\cdot\sqrt{\frac{rn}{m}}.
\end{equation*}

\subsection{Proof of Theorem~\ref{lr_heavy}}	
	The proof is similar to the proof of Theorem~\ref{lr_poisson}.
    We also have that $\alpha\gtrsim_{K,\mu}m$ and $\widetilde{\alpha}\ge \frac{1}{36}m$.
    By \textbf{Part}~$(\mathrm{c})$ and \textbf{Part}~$(\mathrm{d})$ of Lemma~\ref{mul inq for psd}, it holds that
    \begin{equation*}
\beta\lesssim_{K,q}\lV\xi\rV_{L_q}\cdot\sqrt{mrn}\quad \text{and}\quad\widetilde{\beta}\lesssim_{K,q}\lV\xi\rV_{L_q}\cdot\sqrt{mn}.
\end{equation*}
Therefore, we obtain
\begin{equation*}
\lV\Z_{\star}-\X\rV_F\lesssim_{K,\mu,q}\lV\xi\rV_{L_q}\cdot\sqrt{\frac{rn}{m}}.
\end{equation*}

\section{Proofs for Random Blind Deconvolution}

To use the framework outline in Section \ref{Architecture}, we first define the admissible set for this setting.
The descent cone of the nuclear norm at a point $\X\in\C^{n\times n}$ is the set of all possible directions $\M\in\C^{n\times n}$ along which the nuclear norm does not increase; see e.g., \cite{chandrasekaran2012convex}. 
Specifically, for a rank-one matrix $\x\h^*$, the descent cone is given by
     \begin{eqnarray*}	     
     \mD\lk \x\h^*\rk:=\left\{\M\in\C^{n\times n}:\lV\x\h^*_0+t\M\rV_{*}\le\lV\x\h^*\rV_{*}\,\text{for some}\ t>0\right\}.
     	         \end{eqnarray*}
To ensure that our results hold uniformly for all $\x,\h\in\C^n$, we define the admissible set as the union of descent cones over all nonzero pairs:
 \begin{eqnarray*}	     
    \widetilde{\mathcal{E}}:=\bigcup_{\x,\h}\mD\lk \x\h^*\rk,
	         \end{eqnarray*}
	         where the union runs over all $\x,\h\in\C^n \backslash \left\{0\right\}$.
             In what follows, we take $\widetilde{\mathcal{E}}$ as the admissible set for our analysis.
             
       The following proposition characterizes the geometric properties of the admissible set $\widetilde{\mathcal{E}}$, which will be used in the subsequent analysis.
       Its proof can be obtained either directly from Lemma 10 in \cite{kueng2017low} or from Proposition 1 in \cite{huang2025low}; we omit the details here.
\begin{proposition}[\cite{kueng2017low,huang2025low}]\label{des}
For all $\M\in\widetilde{\mathcal{E}}$, one has
\begin{eqnarray*}
\lV\M\rV_{*}\le2\sqrt{2} \lV\M\rV_{F}.
	    \end{eqnarray*}  
\end{proposition}

\subsection{Proof of Theorem \ref{thm:bd}}

We first provide  upper bounds for the \textit{NUBC} with respect to $\lV\,\cdot\,\rV_F$.

\begin{lemma}\label{le:mul inq for bd}
Suppose that $\{\pmb{a}_{k}\}_{k=1}^{m}$ and $\{\pmb{b}_{k}\}_{k=1}^{m}$ satisfy conditions in Theorem \ref{thm:bd},
and the noise term $\left\{\xi_k\right\}_{k=1}^{m}$ satisfies the conditions in Assumption \ref{noise0} $\mathrm{(b)}$ with $q>2$.
Then there exist positive constants $c_1,c_2,C,L$ dependent only on $K$ and $q$ such that if $m\ge L n$,
with probability at least $1-c_1m^{-\lk q/2-1\rk}\log^{q} m-2\exp\lk-c_2n\rk$, for all $\M\in\widetilde{\mathcal{E}}$, 
     \begin{eqnarray*}
	\lv\left\langle\sum_{k=1}^{m}\xi_k\pmb{a}_k\pmb{b}_{k}^{*},\M\right\rangle\rv
	\le C \lV\xi\rV_{L_q}\sqrt{mn}\lV\M\rV_F.
				\end{eqnarray*}
\end{lemma}

\begin{proof}
By \textbf{Part}~$(\mathrm{b})$ of Theorem~\ref{mul inq} (see Remark~\ref{remark1}) combined with Proposition~\ref{des}, we can obtain that
\begin{eqnarray*}
\begin{aligned}
	\lv\left\langle\sum_{k=1}^{m}\xi_k\pmb{a}_k\pmb{b}_{k}^{*},\M\right\rangle\rv
	&\le\lV\sum_{k=1}^{m}\xi_k\pmb{a}_k\pmb{b}_{k}^{*}\rV_{op}\cdot\lV\M\rV_{F}\\
    &\le C \lV\xi\rV_{L_q}\sqrt{mn}\lV\M\rV_F.
    \end{aligned}
				\end{eqnarray*}

\end{proof}
We then provide lower bounds for the \textit{SLBC} with respect to $\lV\,\cdot\,\rV_F$.

	\begin{lemma}\label{le:small ball for bd}
		Suppose that $\{\pmb{a}_{k}\}_{k=1}^{m}$ and $\{\pmb{b}_{k}\}_{k=1}^{m}$ satisfy conditions in Theorem \ref{thm:bd}.
        Then there exist positive constants $L,c,C$ dependent only on $K$ such that if $m\ge L n$, 
        with probability at least $1- \mathcal{O}\lk e^{-cm}\rk$, for all $\M\in\widetilde{\mathcal{E}}$,
		\begin{eqnarray*}
		 \sum_{k=1}^{m}\lv\left\langle\pmb{a}_k\pmb{b}_{k}^{*},\M\right\rangle\rv^2\ge Cm \lV\M\rV^{2}_{F}.
		\end{eqnarray*}
	\end{lemma}	

\begin{proof}

In a manner analogous to Proposition~\ref{small ball function}, for $0<u\le\frac{\sqrt{2}}{4}$ we proof that 
\begin{eqnarray}\label{eq:small ball for ab}
\mathcal{Q}_{2u}\lk \widetilde{\mathcal{E}} \cap \mathbb{S}_F;\pmb{a}\pmb{b}^*\rk	\gtrsim  \frac{1}{K^8 }.
\end{eqnarray} 
Specially, by the Paley–Zygmund inequality (see e.g., \cite{de2012decoupling}), for any $\M\in\mS^n$,
\begin{equation*}
     \mathbb{P}\lk\lv\pmb{a}^*\M\pmb{b}\rv^2\ge\frac{\E\lv\pmb{a}^*\M\pmb{b}\rv^2}{2}\rk
     \ge\frac{\lk\E\lv\pmb{a}^*\M\pmb{b}\rv^2\rk^2}{\E\lv\pmb{a}^*\M\pmb{b}\rv^4}.
 \end{equation*}
 By direct calculation, we have
 \begin{equation*}\label{eq:E^2ab}
\begin{aligned}
  \E\lv\pmb{a}^*\M\pmb{b}\rv^2
&=\E\lk\sum_{i,j}\M_{i,j}\overline{a}_ib_j\rk\lk\sum_{\tilde{i},\tilde{j}}\overline{\M}_{\tilde{i},\tilde{j}}
\overline{b}_{\tilde{j}}\rk\\
&=\sum_{i,j,\tilde{i},\tilde{j}}\E\, \M_{i,j}\overline{\M}_{\tilde{i},\tilde{j}}\overline{a}_{i}a_{\tilde{i}}b_j\overline{b}_{\tilde{j}}
 \\
 &=\sum_{i=,\tilde{i},j=\tilde{j}}\M_{i,j}\overline{\M}_{\tilde{i},\tilde{j}}=\lV\M\rV^2_F.
 \end{aligned}
 \end{equation*}
By Lemma~\ref{geometry} (it still holds in this asymmetric setting), we obtain
 \begin{equation*}\label{eq:E^4ab}
 \begin{aligned}
\E\lv\pmb{a}^*\M\pmb{b}\rv^4&\lesssim\E\lv\pmb{a}^*\M\pmb{b}-\E\pmb{a}^*\M\pmb{b}\rv^4+\lk\E\,\pmb{a}^*\M\pmb{b}\rk^4\\
&\lesssim K^8\lV\M\rV^4_F,
\end{aligned}
 \end{equation*} 
where $\E\,\pmb{a}^*\M\pmb{b}=0$.
Hence, by the definition of the small ball function in Proposition~\ref{small ball}, we establish \eqref{eq:small ball for ab}.

Moreover, we can also upper bound the Rademacher empirical process as
\begin{eqnarray*}
\begin{aligned}
\mathcal{W}_{m}\lk \widetilde{\mathcal{E}} \cap \mathbb{S}_F;\pmb{a}\pmb{b}^* \rk
&\le\mathbb{E} \lV \frac{1}{\sqrt{m}} \sum_{k=1}^{m} \varepsilon_{k} \pmb{a}_k\pmb{b}_k^* \rV_{op} \cdot \lV \M \rV_{*} \notag \\
&\lesssim K^2\sqrt{m}\lk\sqrt{\frac{n}{m}} + \frac{n}{m} \rk.
\end{aligned}
\end{eqnarray*}
Here, in the second inequality we have used Proposition \ref{des} and
\begin{eqnarray*}
\mathbb{E} \lV \frac{1}{\sqrt{m}} \sum_{k=1}^{m} \varepsilon_{k} \pmb{a}_k\pmb{b}_k^* \rV_{op}\lesssim K^2\lk\sqrt{n}+\frac{n}{\sqrt{m}}\rk.
\end{eqnarray*}
The proof then follows by choosing $u=\frac{\sqrt{2}}{4}$ and $t=\frac{c\sqrt{m}}{K^8}$ in Proposition~\ref{small ball}, and assuming 
$m\ge L n$ for some constant $L>0$ depending only on $K$.
 \end{proof}

Now, we turn to the proof of Theorem~\ref{thm:bd}.
By Lemma~\ref{le:small ball for bd} and Lemma~\ref{le:mul inq for bd}, we have that
\begin{eqnarray*}
\alpha\gtrsim_{K}m\quad \text{and}\quad \beta\lesssim_{K,q} \lV\xi\rV_{L_q}\cdot\sqrt{mn}.
\end{eqnarray*}
Thus, we finally obtain
\begin{eqnarray*}
\lV\Z_{\star}-\x\h^*\rV_F\le\frac{2\beta}{\alpha}\lesssim_{K,q} \lV\xi\rV_{L_q}\cdot\sqrt{\frac{n}{m}}.
\end{eqnarray*}

\normalem
\bibliographystyle{plain}
\bibliography{ref}

\end{document}